\documentclass[reqno]{amsart}
\usepackage{amsmath}
\usepackage{amsfonts}
\usepackage{amssymb}
\usepackage{bm}
\usepackage{pdfsync}
\usepackage{color,graphicx,tikz-cd}
\usepackage{a4wide}
\usepackage{amscd,amsthm,latexsym}
\usepackage{hyperref}
\usepackage{kotex}
\usepackage{relsize}
\usepackage{tikz}
\usepackage[misc]{ifsym}
\numberwithin{equation}{section}
\usetikzlibrary{snakes}
\usepackage{cite}
\usetikzlibrary{decorations.pathreplacing,calligraphy}

\usepackage{cleveref}
\crefname{claim}{Claim}{Claims}
\Crefname{claim}{Claim}{Claims}
\crefname{lem}{Lemma}{Lemmas}
\Crefname{lem}{Lemma}{Lemmas}
\crefname{thm}{Theorem}{Theorems}
\Crefname{thm}{Theorem}{Theorems}
\crefname{prop}{Proposition}{Propositions}
\Crefname{prop}{Proposition}{Propositions}
\crefname{question}{Question}{Questions}
\Crefname{question}{Question}{Questions}
\crefname{defn}{Definition}{Definitions}
\Crefname{defn}{Definition}{Definitions}
\crefname{conj}{Conjecture}{Conjectures}
\Crefname{conj}{Conjecture}{Conjectures}
\crefname{exam}{Example}{Examples}
\Crefname{exam}{Example}{Examples}
\crefname{remark}{Remark}{Remarks}
\Crefname{remark}{Remark}{Remarks}
\crefname{figure}{Figure}{Figures}
\Crefname{figure}{Figure}{Figures}
\crefname{cor}{Corollary}{Corollaries}
\Crefname{cor}{Corollary}{Corollaries}
\crefformat{equation}{(#2#1#3)}
\Crefformat{equation}{Equation #2(#1)#3}

\renewcommand\emph[1]{\textcolor{frenchblue}{\it #1}}

\allowdisplaybreaks %

\title[Hall--Littlewood expansions of chromatic quasisymmetric functions]
{Hall--Littlewood expansions of chromatic quasisymmetric functions using
linked rook placements}

\author{Jang Soo Kim}
\address{Department of Mathematics, Sungkyunkwan University, Suwon, South Korea}
\email{jangsookim@skku.edu}

\author{Seung Jin Lee}
\address{Department of Mathematical Sciences, Research Institute of Mathematics, 
Seoul National University, Seoul 151-747,
South Korea}
\email{lsjin@snu.ac.kr}

\author{Meesue Yoo}
\address{Department of Mathematics, Chungbuk National University, Cheongju 28644,
South Korea}
\email{meesueyoo@chungbuk.ac.kr (\Letter)}

\keywords{unicellular LLT polynomials, chromatic quasisymmetric functions, Hall--Littlewood polynomials, rook placements}

\subjclass[2010]{Primary: 05A15; Secondary: 05A30}

\date{\today}

\newtheorem{thm}{Theorem}[section]
\newtheorem{lem}[thm]{Lemma}
\newtheorem{prop}[thm]{Proposition}

\theoremstyle{definition}
\newtheorem{exam}[thm]{Example}
\newtheorem{defn}[thm]{Definition}
\newtheorem{conj}[thm]{Conjecture}
\newtheorem{remark}[thm]{Remark}

\AddToHook{env/lem/begin}{\crefalias{thm}{lem}}
\AddToHook{env/prop/begin}{\crefalias{thm}{prop}}
\AddToHook{env/cor/begin}{\crefalias{thm}{cor}}
\AddToHook{env/exam/begin}{\crefalias{thm}{exam}}
\AddToHook{env/defn/begin}{\crefalias{thm}{defn}}
\AddToHook{env/conj/begin}{\crefalias{thm}{conj}}
\AddToHook{env/remark/begin}{\crefalias{thm}{remark}}
\newcommand\set{\operatorname{set}}

\newcommand\area{\operatorname{area}}
\newcommand\fc{\operatorname{fc}}
\newcommand\ext{\operatorname{ext}}
\newcommand\LRP{\operatorname{LRP}}

\newcommand\ZZ{\mathbb{Z}}
\newcommand\CC{\mathbb{C}}
\newcommand\Qbinom[3]{\genfrac{[}{]}{0pt}{}{#1}{#2}_{#3}}
\newcommand\qbinom[2]{\Qbinom{#1}{#2}{q}}

\newcommand\MM{\mathcal{M}}

\newcommand\D{\mathcal{D}}

\newcommand\Fl{\mathcal{F}l}
\newcommand\Hess{\mathcal{H}ess}
\newcommand\SSYT{\operatorname{SSYT}}

\newcommand\Imm{\operatorname{Imm}}

\newcommand\vx{\boldsymbol{x}}
\newcommand\vy{\boldsymbol{y}}
\newcommand\vm{\boldsymbol{m}}
\newcommand\RS{\operatorname{RS}}
\newcommand\inv{\operatorname{inv}}
\newcommand\LLT{\operatorname{LLT}}
\newcommand\Poin{\operatorname{Poin}}

\definecolor{pink1}{RGB}{219, 48, 122}

\definecolor{dredcolor}{rgb}{0.9,0.3,0.4}

\definecolor{etonblue}{rgb}{0.59, 0.78, 0.64}
\definecolor{canaryyellow}{rgb}{1.0, 0.94, 0.0}
\definecolor{frenchblue}{rgb}{0.0, 0.45, 0.73}
\definecolor{oxfordblue}{rgb}{0.0, 0.13, 0.28}
\definecolor{cerulean}{rgb}{0.0, 0.48, 0.65}
\definecolor{asparagus}{rgb}{0.53, 0.66, 0.42}

\begin{document}

\begin{abstract}
  In this work, we obtain a Hall--Littlewood expansion of the
  chromatic quasisymmetric function arising from a natural unit
  interval order and describe the coefficients in terms of linked rook
  placements. Applying the Carlsson--Mellit relation between chromatic
  quasisymmetric functions and unicellular LLT polynomials, we also
  obtain a combinatorial description for the coefficients of the
  unicellular LLT polynomials expanded in terms of the modified
  transformed Hall--Littlewood polynomials.
\end{abstract}

\maketitle
% \tableofcontents

\section{Introduction}

The main objective of this paper is to provide a combinatorial formula
for the Hall--Littlewood expansion of the chromatic quasisymmetric
function in terms of linked rook placements. We first present the
background and motivation for this work.

The \emph{immanant} of an \( n \times n \) matrix \( A=(a_{i,j}) \)
with respect to a partition \(\lambda \) of \( n \)
is defined by
\[
  \Imm_\lambda (A)=\sum_{w\in S_n}\chi ^\lambda(w)\prod_{i=1}^n a_{i,
    w(i)},
\]
where \( \chi^\lambda \) denotes the irreducible character of the
symmetric group \( S_n \) associated with \( \lambda \). If
\( \lambda=(1^n) \), the immanant becomes the determinant. The
Jacobi--Trudi formula expresses the (skew) Schur function
\( s_{\mu/\nu} \) as
\[
  s_{\mu/\nu} = \det H_{\mu/\nu},
\]
where
\( H_{\mu/\nu} = \left( h_{\mu_i -\nu_j-i+j} \right)_{1\le i, j\le n}
\) is the Jacobi--Trudi matrix of complete homogeneous symmetric
functions \( h_\lambda \). This implies that
\( \Imm_\lambda H_{\mu/\nu} \) is \( s \)-positive when
\( \lambda=(1^n) \). Here, for a symmetric function \( f \) and a
basis \( \{g_\lambda\} \) of the space of symmetric functions, we say
that \( f \) is \emph{\( g \)-positive} if every coefficient
\( c_\lambda \) in the expansion
\( f= \sum_{\lambda}c_\lambda g_\lambda \) is nonnegative. Goulden and
Jackson \cite{Goulden1992a} conjectured that
\( \Imm_\lambda H_{\mu/\nu} \) is \( m \)-positive for any
\( \lambda \), which was proved by Greene \cite{Greene1992},
where \( m_\lambda \) is the monomial symmetric function. More
generally, Haiman \cite{Haiman1993a} showed that
\( \Imm_\lambda H_{\mu/\nu} \) is \( s \)-positive for any partition
\( \lambda \).

In \cite{Ste92}, Stembridge presented several conjectures concerning
immanants. Stanley and Stembridge \cite{Stanley1993a} rephrased one of
them as follows. For partitions \( \mu \) and \( \nu \) with at most
\( n \) parts, let
\[
  F_{\mu/\nu}(\bm x, \bm y) =\sum_{\lambda\vdash n} s_\lambda (\bm y) \Imm_\lambda H_{\mu/\nu}(\bm x),
\]
where the sum is over all partitions \( \lambda \) of \( n \), and
\( \bm x=(x_1,x_2,\dots) \) and \( \bm y=(y_1, y_2,\dots) \) are
independent variables. Let
\( E_{\mu/\nu}^\theta = E_{\mu/\nu}^\theta(\vy) \) be the coefficient
of \( s_\theta(\vx) \) in \( F_{\mu/\nu}(\bm x, \bm y) \), that is,
\[
  F_{\mu/\nu}(\bm x, \bm y) = \sum_{\theta\vdash N}E_{\mu/\nu}^\theta (\bm y)s_{\theta}(\bm x),
\]
where \( N=|\mu|-|\nu| \). The following is a restatement of Stembridge's conjecture \cite{Ste92} 
written in \cite[Conjecture~1.1]{Stanley1993a}.

\begin{conj}[{\cite[Conjecture~1.1]{Stanley1993a}}]\label{con:1}
The symmetric function \( E_{\mu/\nu}^\theta \) is \( h \)-positive.  
\end{conj}

This conjecture implies Haiman's result that
\( \Imm_\lambda H_{\mu/\nu} \) is \( s \)-positive. The case
\( \theta=(N) \) is particularly interesting due to its connection
with chromatic symmetric functions introduced by Stanley
\cite{Stan95}. For a graph \( G \) with a vertex set
\( [n]=\{1,2,\dots,n\} \), the \emph{chromatic symmetric function}
\( X_G(\vx) \) is defined by
\[
  X_G(\vx) = \sum_{\kappa} x_{\kappa(1)} \cdots x_{\kappa(n)},
\]
where the sum is over all proper colorings
\( \kappa:[n]\to \ZZ_{\ge1} \) of \( G \). In
\cite[Conjecture~5.1]{Stan95}, Stanley reformulated \Cref{con:1} for
\( \theta=(N) \) in terms of chromatic symmetric functions \( X_G(\vx) \)
and elementary symmetric functions \( e_\lambda \) as follows.

\begin{thm}[{\cite[Conjecture~5.1]{Stan95}, \cite[Theorem~1.6]{Hikita}}]\label{con:2}
  If \( G \) is the incomparability graph of a natural unit interval
  order, then \( X_G(\vx) \) is \( e \)-positive.
\end{thm}

This statement was originally known as the \emph{Stanley--Stembridge conjecture}.
In fact, the original conjecture was stated using a broader class of
posets, namely, \( (3+1) \)-free posets. However, Guay-Paquet
\cite{Guay-Paquet2013} showed that the two statements are equivalent.
The Stanley--Stembridge conjecture has been studied extensively by
many researchers. We remark that it was recently proved by Hikita
\cite{Hikita}, and later by Griffin et al.~\cite{Griffin2025}.

In the following, we highlight the connection between
\Cref{con:1,con:2} that is described in \cite{Stanley1993a}. Stanley
and Stembridge \cite[Proposition~4.1]{Stanley1993a} expressed
\( E_{\mu/\nu}^{(N)} \) in terms of permutations with restricted
positions, which naturally correspond to rook placements on a subset
of the \( [n]\times [n] \) grid. By \cite[(5.1)]{Stanley1993a}, we
have
\begin{equation}\label{eq:30}
  \omega(E^{(N)}_{\mu/\nu}) = \sum_{\alpha\vdash n} r_{\gamma,\alpha} \left( \prod_{i\ge1} m_i(\alpha)! \right) m_\alpha,
\end{equation}
where \( \omega \) is the involution sending \( s_\lambda \) to
\( s_{\lambda'} \), \( \gamma \) is a certain (Ferrers) board in
\( [n]\times[n] \) that is determined by the partitions \( \mu \) and
\( \nu \), \( m_i(\alpha) \) is the multiplicity of \( i \) in a
partition \( \alpha \), and \( r_{\gamma,\alpha} \) is the number of
non-attacking rook placements of type \( \alpha \) on the board
\( \gamma \); see \Cref{sec:preliminaries,sec:HL} and \Cref{rem:2}
for precise definitions. Here, the board \( \gamma \) may be
identified with a Dyck path. Moreover, it is known that there is a
correspondence between natural unit interval orders and Dyck paths
\cite{GMR2014, SR2003}, so we use the notation \( X_\gamma(\bm x) \)
to denote the chromatic symmetric function indexed by the natural unit
interval order corresponding to the Dyck path \( \gamma \).

On the other hand, considering the definition of a chromatic symmetric
function, one can deduce that the chromatic symmetric function
\( X_\gamma (\bm x) \) also satisfies the same formula.
Hence, we obtain
\( \omega(E^{(N)}_{\mu/\nu}) = X_\gamma (\bm x) \), which explains the
equivalence between \Cref{con:1} in the case \( \theta=(N) \) and
\Cref{con:2}. The goal of this paper is to present a
\( q \)-analogue of this formula.

There are well-known \( q \)-analogues of chromatic symmetric
functions and monomial symmetric functions. In \cite{SW}, Shareshian
and Wachs introduced the \emph{chromatic quasisymmetric function}
\( X_G (\bm x;q) \) for any graph \( G \), which reduces to the
original chromatic symmetric function \( X_G (\bm x) \) when
\( q=1 \). They proved that \( X_G (\bm x ;q) \) is symmetric when
\( G \) is the incomparability graph of a natural unit interval order.
The \emph{Hall--Littlewood polynomials} \( P_\mu(\vx;q) \) are
symmetric functions, which generalize both Schur functions and
monomial symmetric functions, namely, \( P_\mu(\vx;0) = s_\mu(\vx) \)
and \( P_\mu(\vx;1) = m_\mu(\vx) \).

The main result of this paper is a combinatorial description of the
Hall--Littlewood expansion of \( X_\gamma (\bm x;q) \), which gives a
natural \( q \)-analogue of \eqref{eq:30}. To define a
\( q \)-analogue \( r_{\gamma,\mu}(q) \) of the number
\( r_{\gamma,\mu} \) of rook placements with given conditions, we
introduce the notion of \emph{linked rook placements}. Linked rook
placements are essentially the same as ordinary rook placements, but
they provide a useful visualization to determine the \( q \)-weight of a rook
placement; see \Cref{sec:HL} for precise definitions.
 Our main result is then stated as follows.
\begin{thm}\label{thm:main}
  For a Dyck path \(\gamma\in \D_n\), we have
\begin{equation}\label{eq:31}
  X_{\gamma}({\bm x};q)
  =\sum_{\mu\vdash n} q^{{\rm area}(\gamma) -n(\mu)} r_{\gamma,\mu}(q)
  \left( \prod_{i\ge 1}[m_{i} (\mu)]_q ! \right) P_{\mu}({\bm x};q).
\end{equation}
\end{thm}

In \cite{GR86}, the \emph{\( q \)-rook polynomial}
\( R_k (B_\gamma;q) \) for a Ferrers board \( B_\gamma \) associated with \( \gamma \) is defined by
\[
  R_k(B_\gamma;q)=\sum_C q^{\inv (C)},
\]
where the sum is over the rook placements \( C \) of \( k \)
non-attacking rooks in \( B_\gamma \) and \( \inv(C) \) is a certain
statistic. Considering the principal specialization of
\( X_\gamma (\bm x;q) \), we can derive the following relation
\[
   R_{n-k}(B_\gamma ;q) = \sum_{\substack{\mu\vdash n\\ \ell(\mu) = k}}r_{\gamma, \mu}(q),
\]
which implies that \( r_{\gamma,\mu}(q) \) is a refinement of the original \( q \)-rook polynomial; see Section \ref{sec:remarks} for the details.

We prove \Cref{thm:main} by showing that the right-hand side of
\eqref{eq:31} satisfies the modular laws and a multiplicativity
property in \cite{AN}. Note that \( q \)-rook polynomials and
\( q \)-hit polynomials have already appeared in the study of the
chromatic symmetric functions (cf. \cite{AN,CMP,HOY}). Hence the
existence of such polynomials that refine \( q \)-rook polynomials
and that satisfy the modular laws may provide a better understanding of
various positivity phenomena in chromatic symmetric functions.

By applying the Carlsson--Mellit relation (\Cref{prop:CM}) between the chromatic
symmetric functions and the unicellular LLT polynomials (see \Cref{subsec:LLT} for details), we also
obtain the following modified Hall--Littlewood expansions of the
unicellular LLT polynomials.

\begin{thm}\label{thm:main_LLT}
  For a Dyck path \(\gamma\in \D_n\), we have
  \[
      \LLT_\gamma ({\bm x};q)
=\sum_{\mu\vdash n}(1-q)^{n-\ell(\mu)}q^{{\rm area}(\gamma)-n(\mu)}r_{\gamma, \mu}(q)\, \omega (H_\mu ({\bm x};q))
  \]
  and
\[
      \LLT_\gamma ({\bm x};q)
  =\sum_{\mu\vdash n} (1-q^{-1})^{n-\ell (\mu)}r_{\gamma, \mu}(q^{-1})  \tilde{H}_\mu ({\bm x};q),
\]
where \( H_\mu (\bm x;q) \) is the modified Hall--Littlewood polynomial and \( \tilde{H}_\mu ({\bm x};q)=q^{n(\mu)}H_\mu ({\bm x};q^{-1}) \).

\end{thm}

\begin{remark}\label{rem:1}
  There are several known results relating Hall--Littlewood expansions of
  chromatic quasisymmetric functions to Hessenberg varieties \cite{JP, PS, Gag, ANR, Tym}.
  Moreover, Griffin et al.~\cite{Griffin2025} recently obtained a Macdonald
  expansion of \( X_\gamma({\bm x};q) \) whose specialization at \( t=0 \) gives
  a Hall--Littlewood expansion.
  We compare their formula with ours in \Cref{sec:appendix} and explain how our linked-rook
  description differs from the previously known tableau formulas. 
\end{remark}

This paper is organized as follows. In Section
\ref{sec:preliminaries}, we provide definitions and classical
background material that we use throughout the paper. In Section
\ref{sec:HL}, we introduce linked rook placements and extended rook
placements, which are crucial ingredients in describing the
Hall--Littlewood coefficients of the chromatic symmetric functions. In
Section \ref{sec:proof}, we prove the main theorem by showing that the
Hall--Littlewood coefficients satisfy the modular laws and a
multiplicativity property, following Abreu--Nigro's criterion (see Theorem
\ref{thm:AN2}). We conclude this paper by presenting some open
problems that require further investigation in Section
\ref{sec:remarks}.

\section{Preliminaries}\label{sec:preliminaries}

In this section, we introduce necessary definitions and recall some
known results.

\subsection{Basic definitions}
\label{sec:basic-definitions}

For a positive integer \(n\), we use the notation \([n]\) to denote
the set \(\{1,2,\dots,n\}\) and \([n]_q\) to denote the \(q\)-integer
\([n]_q=\frac{1-q^n}{1-q}= 1+q+\cdots+q^{n-1}\). We also define
\( (a;q)_n = (1-a)(1-aq) \cdots (1-aq^{n-1}) \),
\( [n]_q!=[1]_q[2]_q \cdots [n]_q \), and
\( \qbinom{n}{k} = \frac{[n]_q!}{[k]_q! [n-k]_q!} \).

A \emph{partition} is a weakly decreasing sequence
\( \lambda=(\lambda_1,\dots,\lambda_\ell) \) of positive integers,
called \emph{parts}. The number \( \ell \) of parts is called the
\emph{length} of \( \lambda \), denoted by \( \ell(\lambda) \). We
identify the partition \( \lambda \) with its \emph{Young diagram}
that consists of \( \lambda_i \) cells in the \( i \)-th row from the
top for all \( i\in [\ell] \). We use \( |\lambda| \) to denote the
sum of the parts of \( \lambda \). If \( |\lambda|=n \), then we say
that \( \lambda \) is a partition of \( n \) and write
\( \lambda\vdash n \). We also consider the partition
\( \lambda=(\lambda_1,\dots,\lambda_\ell) \) as the infinite sequence
\( (\lambda_1,\lambda_2,\dots) \) or a sequence
\( (\lambda_1,\dots,\lambda_m) \) for some \( m>\ell \), where
\( \lambda_i=0 \) if \( i>\ell \). The \emph{conjugate} of
\( \lambda \), denoted \( \lambda' \), is defined as the partition
whose Young diagram is obtained from the Young diagram of
\( \lambda \) by reflecting across the main diagonal. We denote by
\(\delta_n \) the staircase partition
\( (n-1, n-2, \dots, 2, 1, 0) \). For a partition \( \lambda \), we
let \( m_i (\lambda) \) be the number of \( i \)'s in \( \lambda \)
and \( n(\lambda)=\sum_{i\ge1} (i-1)\lambda_i \).

Given a partition \( \lambda \), a \emph{semistandard Young tableau}
of shape \( \lambda \) is a filling of \( \lambda \) with elements in
\( \mathbb{Z}_{>0} \) such that the entries are weakly increasing
along the rows and strictly increasing along the columns, from top to
bottom. For a semistandard Young tableau \( T \), we let the
\emph{weight} of \( T \), denoted by \( \text{wt}(T) \), be the
sequence \( (\text{wt}_1 (T),\text{wt}_2 (T),\dots )\), where
\( \text{wt}_i (T) \) is the number of occurrences of \( i \)'s in
\( T \). We use \( \SSYT(\lambda) \) to denote the set of semistandard
Young tableaux of shape \( \lambda \). We also use
\( \SSYT(\lambda, \mu) \) to denote the set of semistandard Young
tableaux of shape \( \lambda \) and weight \( \mu \).

A \emph{Dyck path} is a lattice path from \( (0,0) \) to \( (n,n) \)
consisting of North steps \(N= (0,1) \) and East steps \(E= (1,0) \)
that never goes below the line \( y=x \). Let \( \D_n \) denote the
set of Dyck paths from \( (0,0) \) to \( (n,n) \), and let
\(\displaystyle \D=\cup_{n\ge0} \D_n \) be the set of all Dyck paths.
Given a Dyck path \( \gamma\in \D_n \), let \( a_i (\gamma) \), for
\( 1\le i\le n \), denote the number of cells in the \( i \)-th row
from the bottom that lie between \( \gamma \) and the line \( y=x \).
We define the \emph{area} of \( \gamma \),
denoted \( \area(\gamma) \), to be the sum of \( a_i (\gamma) \), for
\( 1\le i\le n \).

\subsection{Symmetric functions}
\label{sec:symmetric-functions}

Let \( \vx= (x_1,x_2,\dots) \) be an infinite sequence of variables.
We denote by \( \Lambda_q=\Lambda_q(\vx) \) the algebra of symmetric
functions in the variables \( \vx \) with coefficients in
\( \mathbb{Q}(q) \). Each element in \( \Lambda_q \) is a formal power
series \( f(\vx) \) in the variables \( \vx \), with a bounded degree,
that is invariant under permutations of the variables. If there is no
possible confusion, we will omit the variables \( \vx \) and simply
write \( f \) instead of \( f(\vx) \).

For a partition \( \lambda=(\lambda_1, \lambda_2, \dots) \), the
\emph{monomial symmetric function} \( m_\lambda \) is defined by
\( m_\lambda =\sum_\nu \vx^\nu \), where the sum is over
all distinct permutations \( \nu=(\nu_1,\nu_2,\dots) \) of the
infinite sequence \( (\lambda_1, \lambda_2, \dots) \) and
\( \vx^\nu= \prod_{i\ge1} x_i ^{\nu_i} \). The \emph{elementary
  symmetric function} \( e_\lambda  \) is defined by
\( e_\lambda =\prod_{i\ge 1} e_{\lambda_i} \), where
\( e_0 = 1 \) and
\( e_k =\sum_{i_1<\cdots < i_k}x_{i_1}\cdots x_{i_k} \), for
\( k\ge 1 \).
The \emph{complete homogeneous symmetric function} \( h_\lambda (\vx) \) is defined by \( h_\lambda =\prod_{i\ge1} h_{\lambda_i} \), where \( h_0 = 1 \) and \( h_k =\sum_{i_1\le \cdots \le  i_k}x_{i_1}\cdots x_{i_k} \), for \( k\ge 1 \).
The \emph{power sum symmetric function} \( p_\lambda  \) is
defined by \( p_\lambda=\prod_{i\ge 1} p_{\lambda_i} \),
where \( p_0 =1 \) and \( p_k =\sum_{i\ge 1} x_i^k \), for
\( k\ge 1 \).

The \emph{Schur functions} \( s_\lambda(x_1,\dots,x_n) \) were first
studied by Cauchy in \cite{Cauchy}, where he defined them as a ratio
of alternants: for a partition \( \lambda \) with at most \( n \)
parts,
\begin{equation}\label{eqn:Schur}
  s_\lambda (x_1,\dots,x_{n}) =\frac{\det \left( x_i ^{\lambda_j +n-j} \right)_{i,j=1}^n}{\det \left( x_i ^{n-j} \right)_{i,j=1}^n}.
\end{equation}
From this definition, we see that \( s_\lambda (x_1,\dots,x_{n}) \) is
a symmetric function in the variables \( x_1,\dots,x_n \) with
homogeneous degree \( |\lambda| \), and that
\( s_\lambda(x_1,\dots,x_n,0) = s_\lambda(x_1,\dots,x_n) \). This
allows us to define the Schur function \( s_\lambda(\vx) \) with
infinitely many variables \( \vx=(x_1,x_2,\dots) \) by
\( s_\lambda(\vx) = \lim_{n\to \infty} s_\lambda(x_1,\dots,x_n) \),
which lies in \( \Lambda_q \).

Later, an equivalent combinatorial description of \( s_\lambda(\vx) \)
was found by Kostka \cite{Kostka}:
\begin{equation}\label{eqn:Schur_Kostka}
  s_\lambda(\vx)  =\sum_{T\in \SSYT (\lambda)}\vx^T=\sum_{\mu\vdash |\lambda|}K_{\lambda\mu}m_{\mu}(\vx),
\end{equation}
where \( \vx^{T} =\prod_i x_i ^{\text{wt}_i (T)} \) and
\( K_{\lambda\mu} \) is the number of semistandard Young tableaux of
shape \( \lambda \) and weight \( \mu \), called the \emph{Kostka
  number}.

\medskip

Plethysm is a certain type of substitution on symmetric functions. Here, we briefly introduce some plethystic
notations. For a fuller account, we refer the reader to \cite{Hag08, Mac}.

Let \( E=E(t_1, t_2,\dots) \) be a formal Laurent series in variables
\( t_1, t_2,\dots \) with rational coefficients. For any symmetric
function \(f\in \Lambda_q \), the \emph{plethystic substitution}
\( f[E] \) is defined as follows. First, for a power sum symmetric
function \( p_k \), we define \( p_k [E]=E(t_1 ^k, t_2 ^k,\dots) \);
that is, \( p_k [E] \) is obtained by replacing each \( t_i \) in
\( E \) by \( t_i ^k \). For a partition \( \lambda \), define
\( p_\lambda [E] = \prod_{i=1}^{\ell(\lambda)} p_{\lambda_i }[E] \).
Then, for a general \(f\in \Lambda_q \), we define
\[
 f[E] = \sum_{\lambda} c_\lambda p_\lambda [E],
\]
where the \( c_\lambda \) are the coefficients in the expansion
\( f=\sum_{\lambda} c_\lambda p_\lambda \).

Let \( X=x_1 + x_2 + \cdots \). Then for any
\( f\in \Lambda_q(\vx) \), we have \( f[X] = f(\vx) \). Moreover, if
\( f \) is homogeneous of degree \( n \), then
\begin{equation}\label{eq:28}
  f[-X] = (-1)^n \omega (f(\vx)),
\end{equation}
where \( \omega \) is the involution on \( \Lambda_q \) sending
\( s_\lambda \) to \( s_{\lambda'} \).

\subsection{Hall--Littlewood polynomials}

The \emph{Hall--Littlewood polynomials} \cite[p. 208]{Mac} are defined
as follows: for a partition \(\mu\) with at most \( n \) parts,
\begin{equation}\label{eq:29}
  P_\mu (x_1,\dots, x_n; q)=\frac{1}{\prod_{i\ge 0} [m_i (\mu)]_q !}
  \sum_{w\in S_n}w\left( x_1 ^{\mu_1}\cdots x_n ^{\mu_n}\prod_{1\le
      i<j\le n}\frac{x_i -qx_j}{x_i -x_j} \right),
\end{equation}
where \(m_i (\mu)\) is the number of \(i\)'s in
\(\mu=(\mu_1,\dots,\mu_n)\), \( S_n \) is the symmetric group on
\( [n] \), and \( w \) is the ring homomorphism with
\( w(x_i)=x_{w(i)} \). From the definition, we have
\( P_\mu(x_1,\dots,x_n,0;q) = P_\mu(x_1,\dots,x_n;q) \). Hence, as in the
case of Schur functions, the Hall--Littlewood polynomials
\( P_\mu(\vx;q) \) with infinitely many variables are defined by
\( P_\mu(\vx;q) = \lim_{n\to \infty} P_\mu(x_1,\dots,x_n;q) \).

Note that by comparing the definition of Schur functions in
\eqref{eqn:Schur} we see that the Schur functions arise from the
Hall--Littlewood polynomials by setting \(q=0\):
\[
  P_\mu (x_1,\dots, x_n ; 0) = s_\mu (x_1,\dots, x_n).
\]
Thus, Hall--Littlewood polynomials can be considered as
\(q\)-analogues of Schur functions. Moreover, if \( q=1 \), then the
right-hand side of \eqref{eq:29} becomes the monomial symmetric
function \( m_\mu(x_1,\dots,x_n) \). Hence, \( P_\mu(\vx;q) \)
generalizes both \( m_\mu(\vx) \) and \( s_\mu(\vx) \).

In terms of the \emph{modified Schur functions} \(s_\lambda [X(1-q)]\),
the Hall--Littlewood polynomials have the following expansion:
\[
  P_\mu (\vx;q)=\sum_{\lambda\vdash
    |\mu|}\frac{K_{\lambda\mu}(q)}{\prod_{i\ge1}
    (q;q)_{m_i(\mu)}}s_{\lambda}[X(1-q)].
\]
Here, \(K_{\lambda\mu}(q)\) is the \emph{Kostka--Foulkes polynomial}
\[K_{\lambda\mu}(q)=\sum_{T\in\SSYT (\lambda,\mu)}q^{{\rm ch}(T)},\]
with \({\rm ch}(T)\) the \emph{charge} statistic (see \cite{LS1978, Sch1978, Butler1994, Mac} for the details).
Then we can see that the expression \(\prod_{i\ge 1} (q;q)_{m_i(\mu)}P_\mu (\vx;q)\) has a positive 
expansion in terms of \(s_\lambda [X(1-q)]\), namely,
\begin{equation}\label{eqn:Jmu}
\prod_{i\ge1} (q;q)_{m_i(\mu)}P_\mu (\vx;q)= \sum_{\lambda\vdash |\mu|} K_{\lambda\mu}(q) s_\lambda [X(1-q)].
\end{equation}
Note that the left-hand side of \eqref{eqn:Jmu} is equal to the integral form Macdonald polynomial \(J_\mu (\vx;q, t)\), with \(q=0\) and \(t\) replaced 
by \(q\).

We define the \emph{modified Hall--Littlewood polynomials} \( H_\mu (\vx;q) \) by
replacing the modified Schur function by the Schur function in the
expansion \eqref{eqn:Jmu}:
\[H_\mu (\vx;q)=\sum_{\lambda\vdash |\mu|} K_{\lambda\mu}(q) s_{\lambda} (\vx).\]
Lastly, we define the \emph{modified transformed Hall--Littlewood polynomials}
\( \tilde{H}_\mu (\vx;q) \) by 
\[\tilde{H}_\mu (\vx;q)=\sum_{\lambda\vdash |\mu|} \tilde{K}_{\lambda\mu}(q) s_{\lambda} (\vx),\]
where \(\tilde{K}_{\lambda\mu}(q)=q^{n(\mu)}K_{\lambda\mu}(q^{-1})\), with \(n(\mu)=\sum_{i\ge1} (i-1)\mu_i\).
Note that, by definition, we have
\[
\tilde{H}_\mu (\vx;q)=q^{n(\mu)}H_\mu (\vx;q^{-1}).
\]

\subsection{Natural unit interval orders}\label{subsec:NUI}
Let \(P\) be a poset. The \emph{incomparability graph} \({\rm inc}(P)\) of
\(P\) is the graph whose vertex set is the set of elements in \(P\), with an
edge between each pair of incomparable elements.

Let \(\vm= (m_1, m_2, \dots, m_{n})\) be a nondecreasing sequence of
positive integers such that \(i\leq m_i \leq n\) for each \(i\in[n]\).
The \emph{natural unit interval order} \(P(\vm)\) corresponding to
\( \vm \) is the poset on \([n]\) with the order relation given by
\(i<_{P(\vm)} j\) if \( m_i+1\le j\le n \). These three related
objects \(\vm\), \(P(\vm)\), and \({\rm inc}(P)\) can be drawn in one
picture, by using the Dyck path with the sequence of column heights
\( (m_1, m_2, \dots, m_{n})\). We use \( (i,j) \) to denote the cell
at the intersection of the \( i \)-th column from the left and the
\( j \)-th row from the bottom. The diagonal cell \( (i,i) \)
represents the poset element \( i \) in \( P \) as well as the vertex
\( i \) of \( {\rm inc}(P) \), for each \( 1\le i\le n \). The cells
between the Dyck path and the main diagonal correspond to the edges of
\( {\rm inc}(P) \); that is, if the cell \((i,j)\) lies in this
region, then \(\{i,j\}\) is an edge of \({\rm inc}(P)\). The region
above the Dyck path represents the poset \(P\); namely, if \((i,j)\)
is a cell in this region, then \(i<_{P(\vm)} j\). Then the edges of
\( {\rm inc}(P) \) correspond to the pairs \( \{i, j\} \) such that
\( 1\le i<j\le n \) and \( i\nless_{P(\vm)} j\). If \( \gamma \) is
the Dyck path with the column height sequence \( \vm \), then we also
use the notation \( i<_{\gamma} j \) for \(i<_{P(\vm)} j\), and
\( i\nless_{\gamma} j \) for \( i\nless_{P(\vm)} j\), respectively.
See Figure \ref{fig:NUI_Dyck} for an example.

\begin{figure}
\begin{tikzpicture}[scale=.6]
   \fill[color=gray!20!white] (0,1)--(0,2)--(1,2)--(1,3)--(2,3)--(2,5)--(3,5)--(3,6)--(5,6)--(5,5)--(4,5)--(4,4)--(3,4)--(3,3)--(2,3)--(2,2)--(1,2)--(1,1)--(0,1)--cycle;
  \foreach \i in {0,...,6} {
 \draw[color=gray] (0,\i)--(6,\i);
  \draw[color=gray] (\i,0)--(\i,6);
  }
  \foreach \i in {1,...,6}{
\node () at (-.5, \i-.5) {$\i$};
\node () at (\i-.5, 6.5) {$\i$};
\node () at (\i-.5, \i-.5) {\bf $\i$};
}
%\draw [dashed] (0,0)--(6,6);
\draw[very thick] (0,0)--(0,2)--(1,2)--(1,3)--(2,3)--(2,5)--(3,5)--(3,6)--(6,6);
\node[] at (3, -.8) {\( \vm =(2,3,5,6,6,6) \)};
\end{tikzpicture}\\
\begin{tikzpicture}[scale=.52]
  \draw[thick] (2.2, 2.7)--(4.9, .4)
  (2.2, 3.3)--(2.8, 5.7)
  (3.3, 5.7)--(9.7, .2)
  (5.1, .4)--(6.8, 4.7)
  (5.3,.4)--(9.8, 4.7)
  (7.2, 4.7)--(9.75, .45)
 (10, .4)--(10, 4.7) ;
\node () at (5,0) {\( \bf{1} \)};
  \node () at (10,0) {\( \bf{2} \)};
  \node () at (2,3) {\( \bf{3} \)};
  \node () at (7,5) {\( \bf{4} \)};
  \node () at (10,5) {\( \bf{5} \)};
  \node () at (3,6) {\( \bf{6} \)};
  \node[] at (0, 7) {\phantom{Sdf}};
  \node[] at (.2, 3) {\( P(\vm): \)};
\end{tikzpicture}\qquad\quad
\begin{tikzpicture}[scale=.52]
  \draw[thick] (0,0)--(10,0);
  \draw [thick]  (4,0) to[out=90,in=90] (8,0);
  \draw [thick]  (6,0) to[out=90,in=90] (10,0);
  \filldraw (0,0) circle (4pt)
  (2,0) circle (4pt)
  (4,0) circle (4pt)
  (6,0) circle (4pt)
  (8,0) circle (4pt)
  (10,0) circle (4pt);
  \node () at (0, -.5) {\( 1 \)};
  \node () at (2, -.5) {\( 2 \)};
  \node () at (4, -.5) {\( 3 \)};
  \node () at (6, -.5) {\( 4 \)};
  \node () at (8, -.5) {\( 5 \)};
  \node () at (10, -.5) {\( 6 \)};
  \node[] at (5, -1.8) {\( \text{inc}(P) \)};
  \node[] at (0, -2) {\phantom{s}};
\end{tikzpicture}  
\caption{The diagram at the top shows the Dyck path with column heights
  \( \vm=(2,3,5,6,6,6) \). The diagrams below show the corresponding
  natural unit interval order \( P(\vm) \) on the left and its
  incomparability graph \( \text{inc}(P) \) on the right.
  The shaded cells correspond to the edges of the incomparability graph, while the
  cells above the Dyck path correspond to order relations in the poset.}
\label{fig:NUI_Dyck}
\end{figure}

In this paper, when referring to a Dyck path \( \gamma \), we
implicitly identify the associated objects \( \vm \), \( P(\vm) \),
and \( \text{inc}(P) \).

\subsection{Chromatic quasisymmetric functions}\label{subsec:CQS}

Given a simple graph \(G=(V, E)\) with \(V=[n]\), a \emph{proper coloring} is any function \(\kappa : V \rightarrow \{1,2,3,\dots\}\)
satisfying the condition that \(\kappa(u)\neq \kappa(v)\) for any \(u, v\in V\) such that \(\{u,v\}\in E\).  

\begin{defn}
The \emph{chromatic quasisymmetric function} of \(G\) is defined by 
\[X_G (\vx ;q)= \sum_{\kappa}q^{{\rm asc}(\kappa)} \vx^{\kappa},\]
where the sum is over all proper colorings
\(\kappa : V \rightarrow \{1,2,3,\dots\}\),
\[
  {\rm asc}(\kappa) = |\{ \{ i, j\}\in E \,:\, i < j \text{ and } \kappa(i)<\kappa(j) \} |,
\]
 and
\(\vx^\kappa =\prod_{v\in V}x_{\kappa(v)}\).
\end{defn}

As explained in Section \ref{subsec:NUI}, if we consider the
incomparability graph of a natural unit interval order, then we can
use the Dyck path associated with the sequence \(\vm\), as illustrated
in Figure \ref{fig:NUI_Dyck}, to compute the corresponding chromatic
quasisymmetric function. Let us use the notation \(X_\gamma (\vx; q)\)
to denote the chromatic quasisymmetric function
\(X_{{\rm inc} (P(\vm))}(\vx;q)\), where \(\gamma\) is the Dyck path
with column heights given by \( \vm \).

We remark that the incomparability graph of a natural unit interval
order is particularly significant, since Shareshian and
Wachs~\cite[Theorem 4.5]{SW} proved that \( X_G (\vx;q) \) is
symmetric if \( G \) arises in this way. Furthermore, in \cite{SW},
they conjectured the \( e \)-positivity of \( X_\gamma (\vx;q) \),
which is a refinement of the \( e \)-positivity conjecture originally
proposed by Stanley and Stembridge \cite{Stanley1993a}. 
The latter has now been proven by Hikita \cite{Hikita} and also by Griffin et
al.~\cite{Griffin2025}.

 Given a Dyck path \( \gamma\in \D \), we let \( \gamma(i) \) denote the \( i \)-th column height.

\begin{defn}\label{def:modular_law}
  Let \( (\gamma_0, \gamma_1, \gamma_2) \in \D_n^3 \) and
  \( i\in [n-1] \). We say that \( (\gamma_0, \gamma_1, \gamma_2) \)
  is a \emph{modular triple of type}~\( (1,i) \) if the following
  conditions hold:
  \begin{enumerate}
  \item \(\gamma_0(i)+1=\gamma_1(i)=\gamma_2(i)-1\) and \(\gamma_1(i-1)<\gamma_1(i)<\gamma_1(i+1)\), where \(\gamma_1(0)=0\).
  \item \(\gamma_0(j)=\gamma_1(j)=\gamma_2(j)\) for all \(j\in [n]\setminus \{i\}\).
  \item \(\gamma_1(\gamma_1(i))=\gamma_1(\gamma_1(i)+1)\).
  \end{enumerate}
  We say that \( (\gamma_0, \gamma_1, \gamma_2) \) is a \emph{modular
    triple of type} \( (2,i) \) if the following conditions hold:
  \begin{enumerate}
  \item \(\gamma_1(i)+1=\gamma_1(i+1)\), \( \gamma_0(i) = \gamma_1(i) = \gamma_2(i)-1 \), and
    \( \gamma_0(i+1)+1 = \gamma_1(i+1) = \gamma_2(i+1) \).
  \item \(\gamma_0(j)=\gamma_1(j)=\gamma_2(j)\) for all \(j\in [n]\setminus \{i,i+1\}\).
  \item \( \gamma_1^{-1}(\{i\})=\emptyset \).
  \end{enumerate}

  For \( t\in \{1,2\} \), we denote by \( \MM_n^{t,i} \) the set of
  modular triples \( (\gamma_0, \gamma_1, \gamma_2) \in \D_n^3 \) of
  type \( (t,i) \). See Figure \ref{fig:hs}.

  A function \( f:\D\rightarrow \Lambda_q \) is said
  to \emph{satisfy the modular law} if
  \[
    (1+q)f(\gamma_1)=q f(\gamma_0)+f(\gamma_2),
  \]
  for any \( (\gamma_0, \gamma_1, \gamma_2)\in \MM_n^{t,i} \) for
  \( n\ge1 \), \( i\in [n-1] \), and \( t\in \{1,2\} \).
\end{defn}

\begin{figure}
  \centering
\begin{tikzpicture}[scale=.45]
\fill[color=cyan!13!white](1,3)--(2,3)--(2,5)--(1,5)--(1,3)--cycle;
\foreach \i in {1,...,7}
\draw[dotted] (0,\i)--(\i,\i);
\foreach \i in {0,...,7}
\draw[dotted] (\i,7)--(\i,\i);
\draw[dotted] (0,0)--(7,7);
\draw[very thick] (3,6)--(5,6)
(1,3)--(2,3)--(2,5);
\draw[dashed] (2,3.5)--(3.5, 3.5)--(3.5, 6)
(2,4.5)--(4.5, 4.5)--(4.5, 6);
\node[] at (4.5, 2.5) {\( \gamma_0 \)};
\node[] at (1.5,7.5) {\( i \)};
\node[] at (-.5,4.5) {\( j \)};
\node[] at (-.9,3.5) {{\small \( j \)-1}};
  \end{tikzpicture}\quad
\begin{tikzpicture}[scale=.45]
\fill[color=cyan!13!white](1,4)--(2,4)--(2,5)--(1,5)--(1,4)--cycle;
\foreach \i in {1,...,7}
\draw[dotted] (0,\i)--(\i,\i);
\foreach \i in {0,...,7}
\draw[dotted] (\i,7)--(\i,\i);
\draw[dotted] (0,0)--(7,7);
\draw[very thick] (3,6)--(5,6)
(1,3)--(1,4)--(2,4)--(2,5);
\draw[dashed] (2,3.5)--(3.5, 3.5)--(3.5, 6)
(2,4.5)--(4.5, 4.5)--(4.5, 6);
\node[] at (4.5, 2.5) {\( \gamma_1 \)};
\node[] at (1.5,7.5) {\( i \)};
\node[] at (-.5,4.5) {\( j \)};
\node[] at (-.9,3.5) {{\small \( j \)-1}};
  \end{tikzpicture}\quad
\begin{tikzpicture}[scale=.45]
\foreach \i in {1,...,7}
\draw[dotted] (0,\i)--(\i,\i);
\foreach \i in {0,...,7}
\draw[dotted] (\i,7)--(\i,\i);
\draw[dotted] (0,0)--(7,7);
\draw[very thick] (3,6)--(5,6)
(1,3)--(1,5)--(2,5);
\draw[dashed] (2,3.5)--(3.5, 3.5)--(3.5, 6)
(2,4.5)--(4.5, 4.5)--(4.5, 6);
\node[] at (4.5, 2.5) {\( \gamma_2 \)};
\node[] at (1.5,7.5) {\( i \)};
\node[] at (-.5,4.5) {\( j \)};
\node[] at (-.9,3.5) {{\small \( j \)-1}};
  \end{tikzpicture}\\
\phantom{sdf}\\
  \begin{tikzpicture}[scale=.45]
\fill[color=cyan!13!white](2,5)--(4,5)--(4,6)--(2,6)--(2,5)--cycle;
\foreach \i in {1,...,7}
\draw[dotted] (0,\i)--(\i,\i);
\foreach \i in {0,...,7}
\draw[dotted] (\i,7)--(\i,\i);
\draw[dotted] (0,0)--(7,7);
\draw[very thick] (2,5)--(4,5)--(4,6)
(1,2)--(1,4);
\draw[dashed] (1,2.5)--(2.5,2.5)--(2.5,5)
(1, 3.5)--(3.5, 3.5)--(3.5, 5);
\node[] at (-.5,5.5) {\( j \)};
\node[] at (2.5,7.5) {\( i \)};
\node[] at (3.6,7.5) {{\small\( i\)+1}};
\node[] at (4.5, 2.5) {\( \gamma_0 \)};
\end{tikzpicture}\quad
    \begin{tikzpicture}[scale=.45]
\fill[color=cyan!13!white](2,5)--(3,5)--(3,6)--(2,6)--(2,5)--cycle;
\foreach \i in {1,...,7}
\draw[dotted] (0,\i)--(\i,\i);
\foreach \i in {0,...,7}
\draw[dotted] (\i,7)--(\i,\i);
\draw[dotted] (0,0)--(7,7);
\draw[very thick] (2,5)--(3,5)--(3,6)--(4,6)
(1,2)--(1,4);
\draw[dashed] (1,2.5)--(2.5,2.5)--(2.5,5)
(1, 3.5)--(3.5, 3.5)--(3.5, 5);
\node[] at (-.5,5.5) {\( j \)};
\node[] at (2.5,7.5) {\( i \)};
\node[] at (3.6,7.5) {{\small\( i\)+1}};
\node[] at (4.5, 2.5) {\( \gamma_1 \)};
\end{tikzpicture}\quad
\begin{tikzpicture}[scale=.45]
\foreach \i in {1,...,7}
\draw[dotted] (0,\i)--(\i,\i);
\foreach \i in {0,...,7}
\draw[dotted] (\i,7)--(\i,\i);
\draw[dotted] (0,0)--(7,7);
\draw[very thick] (2,5)--(2,6)--(4,6)
(1,2)--(1,4);
\draw[dashed] (1,2.5)--(2.5,2.5)--(2.5,5)
(1, 3.5)--(3.5, 3.5)--(3.5, 5);
\node[] at (-.5,5.5) {\( j \)};
\node[] at (2.5,7.5) {\( i \)};
\node[] at (3.6,7.5) {{\small\( i\)+1}};
\node[] at (4.5, 2.5) {\( \gamma_2 \)};
\end{tikzpicture}
    \caption{Modular triples of type \( (1,i) \) (top) and type \( (2,i) \) (bottom).}
  \label{fig:hs}
\end{figure}

Abreu and Nigro \cite{AN} provided the following useful tool for proving
that the chromatic quasisymmetric function has a certain formula.

\begin{thm}[{\cite[Theorem 1.2]{AN}}]\label{thm:AN2}
  Let \(A\) be a \(\mathbb{Q}(q)\)-algebra and let
  \(f:\D\rightarrow A\) be a function that satisfies the modular law.
  Then \(f\) is determined by its values
  \(f(K_{n_1}\sqcup K_{n_2}\sqcup \cdots \sqcup K_{n_\ell})\) on the
  disjoint unions of complete graphs \( K_{n_1},K_{n_2},\dots,K_{n_\ell} \),
  for all \( \ell\ge0 \) and \( n_1,n_2,\dots,n_\ell\ge1 \),
  where the empty union is included when \( \ell=0 \).
\end{thm}

\subsection{LLT polynomials}\label{subsec:LLT}

LLT polynomials are a family of symmetric functions introduced by
Lascoux, Leclerc, and Thibon in \cite{LLT}, which naturally arise in
the description of the power-sum plethysm operators on symmetric
functions. The original definition of LLT polynomials uses the
\emph{cospin} statistic of ribbon tableaux, but Haiman and Bylund
found a consistent statistic, called \emph{inv}, defined on
\(k\)-tuples of semistandard Young tableaux of various skew shapes.
The Bylund--Haiman formula is described in \cite{HHLRU}, where it is
proved that the two definitions are equivalent.

When the skew shape consists of single cells, the LLT
polynomials are called \emph{unicellular}. In this case, the diagram
of a tuple of single cells corresponds to a Dyck path \cite{CM, Hag08}.
Since we only consider the unicellular LLT polynomials in this paper,
we define them over the set of Dyck paths.

\begin{defn}
Given a Dyck path \(\gamma\in\D_n\), the \emph{unicellular LLT polynomial} indexed by \(\gamma\) is
\[
  \LLT_\gamma (\vx; q) = \sum_{w\in \mathbb{Z}_{>0}^n} q^{\inv(w)}\vx^w,
\]
where \(\vx^w=\prod_{i=1}^n x_{w_i}\) and
\(\inv(w) = |\{ (i, j)  \,:\, 1\le i<j\le n,\, i\nless_{\gamma} j, \text{ and } w_i>w_j \} |\).
\end{defn}

Combining the observations from  Sections \ref{subsec:CQS} and
\ref{subsec:LLT}, we see that both the unicellular LLT polynomials and
the chromatic quasisymmetric functions can be computed with respect to a Dyck
path \(\gamma\), and that the \(\inv\) and \({\rm asc}\)
statistics essentially count the same pairs of integers. A subtle  discrepancy
between the two definitions lies in the properness of the colorings
used when computing the chromatic quasisymmetric functions, as
opposed to the fact that all possible words of length \(n\) are
considered when computing \(\LLT_\gamma (\vx;q)\). The precise
relationship between the two has been explained via plethysm by Carlsson and
Mellit in \cite[Proposition 3.4]{CM}.

\begin{prop}[{\cite[Proposition 3.4]{CM}}]
  \label{prop:CM}
Let \(\gamma\) be the Dyck path corresponding to 
the incomparability graph \({\rm inc}(P(\vm))\) of a natural unit interval order \(P(\vm)\) with \(n\) elements. Then we have 
\begin{equation}\label{eqn:CM}
 X_\gamma (\vx;q)= (q-1)^{-n}\LLT_\gamma [(q-1)X;q]\, .
 \end{equation}
\end{prop}

We remark that Alexandersson and Sulzgruber~\cite[Theorem~2.1]{AS22} gave a characterization
of vertical-strip LLT polynomials, which include the unicellular LLT polynomials.
  Their characterization in the unicellular case~\cite[Theorem~3.5]{AS22},
together with the Carlsson--Mellit relation \eqref{eqn:CM}, yields an
alternative characterization of the chromatic quasisymmetric functions
arising from natural unit interval orders, analogous to \Cref{thm:AN2}.

\section{Hall--Littlewood expansions}\label{sec:HL}

In this section, we introduce linked rook placements and state our
main theorem.

\subsection{Linked rook placements}

Let \( n \) be a fixed positive integer.

\begin{defn}
  Let \( \gamma\in \D_n \), and let \( \lambda\subseteq \delta_n \)
  be the diagram consisting of the cells above \( \gamma \) in
  \( [0,n]\times[0,n] \). A \emph{linked rook} of
  \( \gamma \) is a sequence \( L =(\nu_1, \dots, \nu_{\ell}) \) of
  cells \( \nu_i=(a_i, b_i)\in \lambda \), called \emph{rooks}, with
  \( 1\le a_i<b_i\le n \) such that \( b_{i}=a_{i+1} \) for all
  \( i\in[\ell-1] \). We say that two cells \((i,j)\) and
  \((i',j')\), with \(1\le i<j\le n\) and \(1\le i'<j'\le n\),
  are \emph{linked-attacking} if
  \(\{i,j\}\cap\{i',j'\}\ne\varnothing\); see
  \Cref{fig:attacking-cells}. A \emph{linked rook placement} on
  \( \gamma \) is a (possibly empty) set \( \{L_1,\dots,L_k\} \) of linked rooks
  \( L_i \) of \( \gamma \) such that if \( u\in L_i \) and
  \( v\in L_j \) for \( i\ne j \), then \( u \) and \( v \) are not
  linked-attacking. We denote by \( \LRP(\gamma) \) the set of linked
  rook placements on \( \gamma \).
\end{defn}

\begin{figure}[ht]
  \centering
  \begin{tikzpicture}[scale=.5, every node/.style={font=\small}]
    % A cell (c,r) has lower-left corner (c-1,r-1).
    % Highlight cells linked-attacking with (i,j)=(3,6), with c<r.
    \foreach \c in {1,...,8}{
      \foreach \r in {1,...,8}{
        \ifnum\c<\r\relax
          \ifnum\c=3\relax
            \fill[orange!40] (\c-1,\r-1) rectangle (\c,\r);
          \fi
          \ifnum\r=6\relax
            \fill[orange!40] (\c-1,\r-1) rectangle (\c,\r);
          \fi
          \ifnum\r=3\relax
            \fill[orange!40] (\c-1,\r-1) rectangle (\c,\r);
          \fi
          \ifnum\c=6\relax
            \fill[orange!40] (\c-1,\r-1) rectangle (\c,\r);
          \fi
        \fi
      }
    }
    \draw[dotted] (0,0)--(0,8);
    \foreach \k in {1,...,8}{
      \draw[dotted] (0,\k)--(\k,\k);
      \draw[dotted] (\k,\k)--(\k,8);
      \node[below] at (\k-.5,0) {\(\k\)};
      \node[left] at (0,\k-.5) {\(\k\)};
    }
    \draw[dashed] (0,0)--(8,8);
    \filldraw[color=black] (2.5,5.5) circle (5pt);
  \end{tikzpicture}
  \caption{For \(n=8\) and \((i,j)=(3,6)\), the shaded cells are the
    cells \( (i',j') \) such that \((i,j)\) and \( (i',j') \) are
    linked-attacking.}
  \label{fig:attacking-cells}
\end{figure}

\begin{figure}
\begin{tikzpicture}[scale=.5]
\fill[color=cyan!13!white](0,2)--(1,2)--(1,3)--(2,3)--(2,5)--(4,5)--(4,6)--(5,6)--(5,8)--(7,8)--(7,9)--(0,9)--(0,2)--cycle;
\draw[dotted] (0,0)--(0,9);
\foreach \i in {1,...,9}{
  \node[below] at (\i-0.5,0) {\i};
  \node[left] at (0,\i-0.5) {\i};
  \draw[dotted] (0,\i)--(\i,\i);}
\foreach \i in {1,...,9}
\draw[dotted] (\i,9)--(\i,\i);
\draw[dotted] (0,0)--(9,9);
\draw[thick, color=dredcolor] (0.5,3.5)--(3.5, 3.5)--(3.5, 5.5)--(5.5,5.5)--(5.5, 8.5);
\filldraw [red] (0.5,3.5) circle (5pt)
(3.5,5.5) circle (5pt)
(5.5,8.5) circle (5pt);
\end{tikzpicture}\qquad\qquad 
\begin{tikzpicture}[scale=.5]
\fill[color=cyan!13!white](0,2)--(1,2)--(1,3)--(2,3)--(2,5)--(4,5)--(4,6)--(5,6)--(5,8)--(7,8)--(7,9)--(0,9)--(0,2)--cycle;
\draw[dotted] (0,0)--(0,9);
\foreach \i in {1,...,9}{
  \node[below] at (\i-0.5,0) {\i};
  \node[left] at (0,\i-0.5) {\i};
  \draw[dotted] (0,\i)--(\i,\i);}
\foreach \i in {1,...,9}
\draw[dotted] (\i,9)--(\i,\i);
\draw[dotted] (0,0)--(9,9);
\draw[thick, color=dredcolor] (0.5,3.5)--(3.5, 3.5)--(3.5, 5.5)--(5.5,5.5)--(5.5, 8.5);
\draw[thick, color=orange] (1.5,4.5)--(4.5, 4.5)--(4.5, 6.5);
\filldraw [red] (0.5,3.5) circle (5pt)
(3.5,5.5) circle (5pt)
(5.5,8.5) circle (5pt);
\filldraw [orange] (1.5,4.5) circle (5pt)
(4.5,6.5) circle (5pt);
\filldraw [blue] (2.5,7.5) circle (5pt);
\end{tikzpicture}
\caption{The diagram on the left shows a linked rook
  \( ((1,4), (4,6), (6,9)) \) of the Dyck path \( \gamma\in \D_n \)
  corresponding to the partition
  \( \lambda=(7,5,5,4,2,2,1)\subseteq \delta_9 \). The diagram on the right
  shows a linked rook placement \( \{L_1,L_2,L_3\} \) on \( \gamma \),
  where \( L_1=((1,4), (4,6), (6,9)) \), \( L_2=((2,5), (5,7)) \), and
  \( L_3=((3,8)) \).}
\label{fig:linkedrook_set}
\end{figure}

Pictorially, the rooks contained in a linked rook are connected by
``links'' as shown in \Cref{fig:linkedrook_set}. Note that if we
remove all links from a linked rook placement, then we obtain an
ordinary rook placement, that is, a set of rooks such that no two
rooks lie in a row or in a column. There is a unique way to recover
the linked rook placement from its underlying rook placement by
connecting rooks \( (i,j) \) and \( (i',j') \) with a link whenever
\( j=i' \). Hence, one can identify a linked rook placement with its
underlying rook placement.

\begin{defn}\label{def:3}
  Let \( \gamma\in \D_n \). The \emph{extended row} of \( \gamma \) is the region
  between \( y=n \) and \( y=n+1 \). An \emph{extended linked rook} of
  \( \gamma \) is a sequence \( \tilde{L}=(u_1,u_2,\dots,u_{2\ell}) \) of
  cells \( (i,j) \), called \emph{extended rooks}, with
  \( 1\le i\le n \) and \( i\le j\le n+1 \), satisfying the following conditions:
  \begin{enumerate}
  \item if \( \ell\ge 2 \), then \(L= (u_2,u_4,\dots,u_{2\ell-2}) \) is a linked rook of
    \( \gamma \),
  \item \( u_1, u_3,\dots,u_{2\ell -1} \) lie on the main diagonal,
    and \( u_{2\ell} \) lies in the extended row of \( \gamma \),
  \item for each \( i\in [\ell] \), \( u_{2i} \) is above \( u_{2i-1} \) in the same column,
  \item for each \( i\in [\ell-1] \), \( u_{2i+1} \) is to the right of \( u_{2i} \) in the same row.
  \end{enumerate}
  We define the \emph{length} \( \ell(\tilde{L}) \) of \( \tilde{L} \)
  by \( \ell(\tilde{L})=\ell \). For each \( i\in [2\ell] \), the
  \emph{rank} of \( u_{i} \) is defined to be \( \lfloor i/2 \rfloor \). We say
  that \(L= (u_2,u_4,\dots,u_{2\ell-2}) \) is the \emph{underlying
    linked rook} of \(\tilde{L} \).
\end{defn}
See Figure~\ref{fig:image2} for examples of extended linked rooks.

\begin{figure}
  \centering
  \begin{tikzpicture}[scale=.5]
\fill[color=cyan!13!white](0,2)--(1,2)--(1,3)--(2,3)--(2,5)--(4,5)--(4,6)--(5,6)--(5,8)--(7,8)--(7,9)--(0,9)--(0,2)--cycle;
\draw[dotted] (0,0)--(0,9);
\foreach \i in {1,...,9}{
  \node[below] at (\i-0.5,0) {\i};
  \node[left] at (0,\i-0.5) {\i};
  \draw[dotted] (0,\i)--(\i,\i);}
\foreach \i in {1,...,9}
\draw[dotted] (\i,9)--(\i,\i);
\draw[dotted] (0,0)--(9,9);
\draw[thick] (2.5, 2.5)--(2.5, 9.5);
\filldraw (2.5, 2.5) circle (5pt)
(2.5, 9.5) circle (5pt);
\node[] at (2.5, 10) {\tiny\( 1 \)};
\node[] at (3, 2.5) {\tiny\( 0 \)};
\end{tikzpicture}
\qquad \qquad 
    \begin{tikzpicture}[scale=.5]
\fill[color=cyan!13!white](0,2)--(1,2)--(1,3)--(2,3)--(2,5)--(4,5)--(4,6)--(5,6)--(5,8)--(7,8)--(7,9)--(0,9)--(0,2)--cycle;
\draw[dotted] (0,0)--(0,9);
\foreach \i in {1,...,9}{
  \node[below] at (\i-0.5,0) {\i};
  \node[left] at (0,\i-0.5) {\i};
  \draw[dotted] (0,\i)--(\i,\i);}
\foreach \i in {1,...,9}
\draw[dotted] (\i,9)--(\i,\i);
\draw[dotted] (0,0)--(9,9);
\draw[thick] (.5, .5)--(.5, 3.5)--(3.5, 3.5)--(3.5, 5.5)--(5.5, 5.5)--(5.5, 8.5)--(8.5, 8.5)--(8.5, 9.5);
\filldraw (.5, 3.5) circle (5pt)
(.5, .5) circle (5pt)
(3.5, 3.5) circle (5pt)
(3.5, 5.5) circle (5pt)
(5.5, 5.5) circle (5pt)
(8.5, 8.5) circle (5pt)
(8.5, 9.5) circle (5pt)
(5.5, 8.5) circle (5pt); 
\node[] at (.5, 4) {\tiny\( 1 \)};
\node[] at (3.5, 6) {\tiny\( 2 \)};
\node[] at (5, 8.6) {\tiny\( 3 \)};
\node[] at (8.5, 10) {\tiny\( 4 \)};
\node[] at (1, 0.5) {\tiny\( 0 \)};
\end{tikzpicture}
\caption{The diagram on the left shows an extended linked rook of length
  \( 1 \), and the diagram on the right shows an extended linked rook of length
  \( 4 \) whose underlying linked rook is the left diagram in
  \Cref{fig:linkedrook_set}. The ranks of the topmost extended rook in
  its column and the leftmost extended rook in its row are indicated.}
  \label{fig:image2}
\end{figure}

\begin{defn}
  An \emph{extended linked rook placement} on \( \gamma\in\D_n \) is a
  set \( P=\{\tilde{L}_1,\dots,\tilde{L}_k\} \) of extended linked rooks \( \tilde{L}_i \) of
  \( \gamma \) whose nonempty underlying linked rooks form a linked
  rook placement on \( \gamma \) and such that every square in the main diagonal
  is contained in exactly one of the extended linked rooks. The
  \emph{type} of \( P \) is defined to be the partition obtained by
  rearranging the sequence \( (\ell(\tilde{L}_1),\dots,\ell(\tilde{L}_k)) \) in weakly
  decreasing order.
\end{defn}

Note that for a linked rook placement \( P \) of \( \gamma\in\D_n \),
there is a unique extended linked rook placement on \( \gamma \) whose
underlying linked rook placement is \( P \), which we denote by
\( \ext(P) \). Hence, the extended linked rook placements
can be identified with the linked rook placements,
and thus with the ordinary rook placements.
The type of \( P \) is defined to be the type of
\( \ext(P) \). See \Cref{fig:ELRP}. We denote by
\( \LRP(\gamma,\mu) \) the set of linked rook placements on
\( \gamma \) with type \( \mu \).

\begin{figure}
  \centering
  \begin{tikzpicture}[scale=.5]
    \fill[color=cyan!13!white](0,3)--(1,3)--(1,4)--(2,4)--(2,5)--(4,5)--(4,7)--(6,7)--(6,9)--(0,9)--(0,3)--cycle;
\draw[dotted] (0,0)--(0,9);
\foreach \i in {1,...,9}{
  \node[below] at (\i-0.5,0) {\i};
  \node[left] at (0,\i-0.5) {\i};
  \draw[dotted] (0,\i)--(\i,\i);
}
\foreach \i in {1,...,9}
\draw[dotted] (\i,9)--(\i,\i);
\draw[dotted] (0,0)--(9,9);
\draw[thick] (.5, 3.5)--(3.5, 3.5)--(3.5, 8.5);
\draw[thick, color=blue] (1.5, 4.5)--(4.5, 4.5)--(4.5, 7.5);
\filldraw (.5, 3.5) circle (5pt)
(3.5, 8.5) circle (5pt); 
\filldraw[color=blue] (1.5, 4.5) circle (5pt)
(4.5, 7.5) circle (5pt);
\end{tikzpicture}\qquad\quad
  \begin{tikzpicture}[scale=.5]
    \fill[color=cyan!13!white](0,3)--(1,3)--(1,4)--(2,4)--(2,5)--(4,5)--(4,7)--(6,7)--(6,9)--(0,9)--(0,3)--cycle;
\draw[dotted] (0,0)--(0,9);
\foreach \i in {1,...,9}{
  \node[below] at (\i-0.5,0) {\i};
  \node[left] at (0,\i-0.5) {\i};
  \draw[dotted] (0,\i)--(\i,\i);
}
\foreach \i in {1,...,9}
\draw[dotted] (\i,9)--(\i,\i);
\draw[dotted] (0,0)--(9,9);
\draw[thick] (.5, .5)--(.5, 3.5)--(3.5, 3.5)--(3.5, 8.5)--(8.5, 8.5)--(8.5, 9.5);
\draw[thick, color=blue] (1.5, 1.5)--(1.5, 4.5)--(4.5, 4.5)--(4.5, 7.5)--(7.5, 7.5)--(7.5, 9.5);
\draw[thick, color=orange] (2.5, 2.5)--(2.5, 9.5);
\draw[thick, color=red] (5.5, 5.5)--(5.5, 9.5);
\draw[thick, color=purple] (6.5, 6.5)--(6.5, 9.5);
\filldraw (.5, 3.5) circle (5pt)
(.5, .5) circle (5pt)
(3.5, 3.5) circle (5pt)
(8.5, 8.5) circle (5pt)
(8.5, 9.5) circle (5pt)
(3.5, 8.5) circle (5pt); 
\filldraw[color=blue] (1.5, 4.5) circle (5pt)
(4.5, 7.5) circle (5pt)
(1.5, 1.5) circle (5pt)
(4.5, 4.5) circle (5pt)
(7.5, 7.5) circle (5pt)
(7.5, 9.5) circle (5pt);
\filldraw[color=orange] (2.5, 2.5) circle (5pt)
(2.5, 9.5) circle (5pt);
\filldraw[color=red] (5.5, 5.5) circle (5pt)
(5.5, 9.5) circle (5pt);
\filldraw[color=purple] (6.5, 6.5) circle (5pt)
(6.5, 9.5) circle (5pt);
\node[] at (.5, 4) {\tiny\( 1 \)};
\node[] at (2.5, 10) {\tiny\( 1 \)};
\node[] at (5.5, 10) {\tiny\( 1 \)};
\node[] at (6.5, 10) {\tiny\( 1 \)};
\node[] at (1.5, 5) {\tiny\( 1 \)};
\node[] at (3.5, 9) {\tiny\( 2 \)};
\node[] at (4.5, 8) {\tiny\( 2 \)};
\node[] at (8.5, 10) {\tiny\( 3 \)};
\node[] at (7.5, 10) {\tiny\( 3 \)};
\node[] at (1, 0.5) {\tiny\( 0 \)};
\node[] at (2, 1.5) {\tiny\( 0 \)};
\node[] at (3, 2.5) {\tiny\( 0 \)};
\node[] at (6, 5.5) {\tiny\( 0 \)};
\node[] at (7, 6.5) {\tiny\( 0 \)};
\end{tikzpicture}
\caption{The diagram on the left shows a linked rook placement \( P \) of the
  Dyck path \( \gamma\in \D_9 \) corresponding to
  \( \lambda=(6,6,4,4,2,1) \). The right diagram shows its
  corresponding extended linked rook placement \( \ext(P) \), where the
  ranks of the topmost extended rook in each column and the leftmost
  extended rook in each row are indicated. The type of \( P \), or
  equivalently, the type of \( \ext(P) \), is \( (3,3,1,1,1) \).
}
  \label{fig:ELRP}
\end{figure}

\begin{remark}\label{rem:2}
  Using the identification between rook placements and linked rook
  placements, we also define the type of a rook placement to be the
  type of the corresponding linked rook placement. This definition is
  used in the description of \( r_{\gamma,\alpha} \) in \eqref{eq:30}.
\end{remark}

\begin{defn}\label{def:fcpair}
  Let \( \gamma\in \D_n \) and \( P\in \LRP(\gamma) \). Consider a
  cell \( (i,j) \) with \( 1\le i<j\le n \) above \( \gamma \). Let
  \( a \) be the rank of the topmost extended rook \( (i,j') \) of
  \( \ext(P) \) in column \( i \) and let \( b \) be the rank of the
  leftmost extended rook \( (i',j) \) of \( \ext(P) \) in row \( j \), with \( j'>j \).
  The cell \( (i,j) \) of \( P \) is \emph{free} (with respect to
  \( \gamma \)) if
  \begin{center}
   (\( i<i' \) and \( b\le a \)) \quad or \quad (\( i'<i \) and \( b< a \)). 
  \end{center}
   In this case, we call \( ((i,j'),(i',j)) \) an
  \emph{fc-pair}. We denote by \( \fc_\gamma(P) \) the number of free cells
  of \( P \).
\end{defn}

See Figure~\ref{fig:image3} for a pictorial description of fc-pairs.
Note that in Definition \ref{def:fcpair},
\( (i',j) \) may lie on the main diagonal and \( (i,j') \) may lie in
the extended row, i.e., we may have \( i'=j \) or \( j'=n+1 \). Note
also that, by definition, the free cells are in bijection with the
fc-pairs. Hence, \( \fc_{\gamma}(P) \) is equal to the number of fc-pairs of
\( P \).

\begin{figure}[ht]
  \centering
  \begin{tikzpicture}[scale=.36]
    \draw[dotted] (.5, 3.5)--(.5, .5)--(3.5, .5);
    \filldraw (.5, 3.5) circle (7pt)
    (3.5, .5) circle (7pt);
    \draw[thick, fill=white] (.5, .5) circle (7pt);
    \node[] at (.6, 4.4) {\( a \)};
    \node[] at (3.5, 1.4) {\( b \)};
    \node[] at (.4, -.6) {\( (i, j) \)};
    \node[] at (3.7, -.6) {\( (i', j) \)};
    \node[] at (-1.2, 3.8) {\( (i, j') \)};
    \node[rotate=330] at (2.3, 2.3) {\( \ge \)};
  \end{tikzpicture}\qquad\qquad 
  \begin{tikzpicture}[scale=.36]
    \draw[dotted] (.5, .5)--(3.5, .5)--(3.5, 3.5);
    \filldraw (3.5, 3.5) circle (7pt)
    (.5, .5) circle (7pt);
    \draw[thick, fill=white] (3.5, .5) circle (7pt);
    \node[] at (3.5, 4.4) {\( a \)};
    \node[] at (.5, 1.4) {\( b \)};
    \node[] at (.4, -.6) {\( (i', j) \)};
    \node[] at (3.7, -.6) {\( (i, j) \)};
    \node[] at (5.2, 3.8) {\( (i, j') \)};
    \node[rotate=45] at (2, 2.3) {\( < \)};
    \end{tikzpicture}
    \caption{The two cases of fc-pairs, where \( a \) is the rank of
      the topmost extended rook in column \( i \) and \( b \) is the
      rank of the leftmost extended rook in row \( j \).}
  \label{fig:image3}
\end{figure}

\begin{defn}\label{def:1}
  For \( \gamma\in \D_n \) and \( \mu\vdash n \), we define
\[
  r_{\gamma,\mu}(q) = \sum_{P\in\LRP(\gamma,\mu)} q^{\fc_{\gamma}(P)}.
\]
\end{defn}

\subsection{Hall--Littlewood expansion}

We now state our main theorem, which gives a combinatorial
interpretation for the coefficients of the Hall--Littlewood expansion
of chromatic quasisymmetric functions.

\begin{thm}\label{thm:HLexp}
  For a Dyck path \(\gamma\in \D_n\), we have
\begin{equation}\label{eqn:HLexpthm}
X_{\gamma}(\vx;q)=\sum_{\mu\vdash n} q^{\area(\gamma) -n(\mu)} r_{\gamma,\mu}(q) \prod_{i\ge 1}[m_{i} (\mu)]_q ! P_{\mu}(\vx;q).
\end{equation}
\end{thm}

We will prove this theorem in Section \ref{sec:proof}.

\begin{exam}
  Consider the Dyck path \(\gamma\) with column heights
  \((2,2,4,4,5)\). The coefficient of \(P_{(3,2)}(\vx;q)\) in
  \(X_{\gamma}(\vx;q)\) is \(q^2 +2q+1\), which can be computed by
  considering the linked rook placements of type \((3,2)\). In Figure
  \ref{fig:type32}, all possible linked rook placements and the
  corresponding extended linked rook placements of type \( (3,2) \)
  are given. Note that in this example, \(\area(\gamma) = 2\) and
  \(n(\mu)=2\), so \(\area(\gamma)-n(\mu)=0\).

\begin{figure}
  \begin{center}
\begin{tikzpicture}[scale=.5]
\fill[color=cyan!13!white](0,2)--(2,2)--(2,4)--(4,4)--(4,5)--(0,5)--(0,2)--cycle;
\draw[dotted] (0,0)--(0,5);
\foreach \i in {1,...,5}
\draw[dotted] (0,\i)--(\i,\i);
\foreach \i in {1,...,5}
\draw[dotted] (\i,5)--(\i,\i);
\draw[dotted] (0,0)--(5,5);
\draw[thick, color=dredcolor] (0.5,2.5)--(2.5, 2.5)--(2.5, 4.5);
\filldraw [red] (0.5,2.5) circle (5pt)
(2.5,4.5) circle (5pt);
\filldraw [blue] (1.5,3.5) circle (5pt);
\node[] at (1, -.5) {\phantom{sdf}};
\end{tikzpicture}
\qquad 
\begin{tikzpicture}[scale=.5]
\fill[color=cyan!13!white](0,2)--(2,2)--(2,4)--(4,4)--(4,5)--(0,5)--(0,2)--cycle;
\draw[dotted] (0,0)--(0,5);
\foreach \i in {1,...,5}
\draw[dotted] (0,\i)--(\i,\i);
\foreach \i in {1,...,5}
\draw[dotted] (\i,5)--(\i,\i);
\draw[dotted] (0,0)--(5,5);
\draw[thick, color=dredcolor] (1.5,2.5)--(2.5, 2.5)--(2.5, 4.5);
\filldraw [red] (1.5,2.5) circle (5pt)
(2.5,4.5) circle (5pt);
\filldraw [blue] (0.5,3.5) circle (5pt);
\node[] at (1, -.5) {\phantom{sdf}};
\end{tikzpicture}
\qquad
\begin{tikzpicture}[scale=.5]
\fill[color=cyan!13!white](0,2)--(2,2)--(2,4)--(4,4)--(4,5)--(0,5)--(0,2)--cycle;
\draw[dotted] (0,0)--(0,5);
\foreach \i in {1,...,5}
\draw[dotted] (0,\i)--(\i,\i);
\foreach \i in {1,...,5}
\draw[dotted] (\i,5)--(\i,\i);
\draw[dotted] (0,0)--(5,5);
\draw[thick, color=dredcolor] (.5,3.5)--(3.5, 3.5)--(3.5, 4.5);
\filldraw [red] (.5,3.5) circle (5pt)
(3.5,4.5) circle (5pt);
\filldraw [blue] (1.5,2.5) circle (5pt);
\node[] at (1, -.5) {\phantom{sdf}};
\end{tikzpicture}\qquad
\begin{tikzpicture}[scale=.5]
\fill[color=cyan!13!white](0,2)--(2,2)--(2,4)--(4,4)--(4,5)--(0,5)--(0,2)--cycle;
\draw[dotted] (0,0)--(0,5);
\foreach \i in {1,...,5}
\draw[dotted] (0,\i)--(\i,\i);
\foreach \i in {1,...,5}
\draw[dotted] (\i,5)--(\i,\i);
\draw[dotted] (0,0)--(5,5);
\draw[thick, color=dredcolor] (1.5,3.5)--(3.5, 3.5)--(3.5, 4.5);
\filldraw [red] 
(1.5,3.5) circle (5pt)
(3.5,4.5) circle (5pt);
\filldraw [blue] (0.5, 2.5) circle (5pt);
\node[] at (1, -.5) {\phantom{sdf}};
\end{tikzpicture}
\end{center}

\begin{center}
\begin{tikzpicture}[scale=.5]
\fill[color=cyan!13!white](0,2)--(2,2)--(2,4)--(4,4)--(4,5)--(0,5)--(0,2)--cycle;
\draw[dotted] (0,0)--(0,5);
\foreach \i in {1,...,5}
\draw[dotted] (0,\i)--(\i,\i);
\foreach \i in {1,...,5}
\draw[dotted] (\i,5)--(\i,\i);
\draw[dotted] (0,0)--(5,5);
\draw[thick, color=dredcolor] (0.5, 0.5)--(0.5,2.5)--(2.5, 2.5)--(2.5, 4.5)--(4.5, 4.5)--(4.5, 5.5);
\draw[thick, color=blue] (1.5, 1.5)--(1.5, 3.5)--(3.5, 3.5)--(3.5, 5.5);
\filldraw [red] (0.5,2.5) circle (5pt)
(2.5,4.5) circle (5pt)
(.5, .5) circle (5pt)
(2.5, 2.5) circle (5pt)
(4.5, 4.5) circle (5pt)
(4.5, 5.5) circle (5pt);
\filldraw [blue] (1.5,3.5) circle (5pt)
(1.5,1.5) circle (5pt)
(3.5,3.5) circle (5pt)
(3.5, 5.5) circle (5pt);
\node[] at (3.5, 1) {\(1\)};
\end{tikzpicture}
\qquad 
\begin{tikzpicture}[scale=.5]
\fill[color=cyan!13!white](0,2)--(2,2)--(2,4)--(4,4)--(4,5)--(0,5)--(0,2)--cycle;
\draw[dotted] (0,0)--(0,5);
\foreach \i in {1,...,5}
\draw[dotted] (0,\i)--(\i,\i);
\foreach \i in {1,...,5}
\draw[dotted] (\i,5)--(\i,\i);
\draw[dotted] (0,0)--(5,5);
\draw[thick, color=dredcolor] (1.5, 1.5)--(1.5,2.5)--(2.5, 2.5)--(2.5, 4.5)--(4.5, 4.5)--(4.5, 5.5);
\draw[thick, color=blue] (.5,.5)--(.5, 3.5)--(3.5, 3.5)--(3.5, 5.5);
\filldraw [red] (1.5,2.5) circle (5pt)
(2.5,4.5) circle (5pt)
(1.5, 1.5) circle (5pt)
(2.5, 2.5) circle (5pt)
(4.5, 4.5) circle (5pt)
(4.5, 5.5) circle (5pt);
\filldraw [blue] (0.5,3.5) circle (5pt)
(.5,.5) circle (5pt)
(3.5,3.5) circle (5pt)
(3.5, 5.5) circle (5pt);
\draw[very thick, color=orange, fill=white] (.5, 2.5) circle (6pt);
\node[] at (3.5, 1) {\(q\)};
\end{tikzpicture}
\qquad
\begin{tikzpicture}[scale=.5]
\fill[color=cyan!13!white](0,2)--(2,2)--(2,4)--(4,4)--(4,5)--(0,5)--(0,2)--cycle;
\draw[dotted] (0,0)--(0,5);
\foreach \i in {1,...,5}
\draw[dotted] (0,\i)--(\i,\i);
\foreach \i in {1,...,5}
\draw[dotted] (\i,5)--(\i,\i);
\draw[dotted] (0,0)--(5,5);
\draw[thick, color=dredcolor] (.5, .5)--(.5,3.5)--(3.5, 3.5)--(3.5, 4.5)--(4.5, 4.5)--(4.5, 5.5);
\draw[thick, color=blue] (1.5,1.5)--(1.5, 2.5)--(2.5, 2.5)--(2.5, 5.5);
\filldraw [red] (.5,3.5) circle (5pt)
(3.5,4.5) circle (5pt)
(.5, .5) circle (5pt)
(3.5, 3.5) circle (5pt)
(4.5, 4.5) circle (5pt)
(4.5, 5.5) circle (5pt);
\filldraw [blue] (1.5,2.5) circle (5pt)
(1.5,1.5) circle (5pt)
(2.5,2.5) circle (5pt)
(2.5, 5.5) circle (5pt);
\draw[very thick, color=orange, fill=white] (.5, 2.5) circle (6pt)
(2.5, 4.5) circle (6pt);
\node[] at (3.5, 1) {\(q^2\)};
\end{tikzpicture}\qquad
\begin{tikzpicture}[scale=.5]
\fill[color=cyan!13!white](0,2)--(2,2)--(2,4)--(4,4)--(4,5)--(0,5)--(0,2)--cycle;
\draw[dotted] (0,0)--(0,5);
\foreach \i in {1,...,5}
\draw[dotted] (0,\i)--(\i,\i);
\foreach \i in {1,...,5}
\draw[dotted] (\i,5)--(\i,\i);
\draw[dotted] (0,0)--(5,5);
\draw[thick, color=dredcolor] (1.5, 1.5)--(1.5,3.5)--(3.5, 3.5)--(3.5, 4.5)--(4.5, 4.5)--(4.5, 5.5);
\draw[thick, color=blue] (.5,.5)--(.5, 2.5)--(2.5, 2.5)--(2.5, 5.5);
\filldraw [red] 
(1.5,3.5) circle (5pt)
(3.5,4.5) circle (5pt)
(1.5, 1.5) circle (5pt)
(3.5, 3.5) circle (5pt)
(4.5, 4.5) circle (5pt)
(4.5, 5.5) circle (5pt);
\filldraw [blue] (0.5, 2.5) circle (5pt)
(.5,.5) circle (5pt)
(2.5,2.5) circle (5pt)
(2.5, 5.5) circle (5pt);
\draw[very thick, color=orange, fill=white] (2.5, 4.5) circle (6pt);
\node[] at (3, 1) {\(q\)};
\end{tikzpicture}
\end{center}
\caption{Linked rook placements of type \((3,2)\) and their extended
  linked rook placements. Free cells are denoted by white
  bullets.}\label{fig:type32}
\end{figure}
\end{exam}

By applying the Carlsson--Mellit relation \eqref{eqn:CM} to
\Cref{thm:HLexp}, we also obtain the following Hall--Littlewood
expansion of unicellular LLT polynomials.

\begin{thm}\label{thm:HLexp_LLT}
  For a Dyck path \(\gamma\in \D_n\), we have
  \[
      \LLT_\gamma (\vx;q)
=\sum_{\mu\vdash n}(1-q)^{n-\ell(\mu)}q^{\area(\gamma)-n(\mu)}r_{\gamma, \mu}(q)\, \omega (H_\mu (\vx;q))
  \]
  and
\[
      \LLT_\gamma (\vx;q)
  =\sum_{\mu\vdash n} (1-q^{-1})^{n-\ell (\mu)}r_{\gamma, \mu}(q^{-1})  \tilde{H}_\mu (\vx;q).
\]
\end{thm}

\begin{proof}
By \eqref{eqn:CM}, we have
\begin{align*}
  \LLT_\gamma (\vx;q)  &= (q-1)^n X_\gamma \left[  \frac{X}{q-1};q  \right] \\
                         &= (-1)^n \sum_{\mu\vdash n}(1-q)^{n-\ell(\mu)}q^{\area(\gamma)-n(\mu)}r_{\gamma, \mu}(q)
  \prod_{i\ge 1}(q;q)_{m_i (\mu)}P_\mu \left[  \frac{X}{q-1};q \right].
\end{align*}
By \eqref{eqn:Jmu} and the definition of the modified Hall--Littlewood polynomials,
\[
  \prod_{i\ge 1}(q;q)_{m_i (\mu)}P_\mu \left[  \frac{X}{q-1};q \right]
  =H_\mu \left[  \frac{X(1-q)}{q-1};q \right]=(-1)^n \omega (H_\mu (\vx;q)).
\]
Thus we get the first equality in the theorem.

On the other hand, by applying the relation
\( X_\gamma (\vx;q)=q^{\area(\gamma)}X_\gamma(\vx;q^{-1})
\) to \eqref{eqn:CM}, we get
\begin{align*}
  \LLT_\gamma (\vx;q)  &= (q-1)^n q^{\area(\gamma)} X_\gamma \left[ \frac{X}{q-1} ;q^{-1}\right]  \\
                           &=\sum_{\mu\vdash n}(1-q^{-1})^{n-\ell(\mu)}q^{n(\mu)}r_{\gamma,\mu}(q^{-1})\prod_{i\ge 1}(q^{-1};q^{-1})_{m_i (\mu)}P_\mu \left[ \frac{X}{1-q^{-1}}; q^{-1} \right].
\end{align*}
Since
\[
 q^{n(\mu)}\prod_{i\ge 1}(q^{-1};q^{-1})_{m_i (\mu)}P_\mu \left[ \frac{X}{1-q^{-1}}; q^{-1} \right] =
  q^{n(\mu)}H_\mu (\vx;q^{-1})=\tilde{H}_\mu (\vx;q),
\]
we obtain the second equality.
\end{proof}

\section{Proof of the main theorem}\label{sec:proof}

In this section, we prove our main result, \Cref{thm:HLexp}, which states that
\( X_{\gamma}(\vx;q) = Y_{\gamma}(\vx;q) \), where
\begin{equation}\label{eq:Y}
  Y_{\gamma}(\vx;q) := \sum_{\mu\vdash n} q^{\area(\gamma) -n(\mu)} r_{\gamma,\mu}(q)
 P_{\mu}(\vx;q)  \prod_{i\ge1 }[m_i (\mu)]_q!.
\end{equation}

For two Dyck paths \( \gamma_1 \) and \( \gamma_2 \), we denote by
\( \gamma_1+\gamma_2 \) the Dyck path obtained from \( \gamma_1 \) by
attaching \( \gamma_2 \) at the end. Let \( N^kE^k \) be the Dyck path
consisting of \( k \) consecutive North steps followed by \( k \)
consecutive East steps.

By \Cref{thm:AN2}, to establish \Cref{thm:HLexp}, it suffices to prove
the following three propositions.

\begin{prop}[Modular law]\label{pro:1}
  Let \( n\ge1 \), \( i\in [n-1] \), and \( t\in \{1,2\} \), and let
  \( (\gamma_0,\gamma_1,\gamma_2)\in \MM_n^{t,i} \). Then
\begin{equation}\label{eq:4}
    (1+q) Y_{\gamma_1}(\vx; q) = q Y_{\gamma_0}(\vx; q) + Y_{\gamma_2}(\vx; q).
  \end{equation}
\end{prop}

\begin{prop}[Multiplicativity]\label{prop:multiplicativity}
  Let \( \gamma\in\D_n \) and \( \tilde{\gamma}=\gamma+N^kE^k \). Then we have
\[
  Y_{\tilde{\gamma}}(\vx;q) = Y_{\gamma}(\vx;q) Y_{N^kE^k}(\vx;q).
\]
\end{prop}

\begin{prop}[Complete graph case]\label{pro:5}
We have
 \[
   Y_{N^nE^n}(\vx;q)=[n]_q!e_n(\vx).
 \]
\end{prop}

We first prove \Cref{pro:5}.

\begin{proof}[Proof of \Cref{pro:5}]
  Let \( \gamma=N^n E^n \). Observe that
  \( \LRP(\gamma,\mu)=\emptyset \) unless \( \mu=(1^n) \). If
  \( \mu=(1^n) \), then \( n(\mu) = \binom{n}{2} = \area(\gamma) \).
  Since \( \LRP(\gamma,\mu) \) contains only the empty linked rook
  placement, \( r_{\gamma,\mu}(q) = 1 \). We also have
  \( P_{(1^n)}(\vx;q) = e_{n}(\vx) \); see \cite[(2.8), p.~209]{Mac}.
  Thus we obtain the result.
\end{proof}

We prove \Cref{pro:1} and \Cref{prop:multiplicativity} in
\Cref{sec:proof-modular-law} and \Cref{sec:proof-mult}, respectively.

\subsection{Proof of the modular law}
\label{sec:proof-modular-law}

In this subsection, we prove \Cref{pro:1}, which states that the
function \( Y_{\gamma}(\vx; q) \) defined in \eqref{eq:Y} satisfies
the modular law.

By \Cref{def:1}, we can rewrite \eqref{eq:Y} as
\begin{equation}\label{eq:5}
  Y_{\gamma}(\vx;q)= \sum_{\mu\vdash n}
  \left( q^{\area(\gamma)-n(\mu)} P_{\mu}(\vx;q)  \prod_{i\ge1 }[m_i (\mu)]_q! \right)
  \sum_{P\in\LRP(\gamma,\mu)} q^{\fc_{\gamma}(P)}.
\end{equation}
For a modular triple \( (\gamma_0,\gamma_1,\gamma_2) \), we have
\( \area(\gamma_0)+1 = \area(\gamma_1) = \area(\gamma_2)-1 \). Thus,
by \eqref{eq:5}, the identity \eqref{eq:4} is equivalent to the
following: for all \( \mu\vdash n \),
\begin{equation}\label{eq:3}
  (1+q) \sum_{P\in\LRP(\gamma_1,\mu)} q^{\fc_{\gamma_1}(P)}
  = \sum_{P\in\LRP(\gamma_0,\mu)} q^{\fc_{\gamma_0}(P)}
  + q \sum_{P\in\LRP(\gamma_2,\mu)} q^{\fc_{\gamma_2}(P)}.
\end{equation}

The rest of this subsection is devoted to proving \eqref{eq:3} under
the conditions in \Cref{pro:1}. To do this, we introduce some
definitions. Recall the definition of rank in \Cref{def:3}. Recall
also that a linked rook placement is determined by its underlying rook
placement.

\begin{defn}\label{def:2}
  For a linked rook placement \( P \), we define \( R_{i,i+1}(P) \) to
  be the linked rook placement obtained from \( P \) by exchanging row
  \( i \) and row \( i+1 \). More precisely, the underlying rook
  placement of \( R_{i,i+1}(P) \) is obtained from that of \( P \) by
  replacing each rook \( (a,i) \) with \( (a,i+1) \) and each rook
  \( (b,i+1) \) with \( (b,i) \). We also define \( C_{i,i+1}(P) \) to
  be the linked rook placement obtained from \( P \) by exchanging
  column \( i \) and column \( i+1 \).
  These operations are defined whenever all resulting rooks lie
  strictly above the main diagonal.
\end{defn}

See \Cref{fig:row-column-exchanges} for an illustration of the maps
\( R_{i,i+1} \) and \( C_{i,i+1} \).

\begin{figure}
  \centering
  \begin{tikzpicture}[scale=.5]
\fill[color=cyan!20!white](0,2)--(2,2)--(2,3)--(3,3)--(3,4)--(0,4)--(0,2)--cycle;
\draw[dotted] (0,0)--(0,6);
\foreach \i in {1,...,6}{
  \node[below] at (\i-0.5,0) {\i};
  \node[left] at (0,\i-0.5) {\i};
  \draw[dotted] (0,\i)--(\i,\i);}
\foreach \i in {1,...,6}
\draw[dotted] (\i,6)--(\i,\i);
\draw[dotted] (0,0)--(6,6);
\draw[color=blue,thick] (.5, 3.5)--(3.5, 3.5)
(2.5,2.5)--(2.5, 4.5);
\filldraw[color=blue] (.5, 3.5) circle (5pt)
(2.5, 4.5) circle (5pt);
\draw[color=teal,thick] (1.5, 2.5)--(2.5, 2.5)
(3.5,3.5)--(3.5, 5.5);
\filldraw[color=teal] (1.5, 2.5) circle (5pt)
(3.5, 5.5) circle (5pt);
\node[] at (3, -1.6){\( R_{3,4}(P) \)};
\end{tikzpicture}\qquad\quad
  \begin{tikzpicture}[scale=.5]
\draw[dotted] (0,0)--(0,6);
\foreach \i in {1,...,6}{
  \node[below] at (\i-0.5,0) {\i};
  \node[left] at (0,\i-0.5) {\i};
  \draw[dotted] (0,\i)--(\i,\i);}
\foreach \i in {1,...,6}
\draw[dotted] (\i,6)--(\i,\i);
\draw[dotted] (0,0)--(6,6);
\draw[color=blue,thick] (.5, 2.5)--(2.5, 2.5)--(2.5, 4.5);
\filldraw[color=blue] (.5, 2.5) circle (5pt)
(2.5, 4.5) circle (5pt);
\draw[color=teal,thick] (1.5, 3.5)--(3.5, 3.5)--(3.5, 5.5);
\filldraw[color=teal] (1.5, 3.5) circle (5pt)
(3.5, 5.5) circle (5pt);
\node[] at (3, -1.6){\( P \)};
\end{tikzpicture}\qquad\quad
  \begin{tikzpicture}[scale=.5]
\fill[color=cyan!20!white](2,3)--(3,3)--(3,4)--(4,4)--(4,6)--(2,6)--(2,3)--cycle;
\draw[dotted] (0,0)--(0,6);
\foreach \i in {1,...,6}{
  \node[below] at (\i-0.5,0) {\i};
  \node[left] at (0,\i-0.5) {\i};
  \draw[dotted] (0,\i)--(\i,\i);}
\foreach \i in {1,...,6}
\draw[dotted] (\i,6)--(\i,\i);
\draw[dotted] (0,0)--(6,6);
\draw[color=blue,thick] (.5, 2.5)--(2.5, 2.5)
(3.5, 3.5)--(3.5, 4.5);
\filldraw[color=blue] (.5, 2.5) circle (5pt)
(3.5, 4.5) circle (5pt);
\draw[color=teal,thick] (1.5, 3.5)--(3.5, 3.5)
(2.5,2.5)--(2.5, 5.5);
\filldraw[color=teal] (1.5, 3.5) circle (5pt)
(2.5, 5.5) circle (5pt);
\node[] at (3, -1.6){\( C_{3,4}(P) \)};
\end{tikzpicture}
  \caption{Row and column exchanges on \(\gamma=(NE)^6\). The middle
    diagram shows \(P=\{((1,3),(3,5)),((2,4),(4,6))\}\), and the left
    and right diagrams show \(R_{3,4}(P)\) and \(C_{3,4}(P)\),
    respectively. The shaded cells indicate the exchanged rows or
    columns.}
  \label{fig:row-column-exchanges}
\end{figure}

\begin{defn}
  Let \( \gamma\in \D_n \) and \( P\in \LRP(\gamma) \). For
  \( i\in [n] \), we denote by \( r_i(P) \) the rank of the leftmost
  rook in row \( i \) of \( \ext(P) \), and denote by \( c_i(P) \) the
  rank of the topmost rook in column \( i \) of \( \ext(P) \).
\end{defn}

\begin{exam}\label{exa:1}
  For the extended rook placement shown on the right in Figure \ref{fig:ELRP},
 the values of  \(  r_i (P)  \)  and \( c_i (P) \) are as follows.  
\[
 \begin{array}{c|ccccccccc}
  i& 1 & 2 & 3 & 4 & 5 & 6 & 7 & 8 & 9\\
  \hline
  r_i(P) &  0&0&0&1&1&0&0&2&2\\
  c_i (P) & 1&1&1&2&2&1&1&3&3
  \end{array}
\]
\end{exam}

\begin{lem}\label{lem:ML-1}
  Let \( (\gamma_0,\gamma_1,\gamma_2)\in \MM_n^{2,i} \) and
  \( \mu\vdash n \). Then \eqref{eq:3} holds.
\end{lem}

\begin{proof}
  Let \( j=\gamma_1(i)+1 \). Note that
  \[
    \LRP(\gamma_2,\mu) \subseteq \LRP(\gamma_1,\mu) \subseteq \LRP(\gamma_0,\mu).
  \]
  For \( \gamma\in \D_n \), we define
\begin{align*}
  B_0(\gamma)
  &= \{P\in \LRP(\gamma,\mu): \mbox{\( P \) has a rook in neither \( (i,j) \) nor \( (i+1,j) \)} \},\\
  B_1(\gamma)
  &= \{P\in \LRP(\gamma,\mu): \mbox{\( P \) has a rook in \( (i,j) \)} \},\\
  B_2(\gamma)
  &= \{P\in \LRP(\gamma,\mu): \mbox{\( P \) has a rook in \( (i+1,j) \)} \}.
\end{align*}
Then
\begin{align}
  \label{eq:16}
 \LRP(\gamma_0,\mu) &= B_0(\gamma_0) \sqcup B_1(\gamma_0) \sqcup B_2(\gamma_0),\\
  \label{eq:17}
 \LRP(\gamma_1,\mu) &= B_0(\gamma_1) \sqcup B_1(\gamma_1), \\
  \label{eq:18}
 \LRP(\gamma_2,\mu) &= B_0(\gamma_2).
\end{align}

We claim that
\begin{align}
\label{eq:1}  (1+q) \sum_{P\in B_0(\gamma_1)} q^{\fc_{\gamma_1}(P)}
  &= \sum_{P\in B_0(\gamma_0)} q^{\fc_{\gamma_0}(P)}
  + q \sum_{P\in B_0(\gamma_2)} q^{\fc_{\gamma_2}(P)},\\
\label{eq:2}  (1+q) \sum_{P\in B_1(\gamma_1)} q^{\fc_{\gamma_1}(P)}
  &= \sum_{P\in B_1(\gamma_0)} q^{\fc_{\gamma_0}(P)}
  + \sum_{P\in B_2(\gamma_0)} q^{\fc_{\gamma_0}(P)}.
\end{align}
By \eqref{eq:16}, \eqref{eq:17}, and \eqref{eq:18}, adding
\eqref{eq:1} and \eqref{eq:2} yields \eqref{eq:3}. Therefore, it
suffices to prove the claim.

First, we prove \eqref{eq:2}. To do this, it suffices to construct a
bijection \( \phi:B_1(\gamma_1) \to B_2(\gamma_0) \) such that
\begin{equation}\label{eq:20}
  q^{\fc_{\gamma_0}(P)} + q^{\fc_{\gamma_0}(\phi(P))} = q^{\fc_{\gamma_1}(P)} + q^{\fc_{\gamma_1}(P)+1},
\end{equation}
because by the fact that \( B_1(\gamma_1) = B_1(\gamma_0) \), summing
\eqref{eq:20} over all \( P\in B_1(\gamma_1) \) gives \eqref{eq:2}.
Let \( a=r_j(P) \) and \( b=c_{i+1}(P) \). Then we define
\[
  \phi(P) =
  \begin{cases}
    (R_{i,i+1}\circ C_{i,i+1})(P) & \mbox{if \( a\ne b \)},\\
    C_{i,i+1}(P) & \mbox{if \( a= b \)}.
  \end{cases}
\]
By the construction, one can easily check that \( \phi(P) \) has the same type as \( P \).
There are three cases to consider:
\begin{itemize}
\item If \( a>b \), by \Cref{fig:a>b}, we have
  \( \fc_{\gamma_0}(P) = \fc_{\gamma_1}(P) \) and
  \( \fc_{\gamma_0}(\phi(P)) = \fc_{\gamma_1}(P)+1 \).
\item If \( a<b \), by \Cref{fig:image9}, we have
  \( \fc_{\gamma_0}(P) = \fc_{\gamma_1}(P)+1 \) and
  \( \fc_{\gamma_0}(\phi(P)) = \fc_{\gamma_1}(P) \).
\item If \( a=b \), by \Cref{fig:image10}, we have
  \( \fc_{\gamma_0}(P) = \fc_{\gamma_1}(P) \) and
  \( \fc_{\gamma_0}(\phi(P)) = \fc_{\gamma_1}(P)+1 \).
\end{itemize}
In each case, \eqref{eq:20} holds. We note that, in
\Cref{fig:a>b,fig:image9,fig:image10}, the column positions of the
rooks in rows \( i \) and \( i+1 \) are irrelevant. For example, the
rook of rank \( a-1 \) can be farther to the left than the rook of rank
\( b-1 \); however, this does not affect the results. The map
\( \phi \) is clearly a bijection, and we obtain \eqref{eq:2}.

\begin{figure}[ht]
  \centering
    \begin{tikzpicture}[scale=.45]
\foreach \i in {3,...,9}
\draw[dotted] (0,\i)--(\i,\i);
\foreach \i in {2,...,9}
\draw[dotted] (\i,9)--(\i,\i);
\draw[dotted] (0,2)--(0,9)
(1,2)--(1,9);
\draw[dotted] (0,2)--(2,2)--(9,9);
\draw[very thick] (4,7)--(6,7)--(6,8)
(2,4)--(2,6);
\draw[thick] (1.5,4.5)--(4.5, 4.5)--(4.5, 7.5);
\draw[thick, color=blue] (.5, 5.5)--(5.5, 5.5)--(5.5, 8.5);
\filldraw (1.5, 4.5) circle (6pt)
(4.5, 7.5) circle (6pt);
\filldraw[color=blue]
(.5, 5.5) circle (6pt)
(5.5, 8.5) circle (6pt);
\draw[very thick, color=red] (.2, 4.2)--(.8, 4.8)
(.2, 4.8)--(.8, 4.2);
\draw[very thick, color=red] (5.2, 7.2)--(5.8, 7.8)
(5.2, 7.8)--(5.8, 7.2);
\node[] at (4.5, 8.3) {\small \( a \)};
\node[] at (5.5, 9.3) {\small \( b \)};
\node[] at (1.5, 3.6) {\small \( a-1 \)};
\node[] at (.6, 6.2) {\small \( b-1 \)};
\node[] at (5, 4.5) {\small \( i \)};
\node[] at (6.4, 5.3) {\small\( i+1 \)};
\node[] at (8, 7.5) {\small \( j \)};
\node[] at (3.5, 1.2) {\( \fc_{\gamma_0}(P) \)};
\end{tikzpicture}\quad
    \begin{tikzpicture}[scale=.45]
\foreach \i in {3,...,9}
\draw[dotted] (0,\i)--(\i,\i);
\foreach \i in {2,...,9}
\draw[dotted] (\i,9)--(\i,\i);
\draw[dotted] (0,2)--(0,9)
(1,2)--(1,9);
\draw[dotted] (0,2)--(2,2)--(9,9);
\draw[very thick] (4,7)--(5,7)--(5,8)--(6,8)
(2,4)--(2,6);
\draw[thick] (1.5,4.5)--(4.5, 4.5)--(4.5, 7.5);
\draw[thick, color=blue] (.5, 5.5)--(5.5, 5.5)--(5.5, 8.5);
\filldraw (1.5, 4.5) circle (6pt)
(4.5, 7.5) circle (6pt);
\filldraw[color=blue]
(.5, 5.5) circle (6pt)
(5.5, 8.5) circle (6pt);
\draw[very thick, color=red] (.2, 4.2)--(.8, 4.8)
(.2, 4.8)--(.8, 4.2);
\node[] at (4.5, 8.3) {\small \( a \)};
\node[] at (5.5, 9.3) {\small \( b \)};
\node[] at (1.5, 3.6) {\small \( a-1 \)};
\node[] at (.6, 6.2) {\small \( b-1 \)};
\node[] at (5, 4.5) {\small \( i \)};
\node[] at (6.4, 5.3) {\small\( i+1 \)};
\node[] at (8, 7.5) {\small \( j \)};
\node[] at (3.5, 1.2) {\( \fc_{\gamma_1}(P) \)};
\end{tikzpicture}\quad
    \begin{tikzpicture}[scale=.45]
\foreach \i in {3,...,9}
\draw[dotted] (0,\i)--(\i,\i);
\foreach \i in {2,...,9}
\draw[dotted] (\i,9)--(\i,\i);
\draw[dotted] (0,2)--(0,9)
(1,2)--(1,9);
\draw[dotted] (0,2)--(2,2)--(9,9);
\draw[very thick] (4,7)--(6,7)--(6,8)
(2,4)--(2,6);
\draw[thick] (1.5,5.5)--(5.5, 5.5)--(5.5, 7.5);
\draw[thick, color=blue] (.5, 4.5)--(4.5, 4.5)--(4.5, 8.5);
\filldraw (1.5, 5.5) circle (6pt)
(5.5, 7.5) circle (6pt);
\filldraw[color=blue]
(.5, 4.5) circle (6pt)
(4.5, 8.5) circle (6pt);
\draw [very thick, color=red, fill=white] (1.5, 4.5) circle (8pt);
\draw[very thick, color=red] (4.2, 7.2)--(4.8, 7.8)
(4.2, 7.8)--(4.8, 7.2);
\node[] at (4.5, 9.3) {\small \( b \)};
\node[] at (5.5, 8.3) {\small \( a \)};
\node[] at (0.5, 3.6) {\small \( b-1 \)};
\node[] at (1.6, 6.3) {\small \( a-1 \)};
\node[] at (5, 4.5) {\small \( i \)};
\node[] at (6.4, 5.3) {\small\( i+1 \)};
\node[] at (8, 7.5) {\small \( j \)};
\node[] at (3.5, 1.2) {\( \fc_{\gamma_0}(\phi(P)) \)};
\end{tikzpicture}\quad
\caption{If \( a>b \), then \( \fc_{\gamma_0}(P) = \fc_{\gamma_1}(P) = \fc_{\gamma_0}(\phi(P))-1 \).
  Circles indicate free cells.}
  \label{fig:a>b}
\end{figure}

\begin{figure}
  \centering
    \begin{tikzpicture}[scale=.45]
\foreach \i in {3,...,9}
\draw[dotted] (0,\i)--(\i,\i);
\foreach \i in {2,...,9}
\draw[dotted] (\i,9)--(\i,\i);
\draw[dotted] (0,2)--(0,9)
(1,2)--(1,9);
\draw[dotted] (0,2)--(2,2)--(9,9);
\draw[very thick] (4,7)--(6,7)--(6,8)
(2,4)--(2,6);
\draw[thick] (1.5,4.5)--(4.5, 4.5)--(4.5, 7.5);
\draw[thick, color=blue] (.5, 5.5)--(5.5, 5.5)--(5.5, 8.5);
\filldraw (1.5, 4.5) circle (6pt)
(4.5, 7.5) circle (6pt);
\filldraw[color=blue]
(.5, 5.5) circle (6pt)
(5.5, 8.5) circle (6pt);
\draw [very thick, color=red, fill=white] (.5, 4.5) circle (8pt);
\draw [very thick, color=red, fill=white] (5.5, 7.5) circle (8pt);
\node[] at (4.5, 8.3) {\small \( a \)};
\node[] at (5.5, 9.3) {\small \( b \)};
\node[] at (1.5, 3.6) {\small \( a-1 \)};
\node[] at (.6, 6.2) {\small \( b-1 \)};
\node[] at (5, 4.5) {\small \( i \)};
\node[] at (6.4, 5.3) {\small\( i+1 \)};
\node[] at (8, 7.5) {\small \( j \)};
\node[] at (3.5, 1.2) {\( \fc_{\gamma_0}(P) \)};
\end{tikzpicture}\quad
    \begin{tikzpicture}[scale=.45]
\foreach \i in {3,...,9}
\draw[dotted] (0,\i)--(\i,\i);
\foreach \i in {2,...,9}
\draw[dotted] (\i,9)--(\i,\i);
\draw[dotted] (0,2)--(0,9)
(1,2)--(1,9);
\draw[dotted] (0,2)--(2,2)--(9,9);
\draw[very thick] (4,7)--(5,7)--(5,8)--(6,8)
(2,4)--(2,6);
\draw[thick] (1.5,4.5)--(4.5, 4.5)--(4.5, 7.5);
\draw[thick, color=blue] (.5, 5.5)--(5.5, 5.5)--(5.5, 8.5);
\filldraw (1.5, 4.5) circle (6pt)
(4.5, 7.5) circle (6pt);
\filldraw[color=blue]
(.5, 5.5) circle (6pt)
(5.5, 8.5) circle (6pt);
\draw [very thick, color=red, fill=white] (.5, 4.5) circle (8pt);
\node[] at (4.5, 8.3) {\small \( a \)};
\node[] at (5.5, 9.3) {\small \( b \)};
\node[] at (1.5, 3.6) {\small \( a-1 \)};
\node[] at (.6, 6.2) {\small \( b-1 \)};
\node[] at (5, 4.5) {\small \( i \)};
\node[] at (6.4, 5.3) {\small\( i+1 \)};
\node[] at (8, 7.5) {\small \( j \)};
\node[] at (3.5, 1.2) {\( \fc_{\gamma_1}(P) \)};
\end{tikzpicture}\quad
    \begin{tikzpicture}[scale=.45]
\foreach \i in {3,...,9}
\draw[dotted] (0,\i)--(\i,\i);
\foreach \i in {2,...,9}
\draw[dotted] (\i,9)--(\i,\i);
\draw[dotted] (0,2)--(0,9)
(1,2)--(1,9);
\draw[dotted] (0,2)--(2,2)--(9,9);
\draw[very thick] (4,7)--(6,7)--(6,8)
(2,4)--(2,6);
\draw[thick] (1.5,5.5)--(5.5, 5.5)--(5.5, 7.5);
\draw[thick, color=blue] (.5, 4.5)--(4.5, 4.5)--(4.5, 8.5);
\filldraw (1.5, 5.5) circle (6pt)
(5.5, 7.5) circle (6pt);
\filldraw[color=blue]
(.5, 4.5) circle (6pt)
(4.5, 8.5) circle (6pt);
\draw [very thick, color=red, fill=white] (4.5, 7.5) circle (8pt);
\draw[very thick, color=red] (1.2, 4.2)--(1.8, 4.8)
(1.2, 4.8)--(1.8, 4.2);
\node[] at (4.5, 9.3) {\small \( b \)};
\node[] at (5.5, 8.3) {\small \( a \)};
\node[] at (0.5, 3.6) {\small \( b-1 \)};
\node[] at (1.6, 6.3) {\small \( a-1 \)};
\node[] at (5, 4.5) {\small \( i \)};
\node[] at (6.4, 5.3) {\small\( i+1 \)};
\node[] at (8, 7.5) {\small \( j \)};
\node[] at (3.5, 1.2) {\( \fc_{\gamma_0}(\phi(P)) \)};
\end{tikzpicture}\quad
  \caption{If \( a<b \), then \( \fc_{\gamma_0}(P)-1 = \fc_{\gamma_1}(P) = \fc_{\gamma_0}(\phi(P)) \).
  Circles indicate free cells.}
  \label{fig:image9}
\end{figure}

\begin{figure}
  \centering
    \begin{tikzpicture}[scale=.45]
\foreach \i in {3,...,9}
\draw[dotted] (0,\i)--(\i,\i);
\foreach \i in {2,...,9}
\draw[dotted] (\i,9)--(\i,\i);
\draw[dotted] (0,2)--(0,9)
(1,2)--(1,9);
\draw[dotted] (0,2)--(2,2)--(9,9);
\draw[very thick] (4,7)--(6,7)--(6,8)
(2,4)--(2,6);
\draw[thick] (1.5,4.5)--(4.5, 4.5)--(4.5, 7.5);
\draw[thick, color=blue] (.5, 5.5)--(5.5, 5.5)--(5.5, 8.5);
\filldraw (1.5, 4.5) circle (6pt)
(4.5, 7.5) circle (6pt);
\filldraw[color=blue]
(.5, 5.5) circle (6pt)
(5.5, 8.5) circle (6pt);
\draw [very thick, color=red, fill=white] (.5, 4.5) circle (8pt);
\draw[very thick, color=red] (5.2, 7.2)--(5.8, 7.8)
(5.2, 7.8)--(5.8, 7.2);
\node[] at (4.5, 8.3) {\small \( a \)};
\node[] at (5.5, 9.3) {\small \( a \)};
\node[] at (1.5, 3.6) {\small \( a-1 \)};
\node[] at (.6, 6.2) {\small \( a-1 \)};
\node[] at (5, 4.5) {\small \( i \)};
\node[] at (6.4, 5.3) {\small\( i+1 \)};
\node[] at (8, 7.5) {\small \( j \)};
\node[] at (3.5, 1.2) {\( \fc_{\gamma_0}(P) \)};
\end{tikzpicture}\quad
    \begin{tikzpicture}[scale=.45]
\foreach \i in {3,...,9}
\draw[dotted] (0,\i)--(\i,\i);
\foreach \i in {2,...,9}
\draw[dotted] (\i,9)--(\i,\i);
\draw[dotted] (0,2)--(0,9)
(1,2)--(1,9);
\draw[dotted] (0,2)--(2,2)--(9,9);
\draw[very thick] (4,7)--(5,7)--(5,8)--(6,8)
(2,4)--(2,6);
\draw[thick] (1.5,4.5)--(4.5, 4.5)--(4.5, 7.5);
\draw[thick, color=blue] (.5, 5.5)--(5.5, 5.5)--(5.5, 8.5);
\filldraw (1.5, 4.5) circle (6pt)
(4.5, 7.5) circle (6pt);
\filldraw[color=blue]
(.5, 5.5) circle (6pt)
(5.5, 8.5) circle (6pt);
\draw [very thick, color=red, fill=white] (.5, 4.5) circle (8pt);
\node[] at (4.5, 8.3) {\small \( a \)};
\node[] at (5.5, 9.3) {\small \( a \)};
\node[] at (1.5, 3.6) {\small \( a-1 \)};
\node[] at (.6, 6.2) {\small \( a-1 \)};
\node[] at (5, 4.5) {\small \( i \)};
\node[] at (6.4, 5.3) {\small\( i+1 \)};
\node[] at (8, 7.5) {\small \( j \)};
\node[] at (3.5, 1.2) {\( \fc_{\gamma_1}(P) \)};
\end{tikzpicture}\quad
    \begin{tikzpicture}[scale=.45]
\foreach \i in {3,...,9}
\draw[dotted] (0,\i)--(\i,\i);
\foreach \i in {2,...,9}
\draw[dotted] (\i,9)--(\i,\i);
\draw[dotted] (0,2)--(0,9)
(1,2)--(1,9);
\draw[dotted] (0,2)--(2,2)--(9,9);
\draw[very thick] (4,7)--(6,7)--(6,8)
(2,4)--(2,6);
\draw[thick] (1.5,4.5)--(4.5, 4.5)
(4.5,4.5)--(4.5, 8.5);
\draw[thick, color=blue] (.5, 5.5)--(5.5, 5.5)
(5.5, 5.5)--(5.5, 7.5);
\filldraw (1.5, 4.5) circle (6pt)
(4.5, 8.5) circle (6pt);
\filldraw[color=blue]
(.5, 5.5) circle (6pt)
(5.5, 7.5) circle (6pt);
\draw [very thick, color=red, fill=white] (4.5, 7.5) circle (8pt);
\draw [very thick, color=red, fill=white] (.5, 4.5) circle (8pt);
\node[] at (4.5, 9.3) {\small \( a \)};
\node[] at (5.5, 8.3) {\small \( a \)};
\node[] at (0.5, 6.2) {\small \( a-1 \)};
\node[] at (1.6, 3.6) {\small \( a-1 \)};
\node[] at (5, 4.5) {\small \( i \)};
\node[] at (6.4, 5.3) {\small\( i+1 \)};
\node[] at (8, 7.5) {\small \( j \)};
\node[] at (3.5, 1.2) {\( \fc_{\gamma_0}(\phi(P)) \)};
\end{tikzpicture}\quad
  \caption{If \( a=b \), then \( \fc_{\gamma_0}(P) = \fc_{\gamma_1}(P) = \fc_{\gamma_0}(\phi(P))-1 \).
  Circles indicate free cells.}
  \label{fig:image10}
\end{figure}

Now, we prove \eqref{eq:1}. Note that we have
\( B_0(\gamma_0) = B_0 (\gamma_1) = B_0 (\gamma_2) \) since their
elements have no rooks in \( (i,j) \) or \( (i+1, j) \). Given
\(P\in B_0 (\gamma_2) \), let \( a=r_j(P) \), \( b_1=c_i (P) \), and
\( b_2=c_{i+1}(P) \). We define a map
\( \theta:B_{0}(\gamma_2)\rightarrow B_{0}(\gamma_2) \) by
\[
  \theta(P) =
  \begin{cases}
    P, & \text{ if } a\ge b_1, b_2 \text{ or } a<b_1, b_2,\\
    (R_{i,i+1}\circ C_{i,i+1})(P), & \text{ if } b_1\le a< b_2 \text{ or } b_2\le a< b_1.
  \end{cases}
\]
It is easily checked that \( \theta \) is an involution. Hence, to prove
\eqref{eq:1}, it suffices to show that
for each \( P\in B_0 (\gamma_2) \), we have
\begin{equation}\label{eq:21}
  q^{\fc_{\gamma_1}(P)} + q^{\fc_{\gamma_1}(P)+1} = q^{\fc_{\gamma_0}(P)} + q^{\fc_{\gamma_2}(\theta(P))+1}.
\end{equation}
We verify \eqref{eq:21} by checking each case: 
\begin{itemize}
\item If \( a\ge b_1, b_2 \), by Figure \ref{fig:B01case}, we have \( \fc_{\gamma_0}(P) = \fc_{\gamma_1}(P)\)
  and \( \fc_{\gamma_2}(\theta(P))+1 = \fc_{\gamma_1}(P)+1\).
\item If \( a<b_1, b_2 \), by Figure \ref{fig:B02case}, we have \( \fc_{\gamma_0}(P) = \fc_{\gamma_1}(P)+1\)
  and \( \fc_{\gamma_2}(\theta(P))+1 = \fc_{\gamma_1}(P)\).
\item If \( b_1\le a< b_2 \), by Figure \ref{fig:B03case}, we have \( \fc_{\gamma_0}(P) = \fc_{\gamma_1}(P)+1\)
  and \( \fc_{\gamma_2}(\theta(P))+1 = \fc_{\gamma_1}(P)\).
\item If \( b_2\le a< b_1 \), by Figure \ref{fig:B04case}, we have \( \fc_{\gamma_0}(P) = \fc_{\gamma_1}(P)\)
  and \( \fc_{\gamma_2}(\theta(P))+1 = \fc_{\gamma_1}(P)+1\).
  \end{itemize}

\begin{figure}[ht]
  \centering
    \begin{tikzpicture}[scale=.45]
\foreach \i in {0,...,4}
\draw[dotted] (0,\i)--(3+\i,\i);
\foreach \i in {0,...,3}
\draw[dotted] (\i, 0)--(\i, 7);
\foreach \i in {0,...,4}
\draw[dotted] (4+\i, \i+1)--(4+\i,7);
\draw[dotted] (0,5)--(8,5);
\draw[dotted] (0,6)--(8,6);
\draw[dotted] (0,7)--(8,7)
(3,0)--(8,5);
\draw[very thick] (2,1)--(2,3)
(4,4)--(6,4)--(6,5);
\filldraw (3.5, 4.5) circle (6pt);
\draw[thick, color=blue] (1.5, 1.5)--(4.5, 1.5)--(4.5, 6.5); 
\filldraw[color=blue] (1.5, 1.5) circle (6pt)
(4.5, 6.5) circle (6pt);
\draw[thick, color=teal] (.5, 2.5)--(5.5, 2.5)--(5.5, 5.5);
\filldraw[color=teal] (.5, 2.5) circle (6pt)
(5.5, 5.5) circle (6pt);
\draw[very thick, color=red] (4.2, 4.2)--(4.8, 4.8)
(4.2, 4.8)--(4.8, 4.2);
\draw[very thick, color=red] (5.2, 4.2)--(5.8, 4.8)
(5.2, 4.8)--(5.8, 4.2);
\node [] at (3.5, 5.1) {\( a \)};
\node [] at (4.5, 7.4) {\small\( b_1 \)};
\node [] at (5.5, 6.3) {\small\( b_2 \)};
\node [] at (.5, 3.2) {\small\( b_2 -1\)};
\node [] at (1.5, .7) {\small\( b_1 -1\)};
\node[] at (4, -1) {\( \fc_{\gamma_0}(P) \)};
\node[] at (5, 1.5) {\small \( i \)};
\node[] at (6.4, 2.3) {\small\( i+1 \)};
\node[] at (8, 4.5) {\small \( j \)};
\end{tikzpicture}\quad
    \begin{tikzpicture}[scale=.45]
\foreach \i in {0,...,4}
\draw[dotted] (0,\i)--(3+\i,\i);
\foreach \i in {0,...,3}
\draw[dotted] (\i, 0)--(\i, 7);
\foreach \i in {0,...,4}
\draw[dotted] (4+\i, \i+1)--(4+\i,7);
\draw[dotted] (0,5)--(8,5);
\draw[dotted] (0,6)--(8,6);
\draw[dotted] (0,7)--(8,7)
(3,0)--(8,5);
\draw[very thick] (2,1)--(2,3)
(4,4)--(5,4)--(5,5)--(6,5);
\filldraw (3.5, 4.5) circle (6pt);
\draw[thick, color=blue] (1.5, 1.5)--(4.5, 1.5)--(4.5, 6.5); 
\filldraw[color=blue] (1.5, 1.5) circle (6pt)
(4.5, 6.5) circle (6pt);
\draw[thick, color=teal] (.5, 2.5)--(5.5, 2.5)--(5.5, 5.5);
\filldraw[color=teal] (.5, 2.5) circle (6pt)
(5.5, 5.5) circle (6pt);
\draw[very thick, color=red] (4.2, 4.2)--(4.8, 4.8)
(4.2, 4.8)--(4.8, 4.2);
\node [] at (3.5, 5.1) {\( a \)};
\node [] at (4.5, 7.4) {\small\( b_1 \)};
\node [] at (5.5, 6.3) {\small\( b_2 \)};
\node [] at (.5, 3.2) {\small\( b_2 -1\)};
\node [] at (1.5, .7) {\small\( b_1 -1\)};
\node[] at (4, -1) {\( \fc_{\gamma_1}(P) \)};
\node[] at (5, 1.5) {\small \( i \)};
\node[] at (6.4, 2.3) {\small\( i+1 \)};
\node[] at (8, 4.5) {\small \( j \)};
\end{tikzpicture}\quad
    \begin{tikzpicture}[scale=.45]
\foreach \i in {0,...,4}
\draw[dotted] (0,\i)--(3+\i,\i);
\foreach \i in {0,...,3}
\draw[dotted] (\i, 0)--(\i, 7);
\foreach \i in {0,...,4}
\draw[dotted] (4+\i, \i+1)--(4+\i,7);
\draw[dotted] (0,5)--(8,5);
\draw[dotted] (0,6)--(8,6);
\draw[dotted] (0,7)--(8,7)
(3,0)--(8,5);
\draw[very thick] (2,1)--(2,3)
(4,4)--(4,5)--(6,5);
\filldraw (3.5, 4.5) circle (6pt);
\draw[thick, color=blue] (1.5, 1.5)--(4.5, 1.5)--(4.5, 6.5); 
\filldraw[color=blue] (1.5, 1.5) circle (6pt)
(4.5, 6.5) circle (6pt);
\draw[thick, color=teal] (.5, 2.5)--(5.5, 2.5)--(5.5, 5.5);
\filldraw[color=teal] (.5, 2.5) circle (6pt)
(5.5, 5.5) circle (6pt);
\node [] at (3.5, 5.1) {\( a \)};
\node [] at (4.5, 7.4) {\small\( b_1 \)};
\node [] at (5.5, 6.3) {\small\( b_2 \)};
\node [] at (.5, 3.2) {\small\( b_2 -1\)};
\node [] at (1.5, .7) {\small\( b_1 -1\)};
\node[] at (4, -1) {\( \fc_{\gamma_2}(\theta(P)) \)};
\node[] at (5, 1.5) {\small \( i \)};
\node[] at (6.4, 2.3) {\small\( i+1 \)};
\node[] at (8, 4.5) {\small \( j \)};
\end{tikzpicture}
\caption{If \( a\ge b_1, b_2 \), we have \( \fc_{\gamma_0}(P) = \fc_{\gamma_1}(P)
  = \fc_{\gamma_2}(\theta(P)) \).}
  \label{fig:B01case}
\end{figure}

\begin{figure}[ht]
  \centering
    \begin{tikzpicture}[scale=.45]
\foreach \i in {0,...,4}
\draw[dotted] (0,\i)--(3+\i,\i);
\foreach \i in {0,...,3}
\draw[dotted] (\i, 0)--(\i, 7);
\foreach \i in {0,...,4}
\draw[dotted] (4+\i, \i+1)--(4+\i,7);
\draw[dotted] (0,5)--(8,5);
\draw[dotted] (0,6)--(8,6);
\draw[dotted] (0,7)--(8,7)
(3,0)--(8,5);
\draw[very thick] (2,1)--(2,3)
(4,4)--(6,4)--(6,5);
\filldraw (3.5, 4.5) circle (6pt);
\draw[thick, color=blue] (1.5, 1.5)--(4.5, 1.5)--(4.5, 6.5); 
\filldraw[color=blue] (1.5, 1.5) circle (6pt)
(4.5, 6.5) circle (6pt);
\draw[thick, color=teal] (.5, 2.5)--(5.5, 2.5)--(5.5, 5.5);
\filldraw[color=teal] (.5, 2.5) circle (6pt)
(5.5, 5.5) circle (6pt);
\draw [very thick, color=red, fill=white] (5.5, 4.5) circle (8pt)
(4.5, 4.5) circle (8pt);
\node [] at (3.5, 5.1) {\( a \)};
\node [] at (4.5, 7.4) {\small\( b_1 \)};
\node [] at (5.5, 6.3) {\small\( b_2 \)};
\node [] at (.5, 3.2) {\small\( b_2 -1\)};
\node [] at (1.5, .7) {\small\( b_1 -1\)};
\node[] at (4, -1) {\( \fc_{\gamma_0}(P) \)};
\node[] at (5, 1.5) {\small \( i \)};
\node[] at (6.4, 2.3) {\small\( i+1 \)};
\node[] at (8, 4.5) {\small \( j \)};
\end{tikzpicture}\quad
    \begin{tikzpicture}[scale=.45]
\foreach \i in {0,...,4}
\draw[dotted] (0,\i)--(3+\i,\i);
\foreach \i in {0,...,3}
\draw[dotted] (\i, 0)--(\i, 7);
\foreach \i in {0,...,4}
\draw[dotted] (4+\i, \i+1)--(4+\i,7);
\draw[dotted] (0,5)--(8,5);
\draw[dotted] (0,6)--(8,6);
\draw[dotted] (0,7)--(8,7)
(3,0)--(8,5);
\draw[very thick] (2,1)--(2,3)
(4,4)--(5,4)--(5,5)--(6,5);
\filldraw (3.5, 4.5) circle (6pt);
\draw[thick, color=blue] (1.5, 1.5)--(4.5, 1.5)--(4.5, 6.5); 
\filldraw[color=blue] (1.5, 1.5) circle (6pt)
(4.5, 6.5) circle (6pt);
\draw[thick, color=teal] (.5, 2.5)--(5.5, 2.5)--(5.5, 5.5);
\filldraw[color=teal] (.5, 2.5) circle (6pt)
(5.5, 5.5) circle (6pt);
\draw [very thick, color=red, fill=white] (4.5, 4.5) circle (8pt);
\node [] at (3.5, 5.1) {\( a \)};
\node [] at (4.5, 7.4) {\small\( b_1 \)};
\node [] at (5.5, 6.3) {\small\( b_2 \)};
\node [] at (.5, 3.2) {\small\( b_2 -1\)};
\node [] at (1.5, .7) {\small\( b_1 -1\)};
\node[] at (4, -1) {\( \fc_{\gamma_1}(P) \)};
\node[] at (5, 1.5) {\small \( i \)};
\node[] at (6.4, 2.3) {\small\( i+1 \)};
\node[] at (8, 4.5) {\small \( j \)};
\end{tikzpicture}\quad
    \begin{tikzpicture}[scale=.45]
\foreach \i in {0,...,4}
\draw[dotted] (0,\i)--(3+\i,\i);
\foreach \i in {0,...,3}
\draw[dotted] (\i, 0)--(\i, 7);
\foreach \i in {0,...,4}
\draw[dotted] (4+\i, \i+1)--(4+\i,7);
\draw[dotted] (0,5)--(8,5);
\draw[dotted] (0,6)--(8,6);
\draw[dotted] (0,7)--(8,7)
(3,0)--(8,5);
\draw[very thick] (2,1)--(2,3)
(4,4)--(4,5)--(6,5);
\filldraw (3.5, 4.5) circle (6pt);
\draw[thick, color=blue] (1.5, 1.5)--(4.5, 1.5)--(4.5, 6.5); 
\filldraw[color=blue] (1.5, 1.5) circle (6pt)
(4.5, 6.5) circle (6pt);
\draw[thick, color=teal] (.5, 2.5)--(5.5, 2.5)--(5.5, 5.5);
\filldraw[color=teal] (.5, 2.5) circle (6pt)
(5.5, 5.5) circle (6pt);
\node [] at (3.5, 5.1) {\( a \)};
\node [] at (4.5, 7.4) {\small\( b_1 \)};
\node [] at (5.5, 6.3) {\small\( b_2 \)};
\node [] at (.5, 3.2) {\small\( b_2 -1\)};
\node [] at (1.5, .7) {\small\( b_1 -1\)};
\node[] at (4, -1) {\( \fc_{\gamma_2}(\theta(P)) \)};
\node[] at (5, 1.5) {\small \( i \)};
\node[] at (6.4, 2.3) {\small\( i+1 \)};
\node[] at (8, 4.5) {\small \( j \)};
\end{tikzpicture}
\caption{If \( a< b_1, b_2 \), we have \( \fc_{\gamma_0}(P)-1 = \fc_{\gamma_1}(P)
  = \fc_{\gamma_2}(\theta(P))+1 \).}
  \label{fig:B02case}
\end{figure}

\begin{figure}
  \centering
    \begin{tikzpicture}[scale=.45]
\foreach \i in {0,...,4}
\draw[dotted] (0,\i)--(3+\i,\i);
\foreach \i in {0,...,3}
\draw[dotted] (\i, 0)--(\i, 7);
\foreach \i in {0,...,4}
\draw[dotted] (4+\i, \i+1)--(4+\i,7);
\draw[dotted] (0,5)--(8,5);
\draw[dotted] (0,6)--(8,6);
\draw[dotted] (0,7)--(8,7)
(3,0)--(8,5);
\draw[very thick] (2,1)--(2,3)
(4,4)--(6,4)--(6,5);
\filldraw (3.5, 4.5) circle (6pt);
\draw[thick, color=blue] (1.5, 1.5)--(4.5, 1.5)--(4.5, 6.5); 
\filldraw[color=blue] (1.5, 1.5) circle (6pt)
(4.5, 6.5) circle (6pt);
\draw[thick, color=teal] (.5, 2.5)--(5.5, 2.5)--(5.5, 5.5);
\filldraw[color=teal] (.5, 2.5) circle (6pt)
(5.5, 5.5) circle (6pt);
\draw [very thick, color=red, fill=white] (.5, 1.5) circle (8pt)
(5.5, 4.5) circle (8pt);
\draw[very thick, color=red] (4.2, 4.2)--(4.8, 4.8)
(4.2, 4.8)--(4.8, 4.2);
\node [] at (3.5, 5.1) {\( a \)};
\node [] at (4.5, 7.4) {\small\( b_1 \)};
\node [] at (5.5, 6.3) {\small\( b_2 \)};
\node [] at (.5, 3.2) {\small\( b_2 -1\)};
\node [] at (1.5, .7) {\small\( b_1 -1\)};
\node[] at (4, -1) {\( \fc_{\gamma_0}(P) \)};
\draw[very thick, color=red] (4.2, 5.2)--(4.8, 5.8) (4.2, 5.8)--(4.8, 5.2);
\node[] at (5, 1.5) {\small \( i \)};
\node[] at (6.4, 2.3) {\small\( i+1 \)};
\node[] at (8, 4.5) {\small \( j \)};
\end{tikzpicture}\quad
    \begin{tikzpicture}[scale=.45]
\foreach \i in {0,...,4}
\draw[dotted] (0,\i)--(3+\i,\i);
\foreach \i in {0,...,3}
\draw[dotted] (\i, 0)--(\i, 7);
\foreach \i in {0,...,4}
\draw[dotted] (4+\i, \i+1)--(4+\i,7);
\draw[dotted] (0,5)--(8,5);
\draw[dotted] (0,6)--(8,6);
\draw[dotted] (0,7)--(8,7)
(3,0)--(8,5);
\draw[very thick] (2,1)--(2,3)
(4,4)--(5,4)--(5,5)--(6,5);
\filldraw (3.5, 4.5) circle (6pt);
\draw[thick, color=blue] (1.5, 1.5)--(4.5, 1.5)--(4.5, 6.5); 
\filldraw[color=blue] (1.5, 1.5) circle (6pt)
(4.5, 6.5) circle (6pt);
\draw[thick, color=teal] (.5, 2.5)--(5.5, 2.5)--(5.5, 5.5);
\filldraw[color=teal] (.5, 2.5) circle (6pt)
(5.5, 5.5) circle (6pt);
\draw[very thick, color=red] (4.2, 4.2)--(4.8, 4.8)
(4.2, 4.8)--(4.8, 4.2);
\draw [very thick, color=red, fill=white] (.5, 1.5) circle (8pt);
\node [] at (3.5, 5.1) {\( a \)};
\node [] at (4.5, 7.4) {\small\( b_1 \)};
\node [] at (5.5, 6.3) {\small\( b_2 \)};
\node [] at (.5, 3.2) {\small\( b_2 -1\)};
\node [] at (1.5, .7) {\small\( b_1 -1\)};
\node[] at (4, -1) {\( \fc_{\gamma_1}(P) \)};
\draw[very thick, color=red] (4.2, 5.2)--(4.8, 5.8) (4.2, 5.8)--(4.8, 5.2);
\node[] at (5, 1.5) {\small \( i \)};
\node[] at (6.4, 2.3) {\small\( i+1 \)};
\node[] at (8, 4.5) {\small \( j \)};
\end{tikzpicture}\quad
    \begin{tikzpicture}[scale=.45]
\foreach \i in {0,...,4}
\draw[dotted] (0,\i)--(3+\i,\i);
\foreach \i in {0,...,3}
\draw[dotted] (\i, 0)--(\i, 7);
\foreach \i in {0,...,4}
\draw[dotted] (4+\i, \i+1)--(4+\i,7);
\draw[dotted] (0,5)--(8,5);
\draw[dotted] (0,6)--(8,6);
\draw[dotted] (0,7)--(8,7)
(3,0)--(8,5);
\draw[very thick] (2,1)--(2,3)
(4,4)--(4,5)--(6,5);
\filldraw (3.5, 4.5) circle (6pt);
\draw[thick, color=blue] (1.5, 2.5)--(5.5, 2.5)--(5.5, 6.5); 
\filldraw[color=blue] (1.5, 2.5) circle (6pt)
(5.5, 6.5) circle (6pt);
\draw[thick, color=teal] (.5, 1.5)--(4.5, 1.5)--(4.5, 5.5);
\filldraw[color=teal] (.5, 1.5) circle (6pt)
(4.5, 5.5) circle (6pt);
\draw[very thick, color=red] (1.2, 1.2)--(1.8, 1.8)
(1.2, 1.8)--(1.8, 1.2);
\node [] at (3.5, 5.1) {\( a \)};
\node [] at (4.5, 6.4) {\small\( b_2 \)};
\node [] at (5.5, 7.4) {\small\( b_1 \)};
\node [] at (.5, .8) {\small\( b_2 -1\)};
\node [] at (1.5,3.2) {\small\( b_1 -1\)};
\node[] at (4, -1) {\( \fc_{\gamma_2}(\theta(P)) \)};
\draw[very thick, color=red] (5.2, 5.2)--(5.8, 5.8) (5.2, 5.8)--(5.8, 5.2);
\node[] at (5, 1.5) {\small \( i \)};
\node[] at (6.4, 2.3) {\small\( i+1 \)};
\node[] at (8, 4.5) {\small \( j \)};
\end{tikzpicture}
  \caption{If \( b_1\le a <b_2 \), we have \( \fc_{\gamma_0}(P)-1 = \fc_{\gamma_1}(P)
  = \fc_{\gamma_2}(\theta(P))+1 \).}
  \label{fig:B03case}
\end{figure}

\begin{figure}
  \centering
    \begin{tikzpicture}[scale=.45]
\foreach \i in {0,...,4}
\draw[dotted] (0,\i)--(3+\i,\i);
\foreach \i in {0,...,3}
\draw[dotted] (\i, 0)--(\i, 7);
\foreach \i in {0,...,4}
\draw[dotted] (4+\i, \i+1)--(4+\i,7);
\draw[dotted] (0,5)--(8,5);
\draw[dotted] (0,6)--(8,6);
\draw[dotted] (0,7)--(8,7)
(3,0)--(8,5);
\draw[very thick] (2,1)--(2,3)
(4,4)--(6,4)--(6,5);
\filldraw (3.5, 4.5) circle (6pt);
\draw[thick, color=blue] (1.5, 1.5)--(4.5, 1.5)--(4.5, 6.5); 
\filldraw[color=blue] (1.5, 1.5) circle (6pt)
(4.5, 6.5) circle (6pt);
\draw[thick, color=teal] (.5, 2.5)--(5.5, 2.5)--(5.5, 5.5);
\filldraw[color=teal] (.5, 2.5) circle (6pt)
(5.5, 5.5) circle (6pt);
\draw [very thick, color=red, fill=white] (4.5, 4.5) circle (8pt);
\draw[very thick, color=red] (5.2, 4.2)--(5.8, 4.8)
(5.2, 4.8)--(5.8, 4.2);
\draw[very thick, color=red] (0.2, 1.2)--(0.8, 1.8)
(0.2, 1.8)--(0.8, 1.2);
\node [] at (3.5, 5.1) {\( a \)};
\node [] at (4.5, 7.4) {\small\( b_1 \)};
\node [] at (5.5, 6.3) {\small\( b_2 \)};
\node [] at (.5, 3.1) {\small\( b_2 -1\)};
\node [] at (1.5, .8) {\small\( b_1 -1\)};
\node[] at (4, -1) {\( \fc_{\gamma_0}(P) \)};
\draw[very thick, color=red, fill=white] (4.5, 5.5) circle (8pt);
\node[] at (5, 1.5) {\small \( i \)};
\node[] at (6.4, 2.3) {\small\( i+1 \)};
\node[] at (8, 4.5) {\small \( j \)};
\end{tikzpicture}\quad
    \begin{tikzpicture}[scale=.45]
\foreach \i in {0,...,4}
\draw[dotted] (0,\i)--(3+\i,\i);
\foreach \i in {0,...,3}
\draw[dotted] (\i, 0)--(\i, 7);
\foreach \i in {0,...,4}
\draw[dotted] (4+\i, \i+1)--(4+\i,7);
\draw[dotted] (0,5)--(8,5);
\draw[dotted] (0,6)--(8,6);
\draw[dotted] (0,7)--(8,7)
(3,0)--(8,5);
\draw[very thick] (2,1)--(2,3)
(4,4)--(5,4)--(5,5)--(6,5);
\filldraw (3.5, 4.5) circle (6pt);
\draw[thick, color=blue] (1.5, 1.5)--(4.5, 1.5)--(4.5, 6.5); 
\filldraw[color=blue] (1.5, 1.5) circle (6pt)
(4.5, 6.5) circle (6pt);
\draw[thick, color=teal] (.5, 2.5)--(5.5, 2.5)--(5.5, 5.5);
\filldraw[color=teal] (.5, 2.5) circle (6pt)
(5.5, 5.5) circle (6pt);
\draw [very thick, color=red, fill=white] (4.5, 4.5) circle (8pt);
\draw[very thick, color=red] (0.2, 1.2)--(0.8, 1.8)
(0.2, 1.8)--(0.8, 1.2);
\node [] at (3.5, 5.1) {\( a \)};
\node [] at (4.5, 7.4) {\small\( b_1 \)};
\node [] at (5.5, 6.3) {\small\( b_2 \)};
\node [] at (.5, 3.1) {\small\( b_2 -1\)};
\node [] at (1.5, .8) {\small\( b_1 -1\)};
\node[] at (4, -1) {\( \fc_{\gamma_1}(P) \)};
\draw[very thick, color=red, fill=white] (4.5, 5.5) circle (8pt);
\node[] at (5, 1.5) {\small \( i \)};
\node[] at (6.4, 2.3) {\small\( i+1 \)};
\node[] at (8, 4.5) {\small \( j \)};
\end{tikzpicture}\quad
    \begin{tikzpicture}[scale=.45]
\foreach \i in {0,...,4}
\draw[dotted] (0,\i)--(3+\i,\i);
\foreach \i in {0,...,3}
\draw[dotted] (\i, 0)--(\i, 7);
\foreach \i in {0,...,4}
\draw[dotted] (4+\i, \i+1)--(4+\i,7);
\draw[dotted] (0,5)--(8,5);
\draw[dotted] (0,6)--(8,6);
\draw[dotted] (0,7)--(8,7)
(3,0)--(8,5);
\draw[very thick] (2,1)--(2,3)
(4,4)--(4,5)--(6,5);
\filldraw (3.5, 4.5) circle (6pt);
\draw[thick, color=blue] (1.5, 2.5)--(5.5, 2.5)--(5.5, 6.5); 
\filldraw[color=blue] (1.5, 2.5) circle (6pt)
(5.5, 6.5) circle (6pt);
\draw[thick, color=teal] (.5, 1.5)--(4.5, 1.5)--(4.5, 5.5);
\filldraw[color=teal] (.5, 1.5) circle (6pt)
(4.5, 5.5) circle (6pt);
\draw [very thick, color=red, fill=white] (1.5, 1.5) circle (8pt);
\node [] at (3.5, 5.1) {\( a \)};
\node [] at (4.5, 6.4) {\small\( b_2 \)};
\node [] at (5.5, 7.4) {\small\( b_1 \)};
\node [] at (.5, .8) {\small\( b_2 -1\)};
\node [] at (1.5,3.2) {\small\( b_1 -1\)};
\node[] at (4, -1) {\( \fc_{\gamma_2}(\theta(P)) \)};
\draw[very thick, color=red, fill=white] (5.5, 5.5) circle (8pt);
\node[] at (5, 1.5) {\small \( i \)};
\node[] at (6.4, 2.3) {\small\( i+1 \)};
\node[] at (8, 4.5) {\small \( j \)};
\end{tikzpicture}
\caption{If \( b_2\le a <b_1 \), we have \( \fc_{\gamma_0}(P) = \fc_{\gamma_1}(P)
  =\fc_{\gamma_2}(\theta(P)) \).}
  \label{fig:B04case}
\end{figure}

Hence, we obtain \eqref{eq:1}, which completes the proof.
\end{proof}

Similarly, we prove the modular law for the other modular triples.

\begin{lem}\label{lem:ML-2}
  Let \( (\gamma_0,\gamma_1,\gamma_2)\in \MM_n^{1,i} \) and
  \( \mu\vdash n \). Then \eqref{eq:3} holds.
\end{lem}

\begin{proof}
  Let \( j=\gamma_1(i)+1 \). Note that
  \( \LRP(\gamma_2,\mu) \subseteq \LRP(\gamma_1,\mu) \subseteq
  \LRP(\gamma_0,\mu) \). For \( \gamma\in \D_n \), we define
\begin{align*}
  A_0(\gamma)
  &= \{P\in \LRP(\gamma,\mu): \mbox{\( P \) has a rook in neither \( (i,j) \) nor \( (i,j-1) \)} \},\\
  A_1(\gamma)
  &= \{P\in \LRP(\gamma,\mu): \mbox{\( P \) has a rook in \( (i,j) \)} \},\\
  A_2(\gamma)
  &= \{P\in \LRP(\gamma,\mu): \mbox{\( P \) has a rook in \( (i,j-1) \)} \}.
\end{align*}
We claim that 
\begin{equation}\label{eqn:ML2id1}
  (1+q) \sum_{P\in A_0(\gamma_1)} q^{\fc_{\gamma_1}(P)}
  = \sum_{P\in A_0(\gamma_0)} q^{\fc_{\gamma_0}(P)}
  + q \sum_{P\in A_0(\gamma_2)} q^{\fc_{\gamma_2}(P)}
\end{equation}
and 
\begin{equation}\label{eqn:ML2id2}
  (1+q) \sum_{P\in A_1(\gamma_1)} q^{\fc_{\gamma_1}(P)}
  = \sum_{P\in A_1(\gamma_0)} q^{\fc_{\gamma_0}(P)}
  + \sum_{P\in A_2(\gamma_0)} q^{\fc_{\gamma_0}(P)},
\end{equation}
which imply \eqref{eq:3}.

To prove \eqref{eqn:ML2id2}, it suffices to construct a
bijection \( \varphi:A_1(\gamma_1) \to A_2(\gamma_0) \) such that
\begin{equation}\label{eq:20-2}
  q^{\fc_{\gamma_0}(P)} + q^{\fc_{\gamma_0}(\varphi(P))} = q^{\fc_{\gamma_1}(P)} + q^{\fc_{\gamma_1}(P)+1},
\end{equation}
because by the fact that \( A_1(\gamma_1) = A_1(\gamma_0) \), summing
\eqref{eq:20-2} over all \( P\in A_1(\gamma_1) \) gives
\eqref{eqn:ML2id2}. Given \( P\in A_1(\gamma_1) \), let \( a=r_j(P) \)
and \( b=r_{j-1}(P) \). Then we define
\[
 \varphi(P) =
  \begin{cases}
    (R_{j-1,j}\circ C_{j-1,j})(P) & \mbox{if \( a\ne b \)},\\
    R_{j-1,j}(P) & \mbox{if \( a= b \)}.
  \end{cases}
\]
We can prove \eqref{eq:20-2} by considering the possible cases:
\begin{itemize}
\item If \( a>b \), by \Cref{fig:a>b2}, we have
  \( \fc_{\gamma_0}(P) = \fc_{\gamma_1}(P)+1 \) and
  \( \fc_{\gamma_0}(\varphi(P)) = \fc_{\gamma_1}(P) \).
\item If \( a<b \), by \Cref{fig:a<b2}, we have
  \( \fc_{\gamma_0}(P) = \fc_{\gamma_1}(P) \) and
  \( \fc_{\gamma_0}(\varphi(P)) = \fc_{\gamma_1}(P)+1 \).
\item If \( a=b \), by \Cref{fig:a=b2}, we have
  \( \fc_{\gamma_0}(P) = \fc_{\gamma_1}(P) \) and
  \( \fc_{\gamma_0}(\varphi(P)) = \fc_{\gamma_1}(P)+1 \). 
\end{itemize}

\begin{figure}[ht]
  \centering
    \begin{tikzpicture}[scale=.4]
\foreach \i in {3,...,9}
\draw[dotted] (0,\i)--(\i,\i);
\foreach \i in {2,...,9}
\draw[dotted] (\i,9)--(\i,\i);
\draw[dotted] (0,2)--(0,9)
(1,2)--(1,9);
\draw[dotted] (0,2)--(2,2)--(9,9);
\draw[very thick] (4,7)--(6,7)
(1,4)--(2,4)--(2,6);
\draw[thick] (1.5,5.5)--(5.5, 5.5)--(5.5, 8.5);
\draw[thick, color=blue] (.5, 4.5)--(4.5, 4.5)--(4.5, 7.5);
\filldraw (1.5, 5.5) circle (6pt)
(5.5, 8.5) circle (6pt);
\filldraw[color=blue]
(.5, 4.5) circle (6pt)
(4.5, 7.5) circle (6pt);
\draw [very thick, color=red, fill=white] (1.5, 4.5) circle (8pt);
\node[] at (4.3, 8.3) {\small \( b+1 \)};
\node[] at (5.5, 9.4) {\small \( a+1 \)};
\node[] at (.5, 5.2) {\small \( b \)};
\node[] at (1.5, 6.2) {\small \( a \)};
\node[] at (1.5, 9.5) {\small \( i \)};
\node[] at (-.5, 5.5) {\small \( j \)};
\node[] at (-.9, 4.5) {\small \( j-1 \)};
\node[] at (3.5, 1.2) {\( \fc_{\gamma_0}(P) \)};
\draw[very thick, color=red, fill=white] (5.5, 7.5) circle (8pt);
\end{tikzpicture}
    \begin{tikzpicture}[scale=.4]
\foreach \i in {3,...,9}
\draw[dotted] (0,\i)--(\i,\i);
\foreach \i in {2,...,9}
\draw[dotted] (\i,9)--(\i,\i);
\draw[dotted] (0,2)--(0,9)
(1,2)--(1,9);
\draw[dotted] (0,2)--(2,2)--(9,9);
\draw[very thick] (4,7)--(6,7)
(1,4)--(1,5)--(2,5)--(2,6);
\draw[thick] (1.5,5.5)--(5.5, 5.5)--(5.5, 8.5);
\draw[thick, color=blue] (.5, 4.5)--(4.5, 4.5)--(4.5, 7.5);
\filldraw (1.5, 5.5) circle (6pt)
(5.5, 8.5) circle (6pt);
\filldraw[color=blue]
(.5, 4.5) circle (6pt)
(4.5, 7.5) circle (6pt);
\node[] at (4.3, 8.3) {\small \( b+1 \)};
\node[] at (5.5, 9.4) {\small \( a+1 \)};
\node[] at (.5, 5.2) {\small \( b \)};
\node[] at (1.5, 6.2) {\small \( a \)};
\node[] at (1.5, 9.5) {\small \( i \)};
\node[] at (-.5, 5.5) {\small \( j \)};
\node[] at (-.9, 4.5) {\small \( j-1 \)};
\node[] at (3.5, 1.2) {\( \fc_{\gamma_1}(P) \)};
\draw[very thick, color=red, fill=white] (5.5, 7.5) circle (8pt);
\end{tikzpicture}
    \begin{tikzpicture}[scale=.4]
\foreach \i in {3,...,9}
\draw[dotted] (0,\i)--(\i,\i);
\foreach \i in {2,...,9}
\draw[dotted] (\i,9)--(\i,\i);
\draw[dotted] (0,2)--(0,9)
(1,2)--(1,9);
\draw[dotted] (0,2)--(2,2)--(9,9);
\draw[very thick] (4,7)--(6,7)
(1,4)--(2,4)--(2,6);
\draw[thick] (1.5,4.5)--(4.5, 4.5)--(4.5, 8.5);
\draw[thick, color=blue] (.5, 5.5)--(5.5, 5.5)--(5.5, 7.5);
\filldraw (1.5, 4.5) circle (6pt)
(4.5, 8.5) circle (6pt);
\filldraw[color=blue]
(.5, 5.5) circle (6pt)
(5.5, 7.5) circle (6pt);
\draw[very thick, color=red] (0.2, 4.2)--(0.8, 4.8)
(0.2, 4.8)--(0.8, 4.2);
\node[] at (4.5, 9.4) {\small \( a+1 \)};
\node[] at (5.9, 8.3) {\small \( b+1 \)};
\node[] at (.5, 6.2) {\small \( b \)};
\node[] at (1.5, 5.2) {\small \( a \)};
\node[] at (1.5, 9.5) {\small \( i \)};
\node[] at (-.5, 5.5) {\small \( j \)};
\node[] at (-.9, 4.5) {\small \( j-1 \)};
\node[] at (3.5, 1.2) {\( \fc_{\gamma_0}(\varphi(P)) \)};
\draw[very thick, color=red, fill=white] (4.5, 7.5) circle (8pt);
\end{tikzpicture}
  \caption{If \( a>b \), then \( \fc_{\gamma_0}(P)-1 = \fc_{\gamma_1}(P) = \fc_{\gamma_0}(\varphi(P)) \).}
  \label{fig:a>b2}
\end{figure}

\begin{figure}[ht]
  \centering
    \begin{tikzpicture}[scale=.4]
\foreach \i in {3,...,9}
\draw[dotted] (0,\i)--(\i,\i);
\foreach \i in {2,...,9}
\draw[dotted] (\i,9)--(\i,\i);
\draw[dotted] (0,2)--(0,9)
(1,2)--(1,9);
\draw[dotted] (0,2)--(2,2)--(9,9);
\draw[very thick] (4,7)--(6,7)
(1,4)--(2,4)--(2,6);
\draw[thick] (1.5,5.5)--(5.5, 5.5)--(5.5, 8.5);
\draw[thick, color=blue] (.5, 4.5)--(4.5, 4.5)--(4.5, 7.5);
\filldraw (1.5, 5.5) circle (6pt)
(5.5, 8.5) circle (6pt);
\filldraw[color=blue]
(.5, 4.5) circle (6pt)
(4.5, 7.5) circle (6pt);
\draw[very thick, color=red] (1.2, 4.2)--(1.8, 4.8)
(1.2, 4.8)--(1.8, 4.2);
\node[] at (4.3, 8.3) {\small \( b+1 \)};
\node[] at (5.5, 9.4) {\small \( a+1 \)};
\node[] at (.5, 5.2) {\small \( b \)};
\node[] at (1.5, 6.2) {\small \( a \)};
\node[] at (1.5, 9.5) {\small \( i \)};
\node[] at (-.5, 5.5) {\small \( j \)};
\node[] at (-.9, 4.5) {\small \( j-1 \)};
\node[] at (3.5, 1.2) {\( \fc_{\gamma_0}(P) \)};
\draw[very thick, color=red] (5.2, 7.2)--(5.8, 7.8) (5.2, 7.8)--(5.8, 7.2);
\end{tikzpicture}
    \begin{tikzpicture}[scale=.4]
\foreach \i in {3,...,9}
\draw[dotted] (0,\i)--(\i,\i);
\foreach \i in {2,...,9}
\draw[dotted] (\i,9)--(\i,\i);
\draw[dotted] (0,2)--(0,9)
(1,2)--(1,9);
\draw[dotted] (0,2)--(2,2)--(9,9);
\draw[very thick] (4,7)--(6,7)
(1,4)--(1,5)--(2,5)--(2,6);
\draw[thick] (1.5,5.5)--(5.5, 5.5)--(5.5, 8.5);
\draw[thick, color=blue] (.5, 4.5)--(4.5, 4.5)--(4.5, 7.5);
\filldraw (1.5, 5.5) circle (6pt)
(5.5, 8.5) circle (6pt);
\filldraw[color=blue]
(.5, 4.5) circle (6pt)
(4.5, 7.5) circle (6pt);
\node[] at (4.3, 8.3) {\small \( b+1 \)};
\node[] at (5.5, 9.4) {\small \( a+1 \)};
\node[] at (.5, 5.2) {\small \( b \)};
\node[] at (1.5, 6.2) {\small \( a \)};
\node[] at (1.5, 9.5) {\small \( i \)};
\node[] at (-.5, 5.5) {\small \( j \)};
\node[] at (-.9, 4.5) {\small \( j-1 \)};
\node[] at (3.5, 1.2) {\( \fc_{\gamma_1}(P) \)};
\draw[very thick, color=red] (5.2, 7.2)--(5.8, 7.8) (5.2, 7.8)--(5.8, 7.2);
\end{tikzpicture}
    \begin{tikzpicture}[scale=.4]
\foreach \i in {3,...,9}
\draw[dotted] (0,\i)--(\i,\i);
\foreach \i in {2,...,9}
\draw[dotted] (\i,9)--(\i,\i);
\draw[dotted] (0,2)--(0,9)
(1,2)--(1,9);
\draw[dotted] (0,2)--(2,2)--(9,9);
\draw[very thick] (4,7)--(6,7)
(1,4)--(2,4)--(2,6);
\draw[thick] (1.5,4.5)--(4.5, 4.5)--(4.5, 8.5);
\draw[thick, color=blue] (.5, 5.5)--(5.5, 5.5)--(5.5, 7.5);
\filldraw (1.5, 4.5) circle (6pt)
(4.5, 8.5) circle (6pt);
\filldraw[color=blue]
(.5, 5.5) circle (6pt)
(5.5, 7.5) circle (6pt);
\draw [very thick, color=red, fill=white] (0.5, 4.5) circle (8pt);
\node[] at (4.5, 9.4) {\small \( a+1 \)};
\node[] at (5.8, 8.3) {\small \( b+1 \)};
\node[] at (.5, 6.2) {\small \( b \)};
\node[] at (1.5, 5.2) {\small \( a \)};
\node[] at (1.5, 9.5) {\small \( i \)};
\node[] at (-.5, 5.5) {\small \( j \)};
\node[] at (-.9, 4.5) {\small \( j-1 \)};
\node[] at (3.5, 1.2) {\( \fc_{\gamma_0}(\varphi(P)) \)};
\draw[very thick, color=red] (4.2, 7.2)--(4.8, 7.8) (4.2, 7.8)--(4.8, 7.2);
\end{tikzpicture}
  \caption{If \( a<b \), then \( \fc_{\gamma_0}(P) = \fc_{\gamma_1}(P) = \fc_{\gamma_0}(\varphi(P))-1 \).}
  \label{fig:a<b2}
\end{figure}
\begin{figure}
  \centering
    \begin{tikzpicture}[scale=.4]
\foreach \i in {3,...,9}
\draw[dotted] (0,\i)--(\i,\i);
\foreach \i in {2,...,9}
\draw[dotted] (\i,9)--(\i,\i);
\draw[dotted] (0,2)--(0,9)
(1,2)--(1,9);
\draw[dotted] (0,2)--(2,2)--(9,9);
\draw[very thick] (4,7)--(6,7)
(1,4)--(2,4)--(2,6);
\draw[thick] (1.5,5.5)--(5.5, 5.5)--(5.5, 8.5);
\draw[thick, color=blue] (.5, 4.5)--(4.5, 4.5)--(4.5, 7.5);
\filldraw (1.5, 5.5) circle (6pt)
(5.5, 8.5) circle (6pt);
\filldraw[color=blue]
(.5, 4.5) circle (6pt)
(4.5, 7.5) circle (6pt);
\draw[very thick, color=red] (1.2, 4.2)--(1.8, 4.8)
(1.2, 4.8)--(1.8, 4.2);
\node[] at (4.3, 8.3) {\small \( a+1 \)};
\node[] at (5.5, 9.4) {\small \( a+1 \)};
\node[] at (.5, 5.2) {\small \( a \)};
\node[] at (1.5, 6.2) {\small \( a \)};
\node[] at (1.5, 9.5) {\small \( i \)};
\node[] at (-.5, 5.5) {\small \( j \)};
\node[] at (-.9, 4.5) {\small \( j-1 \)};
\node[] at (3.5, 1.2) {\( \fc_{\gamma_0}(P) \)};
\end{tikzpicture}
    \begin{tikzpicture}[scale=.4]
\foreach \i in {3,...,9}
\draw[dotted] (0,\i)--(\i,\i);
\foreach \i in {2,...,9}
\draw[dotted] (\i,9)--(\i,\i);
\draw[dotted] (0,2)--(0,9)
(1,2)--(1,9);
\draw[dotted] (0,2)--(2,2)--(9,9);
\draw[very thick] (4,7)--(6,7)
(1,4)--(1,5)--(2,5)--(2,6);
\draw[thick] (1.5,5.5)--(5.5, 5.5)--(5.5, 8.5);
\draw[thick, color=blue] (.5, 4.5)--(4.5, 4.5)--(4.5, 7.5);
\filldraw (1.5, 5.5) circle (6pt)
(5.5, 8.5) circle (6pt);
\filldraw[color=blue]
(.5, 4.5) circle (6pt)
(4.5, 7.5) circle (6pt);
\node[] at (4.3, 8.3) {\small \( a+1 \)};
\node[] at (5.5, 9.4) {\small \( a+1 \)};
\node[] at (.5, 5.2) {\small \( a \)};
\node[] at (1.5, 6.2) {\small \( a \)};
\node[] at (1.5, 9.5) {\small \( i \)};
\node[] at (-.5, 5.5) {\small \( j \)};
\node[] at (-.9, 4.5) {\small \( j-1 \)};
\node[] at (3.5, 1.2) {\( \fc_{\gamma_1}(P) \)};
\end{tikzpicture}
    \begin{tikzpicture}[scale=.4]
\foreach \i in {3,...,9}
\draw[dotted] (0,\i)--(\i,\i);
\foreach \i in {2,...,9}
\draw[dotted] (\i,9)--(\i,\i);
\draw[dotted] (0,2)--(0,9)
(1,2)--(1,9);
\draw[dotted] (0,2)--(2,2)--(9,9);
\draw[very thick] (4,7)--(6,7)
(1,4)--(2,4)--(2,6);
\draw[thick] (1.5,4.5)--(4.5, 4.5)
(4.5,4.5)--(4.5, 7.5);
\draw[thick, color=blue] (.5, 5.5)--(5.5, 5.5)
(5.5, 5.5)--(5.5, 8.5);
\filldraw (1.5, 4.5) circle (6pt)
(4.5, 7.5) circle (6pt);
\filldraw[color=blue]
(.5, 5.5) circle (6pt)
(5.5, 8.5) circle (6pt);
\draw [very thick, color=red, fill=white] (0.5, 4.5) circle (8pt);
\node[] at (4.3, 8.3) {\small \( a+1 \)};
\node[] at (5.6, 9.4) {\small \( a+1 \)};
\node[] at (.5, 6.2) {\small \( a \)};
\node[] at (1.5, 5.2) {\small \( a \)};
\node[] at (1.5, 9.5) {\small \( i \)};
\node[] at (-.5, 5.5) {\small \( j \)};
\node[] at (-.9, 4.5) {\small \( j-1 \)};
\node[] at (3.5, 1.2) {\( \fc_{\gamma_0}(\varphi(P)) \)};
\end{tikzpicture}
  \caption{If \( a=b \), then \( \fc_{\gamma_0}(P) = \fc_{\gamma_1}(P) = \fc_{\gamma_0}(\varphi(P))-1 \).}
  \label{fig:a=b2}
\end{figure}

Now, we prove \eqref{eqn:ML2id1}. Note that
\( A_0(\gamma_0) = A_0 (\gamma_1) = A_0 (\gamma_2) \). Given
\(P\in A_0 (\gamma_2) \), let \( a=c_i(P) \), \( b_1=r_{j-1} (P) \), and
\( b_2=r_{j}(P) \). We define a map
\( \vartheta:A_{0}(\gamma_2)\rightarrow A_{0}(\gamma_2) \) by
\[
  \vartheta(P) =
  \begin{cases}
    P, & \text{ if } a> b_1, b_2 \text{ or } a\le b_1, b_2,\\
    (R_{j-1,j}\circ C_{j-1,j})(P), & \text{ if } b_1< a\le b_2 \text{ or } b_2< a\le b_1.
  \end{cases}
\]
It is easy to check that \( \vartheta \) is an involution. Hence, to prove
\eqref{eqn:ML2id1}, it suffices to show that
for each \( P\in A_0 (\gamma_2) \), we have
\[
  q^{\fc_{\gamma_1}(P)} + q^{\fc_{\gamma_1}(P)+1} = q^{\fc_{\gamma_0}(P)} + q^{\fc_{\gamma_2}(\vartheta(P))+1}.
\]
This can be verified by considering each case:
\begin{itemize}
\item If \( a> b_1, b_2 \), by \Cref{fig:case4-1} we have \( \fc_{\gamma_0}(P) = \fc_{\gamma_1}(P)+1\)
  and \( \fc_{\gamma_2}(\vartheta(P))+1 = \fc_{\gamma_1}(P)\).
\item If \( a\le b_1, b_2 \), by \Cref{fig:case4-2} we have \( \fc_{\gamma_0}(P) = \fc_{\gamma_1}(P)\)
  and \( \fc_{\gamma_2}(\vartheta(P))+1 = \fc_{\gamma_1}(P)+1\).
\item If \( b_1< a\le b_2 \), by \Cref{fig:case4-3} we have \( \fc_{\gamma_0}(P) = \fc_{\gamma_1}(P)+1\)
  and \( \fc_{\gamma_2}(\vartheta(P))+1 = \fc_{\gamma_1}(P)\).
\item If \( b_2< a\le b_1 \), by \Cref{fig:case4-4} we have \( \fc_{\gamma_0}(P) = \fc_{\gamma_1}(P)\)
  and \( \fc_{\gamma_2}(\vartheta(P))+1 = \fc_{\gamma_1}(P)+1\).
  \end{itemize}
  Therefore, we obtain \eqref{eqn:ML2id1}, completing the proof.
  
\begin{figure}
  \centering
    \begin{tikzpicture}[scale=.4]
\foreach \i in {0,...,4}
\draw[dotted] (0,\i)--(4+\i,\i);
\foreach \i in {0,...,4}
\draw[dotted] (\i, 0)--(\i, 7);
\foreach \i in {0,...,3}
\draw[dotted] (5+\i, \i+1)--(5+\i,7);
\draw[dotted] (0,5)--(8,5);
\draw[dotted] (0,6)--(8,6);
\draw[dotted] (0,7)--(8,7)
(4,0)--(8,4);
\draw[very thick] (5,4)--(7,4)
(2,1)--(3,1)--(3,3);
\filldraw (2.5, 4.5) circle (6pt);
\draw[thick, color=blue] (1.5, 1.5)--(5.5, 1.5)--(5.5, 6.5); 
\filldraw[color=blue] (1.5, 1.5) circle (6pt)
(5.5, 6.5) circle (6pt);
\draw[thick, color=teal] (.5, 2.5)--(6.5, 2.5)--(6.5, 5.5);
\filldraw[color=teal] (.5, 2.5) circle (6pt)
(6.5, 5.5) circle (6pt);
\node[] at (2.5, 7.5) {\small \( i \)};
\node[] at (-.5, 2.5) {\small \( j \)};
\node[] at (-1, 1.5) {\small \( j-1 \)};
\node[] at (2.5, 5.2) {\small \( a \)};
\node [] at (5.5, 7.4) {\small\( b_1+1 \)};
\node [] at (6.8, 6.2) {\small\( b_2+1 \)};
\node [] at (.6, 3.3) {\small\( b_2\)};
\node [] at (1.5, .7) {\small\( b_1 \)};
\node[] at (4, -1) {\( \fc_{\gamma_0}(P) \)};
\draw [very thick, color=red, fill=white] (2.5, 2.5) circle (8pt);
\draw [very thick, color=red, fill=white] (2.5, 1.5) circle (8pt);
\end{tikzpicture}
    \begin{tikzpicture}[scale=.4]
\foreach \i in {0,...,4}
\draw[dotted] (0,\i)--(4+\i,\i);
\foreach \i in {0,...,4}
\draw[dotted] (\i, 0)--(\i, 7);
\foreach \i in {0,...,3}
\draw[dotted] (5+\i, \i+1)--(5+\i,7);
\draw[dotted] (0,5)--(8,5);
\draw[dotted] (0,6)--(8,6);
\draw[dotted] (0,7)--(8,7)
(4,0)--(8,4);
\draw[very thick] (5,4)--(7,4)
(2,1)--(2,2)--(3,2)--(3,3);
\filldraw (2.5, 4.5) circle (6pt);
\draw[thick, color=blue] (1.5, 1.5)--(5.5, 1.5)--(5.5, 6.5); 
\filldraw[color=blue] (1.5, 1.5) circle (6pt)
(5.5, 6.5) circle (6pt);
\draw[thick, color=teal] (.5, 2.5)--(6.5, 2.5)--(6.5, 5.5);
\filldraw[color=teal] (.5, 2.5) circle (6pt)
(6.5, 5.5) circle (6pt);
\node[] at (2.5, 7.5) {\small \( i \)};
\node[] at (-.5, 2.5) {\small \( j \)};
\node[] at (-1, 1.5) {\small \( j-1 \)};
\node[] at (2.5, 5.2) {\small \( a \)};
\node [] at (5.5, 7.4) {\small\( b_1+1 \)};
\node [] at (6.8, 6.2) {\small\( b_2+1 \)};
\node [] at (.6, 3.3) {\small\( b_2\)};
\node [] at (1.5, .7) {\small\( b_1 \)};
\node[] at (4, -1) {\( \fc_{\gamma_1}(P) \)};
\draw [very thick, color=red, fill=white] (2.5, 2.5) circle (8pt);
\end{tikzpicture}
    \begin{tikzpicture}[scale=.4]
\foreach \i in {0,...,4}
\draw[dotted] (0,\i)--(4+\i,\i);
\foreach \i in {0,...,4}
\draw[dotted] (\i, 0)--(\i, 7);
\foreach \i in {0,...,3}
\draw[dotted] (5+\i, \i+1)--(5+\i,7);
\draw[dotted] (0,5)--(8,5);
\draw[dotted] (0,6)--(8,6);
\draw[dotted] (0,7)--(8,7)
(4,0)--(8,4);
\draw[very thick] (5,4)--(7,4)
(2,1)--(2,3)--(3,3);
\filldraw (2.5, 4.5) circle (6pt);
\draw[thick, color=blue] (1.5, 1.5)--(5.5, 1.5)--(5.5, 6.5); 
\filldraw[color=blue] (1.5, 1.5) circle (6pt)
(5.5, 6.5) circle (6pt);
\draw[thick, color=teal] (.5, 2.5)--(6.5, 2.5)--(6.5, 5.5);
\filldraw[color=teal] (.5, 2.5) circle (6pt)
(6.5, 5.5) circle (6pt);
\node[] at (2.5, 7.5) {\small \( i \)};
\node[] at (-.5, 2.5) {\small \( j \)};
\node[] at (-1, 1.5) {\small \( j-1 \)};
\node[] at (2.5, 5.2) {\small \( a \)};
\node [] at (5.5, 7.4) {\small\( b_1+1 \)};
\node [] at (6.8, 6.2) {\small\( b_2+1 \)};
\node [] at (.6, 3.3) {\small\( b_2\)};
\node [] at (1.5, .7) {\small\( b_1 \)};
\node[] at (4, -1) {\( \fc_{\gamma_2}(\vartheta(P)) \)};
\end{tikzpicture}
\caption{If \( a>b_1, b_2 \), then \( \fc_{\gamma_0}(P)-1 = \fc_{\gamma_1}(P)=\fc_{\gamma_2}(\vartheta(P))+1 \).}
\label{fig:case4-1}
\end{figure}

\begin{figure}[ht]
  \centering
    \begin{tikzpicture}[scale=.4]
\foreach \i in {0,...,4}
\draw[dotted] (0,\i)--(4+\i,\i);
\foreach \i in {0,...,4}
\draw[dotted] (\i, 0)--(\i, 7);
\foreach \i in {0,...,3}
\draw[dotted] (5+\i, \i+1)--(5+\i,7);
\draw[dotted] (0,5)--(8,5);
\draw[dotted] (0,6)--(8,6);
\draw[dotted] (0,7)--(8,7)
(4,0)--(8,4);
\draw[very thick] (5,4)--(7,4)
(2,1)--(3,1)--(3,3);
\filldraw (2.5, 4.5) circle (6pt);
\draw[thick, color=blue] (1.5, 1.5)--(5.5, 1.5)--(5.5, 6.5); 
\filldraw[color=blue] (1.5, 1.5) circle (6pt)
(5.5, 6.5) circle (6pt);
\draw[thick, color=teal] (.5, 2.5)--(6.5, 2.5)--(6.5, 5.5);
\filldraw[color=teal] (.5, 2.5) circle (6pt)
(6.5, 5.5) circle (6pt);
\node[] at (2.5, 7.5) {\small \( i \)};
\node[] at (-.5, 2.5) {\small \( j \)};
\node[] at (-1, 1.5) {\small \( j-1 \)};
\node[] at (2.5, 5.2) {\small \( a \)};
\node [] at (5.5, 7.4) {\small\( b_1+1 \)};
\node [] at (6.8, 6.2) {\small\( b_2+1 \)};
\node [] at (.6, 3.3) {\small\( b_2\)};
\node [] at (1.5, .7) {\small\( b_1 \)};
\node[] at (4, -1) {\( \fc_{\gamma_0}(P) \)};
\draw[very thick, color=red] (2.2, 1.2)--(2.8, 1.8)
(2.2, 1.8)--(2.8, 1.2);
\draw[very thick, color=red] (2.2, 2.2)--(2.8, 2.8)
(2.2, 2.8)--(2.8, 2.2);
\end{tikzpicture}
    \begin{tikzpicture}[scale=.4]
\foreach \i in {0,...,4}
\draw[dotted] (0,\i)--(4+\i,\i);
\foreach \i in {0,...,4}
\draw[dotted] (\i, 0)--(\i, 7);
\foreach \i in {0,...,3}
\draw[dotted] (5+\i, \i+1)--(5+\i,7);
\draw[dotted] (0,5)--(8,5);
\draw[dotted] (0,6)--(8,6);
\draw[dotted] (0,7)--(8,7)
(4,0)--(8,4);
\draw[very thick] (5,4)--(7,4)
(2,1)--(2,2)--(3,2)--(3,3);
\filldraw (2.5, 4.5) circle (6pt);
\draw[thick, color=blue] (1.5, 1.5)--(5.5, 1.5)--(5.5, 6.5); 
\filldraw[color=blue] (1.5, 1.5) circle (6pt)
(5.5, 6.5) circle (6pt);
\draw[thick, color=teal] (.5, 2.5)--(6.5, 2.5)--(6.5, 5.5);
\filldraw[color=teal] (.5, 2.5) circle (6pt)
(6.5, 5.5) circle (6pt);
\node[] at (2.5, 7.5) {\small \( i \)};
\node[] at (-.5, 2.5) {\small \( j \)};
\node[] at (-1, 1.5) {\small \( j-1 \)};
\node[] at (2.5, 5.2) {\small \( a \)};
\node [] at (5.5, 7.4) {\small\( b_1+1 \)};
\node [] at (6.8, 6.2) {\small\( b_2+1 \)};
\node [] at (.6, 3.3) {\small\( b_2\)};
\node [] at (1.5, .7) {\small\( b_1 \)};
\node[] at (4, -1) {\( \fc_{\gamma_1}(P) \)};
\draw[very thick, color=red] (2.2, 2.2)--(2.8, 2.8)
(2.2, 2.8)--(2.8, 2.2);
\end{tikzpicture}
    \begin{tikzpicture}[scale=.4]
\foreach \i in {0,...,4}
\draw[dotted] (0,\i)--(4+\i,\i);
\foreach \i in {0,...,4}
\draw[dotted] (\i, 0)--(\i, 7);
\foreach \i in {0,...,3}
\draw[dotted] (5+\i, \i+1)--(5+\i,7);
\draw[dotted] (0,5)--(8,5);
\draw[dotted] (0,6)--(8,6);
\draw[dotted] (0,7)--(8,7)
(4,0)--(8,4);
\draw[very thick] (5,4)--(7,4)
(2,1)--(2,3)--(3,3);
\filldraw (2.5, 4.5) circle (6pt);
\draw[thick, color=blue] (1.5, 1.5)--(5.5, 1.5)--(5.5, 6.5); 
\filldraw[color=blue] (1.5, 1.5) circle (6pt)
(5.5, 6.5) circle (6pt);
\draw[thick, color=teal] (.5, 2.5)--(6.5, 2.5)--(6.5, 5.5);
\filldraw[color=teal] (.5, 2.5) circle (6pt)
(6.5, 5.5) circle (6pt);
\node[] at (2.5, 7.5) {\small \( i \)};
\node[] at (-.5, 2.5) {\small \( j \)};
\node[] at (-1, 1.5) {\small \( j-1 \)};
\node[] at (2.5, 5.2) {\small \( a \)};
\node [] at (5.5, 7.4) {\small\( b_1+1 \)};
\node [] at (6.8, 6.2) {\small\( b_2+1 \)};
\node [] at (.6, 3.3) {\small\( b_2\)};
\node [] at (1.5, .7) {\small\( b_1 \)};
\node[] at (4, -1) {\( \fc_{\gamma_2}(\vartheta(P)) \)};
\end{tikzpicture}
\caption{If \( a\le b_1, b_2 \), then \( \fc_{\gamma_0}(P)= \fc_{\gamma_1}(P)=\fc_{\gamma_2}(\vartheta(P)) \).}
\label{fig:case4-2}
\end{figure}
\begin{figure}
  \centering
    \begin{tikzpicture}[scale=.4]
\foreach \i in {0,...,4}
\draw[dotted] (0,\i)--(4+\i,\i);
\foreach \i in {0,...,4}
\draw[dotted] (\i, 0)--(\i, 7);
\foreach \i in {0,...,3}
\draw[dotted] (5+\i, \i+1)--(5+\i,7);
\draw[dotted] (0,5)--(8,5);
\draw[dotted] (0,6)--(8,6);
\draw[dotted] (0,7)--(8,7)
(4,0)--(8,4);
\draw[very thick] (5,4)--(7,4)
(2,1)--(3,1)--(3,3);
\filldraw (2.5, 4.5) circle (6pt);
\draw[thick, color=blue] (1.5, 1.5)--(5.5, 1.5)--(5.5, 6.5); 
\filldraw[color=blue] (1.5, 1.5) circle (6pt)
(5.5, 6.5) circle (6pt);
\draw[thick, color=teal] (.5, 2.5)--(6.5, 2.5)--(6.5, 5.5);
\filldraw[color=teal] (.5, 2.5) circle (6pt)
(6.5, 5.5) circle (6pt);
\node[] at (2.5, 7.5) {\small \( i \)};
\node[] at (-.5, 2.5) {\small \( j \)};
\node[] at (-1, 1.5) {\small \( j-1 \)};
\node[] at (2.5, 5.2) {\small \( a \)};
\node [] at (5.5, 7.4) {\small\( b_1+1 \)};
\node [] at (6.8, 6.2) {\small\( b_2+1 \)};
\node [] at (.6, 3.3) {\small\( b_2\)};
\node [] at (1.5, .7) {\small\( b_1 \)};
\node[] at (4, -1) {\( \fc_{\gamma_0}(P) \)};
\draw[very thick, color=red] (2.2, 2.2)--(2.8, 2.8)
(2.2, 2.8)--(2.8, 2.2);
\draw [very thick, color=red, fill=white] (2.5, 1.5) circle (8pt)
(.5, 1.5) circle (8pt);
\draw[very thick, color=red] (5.2, 5.2)--(5.8, 5.8)
(5.2, 5.8)--(5.8, 5.2);
\end{tikzpicture}
    \begin{tikzpicture}[scale=.4]
\foreach \i in {0,...,4}
\draw[dotted] (0,\i)--(4+\i,\i);
\foreach \i in {0,...,4}
\draw[dotted] (\i, 0)--(\i, 7);
\foreach \i in {0,...,3}
\draw[dotted] (5+\i, \i+1)--(5+\i,7);
\draw[dotted] (0,5)--(8,5);
\draw[dotted] (0,6)--(8,6);
\draw[dotted] (0,7)--(8,7)
(4,0)--(8,4);
\draw[very thick] (5,4)--(7,4)
(2,1)--(2,2)--(3,2)--(3,3);
\filldraw (2.5, 4.5) circle (6pt);
\draw[thick, color=blue] (1.5, 1.5)--(5.5, 1.5)--(5.5, 6.5); 
\filldraw[color=blue] (1.5, 1.5) circle (6pt)
(5.5, 6.5) circle (6pt);
\draw[thick, color=teal] (.5, 2.5)--(6.5, 2.5)--(6.5, 5.5);
\filldraw[color=teal] (.5, 2.5) circle (6pt)
(6.5, 5.5) circle (6pt);
\node[] at (2.5, 7.5) {\small \( i \)};
\node[] at (-.5, 2.5) {\small \( j \)};
\node[] at (-1, 1.5) {\small \( j-1 \)};
\node[] at (2.5, 5.2) {\small \( a \)};
\node [] at (5.5, 7.4) {\small\( b_1+1 \)};
\node [] at (6.8, 6.2) {\small\( b_2+1 \)};
\node [] at (.6, 3.3) {\small\( b_2\)};
\node [] at (1.5, .7) {\small\( b_1 \)};
\node[] at (4, -1) {\( \fc_{\gamma_1}(P) \)};
\draw[very thick, color=red] (2.2, 2.2)--(2.8, 2.8)
(2.2, 2.8)--(2.8, 2.2);
\draw [very thick, color=red, fill=white] (.5, 1.5) circle (8pt);
\draw[very thick, color=red] (5.2, 5.2)--(5.8, 5.8)
(5.2, 5.8)--(5.8, 5.2);
\end{tikzpicture}
    \begin{tikzpicture}[scale=.4]
\foreach \i in {0,...,4}
\draw[dotted] (0,\i)--(4+\i,\i);
\foreach \i in {0,...,4}
\draw[dotted] (\i, 0)--(\i, 7);
\foreach \i in {0,...,3}
\draw[dotted] (5+\i, \i+1)--(5+\i,7);
\draw[dotted] (0,5)--(8,5);
\draw[dotted] (0,6)--(8,6);
\draw[dotted] (0,7)--(8,7)
(4,0)--(8,4);
\draw[very thick] (5,4)--(7,4)
(2,1)--(2,3)--(3,3);
\filldraw (2.5, 4.5) circle (6pt);
\draw[thick, color=blue] (1.5, 2.5)--(6.5, 2.5)--(6.5, 6.5); 
\filldraw[color=blue] (1.5, 2.5) circle (6pt)
(6.5, 6.5) circle (6pt);
\draw[thick, color=teal] (.5, 1.5)--(5.5, 1.5)--(5.5, 5.5);
\filldraw[color=teal] (.5, 1.5) circle (6pt)
(5.5, 5.5) circle (6pt);
\node[] at (2.5, 7.5) {\small \( i \)};
\node[] at (-.5, 2.5) {\small \( j \)};
\node[] at (-1, 1.5) {\small \( j-1 \)};
\node[] at (2.5, 5.2) {\small \( a \)};
\node [] at (5.5, 6.1) {\small\( b_2+1 \)};
\node [] at (6.8, 7.2) {\small\( b_1+1 \)};
\node [] at (1.6, 3.3) {\small\( b_1\)};
\node [] at (.5, .7) {\small\( b_2 \)};
\node[] at (4, -1) {\( \fc_{\gamma_2}(\vartheta(P)) \)};
\draw[very thick, color=red] (1.2, 1.2)--(1.8, 1.8)
(1.2, 1.8)--(1.8, 1.2);
\draw[very thick, color=red] (6.2, 5.2)--(6.8, 5.8)
(6.2, 5.8)--(6.8, 5.2);
\end{tikzpicture}
\caption{If \(b_1< a\le b_2 \), then \( \fc_{\gamma_0}(P)-1= \fc_{\gamma_1}(P)=\fc_{\gamma_2}(\vartheta(P))+1 \).}
\label{fig:case4-3}
\end{figure}
\begin{figure}
  \centering
    \begin{tikzpicture}[scale=.4]
\foreach \i in {0,...,4}
\draw[dotted] (0,\i)--(4+\i,\i);
\foreach \i in {0,...,4}
\draw[dotted] (\i, 0)--(\i, 7);
\foreach \i in {0,...,3}
\draw[dotted] (5+\i, \i+1)--(5+\i,7);
\draw[dotted] (0,5)--(8,5);
\draw[dotted] (0,6)--(8,6);
\draw[dotted] (0,7)--(8,7)
(4,0)--(8,4);
\draw[very thick] (5,4)--(7,4)
(2,1)--(3,1)--(3,3);
\filldraw (2.5, 4.5) circle (6pt);
\draw[thick, color=blue] (1.5, 1.5)--(5.5, 1.5)--(5.5, 6.5); 
\filldraw[color=blue] (1.5, 1.5) circle (6pt)
(5.5, 6.5) circle (6pt);
\draw[thick, color=teal] (.5, 2.5)--(6.5, 2.5)--(6.5, 5.5);
\filldraw[color=teal] (.5, 2.5) circle (6pt)
(6.5, 5.5) circle (6pt);
\node[] at (2.5, 7.5) {\small \( i \)};
\node[] at (-.5, 2.5) {\small \( j \)};
\node[] at (-1, 1.5) {\small \( j-1 \)};
\node[] at (2.5, 5.2) {\small \( a \)};
\node [] at (5.5, 7.4) {\small\( b_1+1 \)};
\node [] at (6.8, 6.2) {\small\( b_2+1 \)};
\node [] at (.6, 3.3) {\small\( b_2\)};
\node [] at (1.5, .7) {\small\( b_1 \)};
\node[] at (4, -1) {\( \fc_{\gamma_0}(P) \)};
\draw[very thick, color=red] (.2, 1.2)--(.8, 1.8)
(.2, 1.8)--(.8, 1.2);
\draw [very thick, color=red, fill=white] (2.5, 2.5) circle (8pt)
(5.5, 5.5) circle (8pt);
\draw[very thick, color=red] (2.2, 1.2)--(2.8, 1.8)
(2.2, 1.8)--(2.8, 1.2);
\end{tikzpicture}
    \begin{tikzpicture}[scale=.4]
\foreach \i in {0,...,4}
\draw[dotted] (0,\i)--(4+\i,\i);
\foreach \i in {0,...,4}
\draw[dotted] (\i, 0)--(\i, 7);
\foreach \i in {0,...,3}
\draw[dotted] (5+\i, \i+1)--(5+\i,7);
\draw[dotted] (0,5)--(8,5);
\draw[dotted] (0,6)--(8,6);
\draw[dotted] (0,7)--(8,7)
(4,0)--(8,4);
\draw[very thick] (5,4)--(7,4)
(2,1)--(2,2)--(3,2)--(3,3);
\filldraw (2.5, 4.5) circle (6pt);
\draw[thick, color=blue] (1.5, 1.5)--(5.5, 1.5)--(5.5, 6.5); 
\filldraw[color=blue] (1.5, 1.5) circle (6pt)
(5.5, 6.5) circle (6pt);
\draw[thick, color=teal] (.5, 2.5)--(6.5, 2.5)--(6.5, 5.5);
\filldraw[color=teal] (.5, 2.5) circle (6pt)
(6.5, 5.5) circle (6pt);
\node[] at (2.5, 7.5) {\small \( i \)};
\node[] at (-.5, 2.5) {\small \( j \)};
\node[] at (-1, 1.5) {\small \( j-1 \)};
\node[] at (2.5, 5.2) {\small \( a \)};
\node [] at (5.5, 7.4) {\small\( b_1+1 \)};
\node [] at (6.8, 6.2) {\small\( b_2+1 \)};
\node [] at (.6, 3.3) {\small\( b_2\)};
\node [] at (1.5, .7) {\small\( b_1 \)};
\node[] at (4, -1) {\( \fc_{\gamma_1}(P) \)};
\draw[very thick, color=red] (.2, 1.2)--(.8, 1.8)
(.2, 1.8)--(.8, 1.2);
\draw [very thick, color=red, fill=white] (2.5, 2.5) circle (8pt)
(5.5, 5.5) circle (8pt);
\end{tikzpicture}
    \begin{tikzpicture}[scale=.4]
\foreach \i in {0,...,4}
\draw[dotted] (0,\i)--(4+\i,\i);
\foreach \i in {0,...,4}
\draw[dotted] (\i, 0)--(\i, 7);
\foreach \i in {0,...,3}
\draw[dotted] (5+\i, \i+1)--(5+\i,7);
\draw[dotted] (0,5)--(8,5);
\draw[dotted] (0,6)--(8,6);
\draw[dotted] (0,7)--(8,7)
(4,0)--(8,4);
\draw[very thick] (5,4)--(7,4)
(2,1)--(2,3)--(3,3);
\filldraw (2.5, 4.5) circle (6pt);
\draw[thick, color=blue] (1.5, 2.5)--(6.5, 2.5)--(6.5, 6.5); 
\filldraw[color=blue] (1.5, 2.5) circle (6pt)
(6.5, 6.5) circle (6pt);
\draw[thick, color=teal] (.5, 1.5)--(5.5, 1.5)--(5.5, 5.5);
\filldraw[color=teal] (.5, 1.5) circle (6pt)
(5.5, 5.5) circle (6pt);
\node[] at (2.5, 7.5) {\small \( i \)};
\node[] at (-.5, 2.5) {\small \( j \)};
\node[] at (-1, 1.5) {\small \( j-1 \)};
\node[] at (2.5, 5.2) {\small \( a \)};
\node [] at (5.5, 6.1) {\small\( b_2+1 \)};
\node [] at (6.8, 7.2) {\small\( b_1+1 \)};
\node [] at (1.6, 3.3) {\small\( b_1\)};
\node [] at (.5, .7) {\small\( b_2 \)};
\node[] at (4, -1) {\( \fc_{\gamma_2}(\vartheta(P)) \)};
\draw [very thick, color=red, fill=white] (1.5, 1.5) circle (8pt)
(6.5, 5.5) circle (8pt);
\end{tikzpicture}
\caption{If \(b_2< a\le b_1 \), then \( \fc_{\gamma_0}(P)= \fc_{\gamma_1}(P)=\fc_{\gamma_2}(\vartheta(P)) \).}
\label{fig:case4-4}
\end{figure}
\end{proof}

By \Cref{lem:ML-1,lem:ML-2}, we obtain \Cref{pro:1}.

\subsection{Proof of the multiplicativity}
\label{sec:proof-mult}

In this subsection, we prove \Cref{prop:multiplicativity}, which states that
for  \( \gamma\in\D_n \) and \( \tilde{\gamma}=\gamma+N^kE^k \), 
\begin{equation}\label{eq:19}
  Y_{\tilde{\gamma}}(\vx;q) = Y_{\gamma}(\vx;q) Y_{N^kE^k}(\vx;q).
\end{equation}
To prove \eqref{eq:19}, we first reformulate it.

Fix a partition \( \nu\vdash n+k \). Let \( L \) and \( R \) be the
 coefficients of \(P_{\nu}({\bf x};q)\) in the left-hand side and
in the right-hand side of \eqref{eq:19}, respectively. Then, it
suffices to show that \( L=R \). By definition,
\begin{equation}\label{eq:7}
  L =q^{\area(\gamma)+\binom{k}{2} -n(\nu)} r_{\tilde{\gamma},\nu}(q)
  \prod_{i\ge1}[m_i (\nu)]_q !.
\end{equation}
By \Cref{pro:5} with \( P_{(1^k)}(\vx;q) = e_{k}(\vx) \), we
have
\[
  Y_{\gamma}(\vx;q) Y_{N^kE^k}(\vx;q)
  =\sum_{\mu\vdash n} q^{\area(\gamma) -n(\mu)} r_{\gamma,\mu}(q)
  P_{\mu}({\bf x};q)  P_{(1^k)}({\bf x};q) [k]_q ! \prod_{i\ge1}[m_i (\mu)]_q ! .
\]

A skew diagram \( \rho/\mu \) is called a \emph{vertical strip} if it
contains at most one cell in each row.
Then, by the Pieri rule \cite[(3.2), p.~215]{Mac} of the Hall--Littlewood polynomials 
\[
  P_\mu(\vx ; q) P_{(1^k)}(\vx ; q)=\sum_{\substack{\rho\vdash n+k \\ \rho/\mu: \text{ vertical strip}}}  P_\rho(\vx ; q)
  \prod_{i \ge 1}
  \qbinom{\rho_i'-\rho_{i+1}'}{\rho_i'-\mu_i'},
\]
we obtain
\begin{equation}\label{eq:8}
  R =\sum_{\substack{\mu\vdash n \\ \nu/\mu: \text{ vertical strip}}} q^{\area(\gamma) -n(\mu)} r_{\gamma,\mu}(q)
  [k]_q !  \prod_{i\ge1}[m_i (\mu)]_q ! \qbinom{\nu_i'-\nu_{i+1}'}{\nu_i'-\mu_i'} .
\end{equation}

By \eqref{eq:7} and \eqref{eq:8}, to prove
\Cref{prop:multiplicativity}, it suffices to show that
\begin{equation}\label{eq:9}
  r_{\tilde{\gamma},\nu}(q) =
  \sum_{\substack{\mu\vdash n \\ \nu/\mu: \text{ vertical strip}}} q^{n(\nu)-n(\mu)-\binom{k}{2}} r_{\gamma,\mu}(q)
  [k]_q !  \prod_{i\ge1} \frac{[m_i (\mu)]_q !}{[m_i (\nu)]_q !} 
      \qbinom{\nu_i'-\nu_{i+1}'}{\nu_i'-\mu_i'} .
\end{equation}
Observe that if \( \nu/\mu \) is a vertical strip, both
\( \mu_i'-\nu_{i+1}' \) and \( \nu_{i+1}'-\mu'_{i+1} \) are
nonnegative. Since \( m_i(\nu) = \nu_i'-\nu_{i+1}' \) and
\begin{equation}\label{eq:22}
  m_i(\mu) = \mu_i'-\mu_{i+1}' = (\mu_i'-\nu_{i+1}') +
  (\nu_{i+1}'-\mu'_{i+1}),
\end{equation}
the identity \eqref{eq:9} (and hence also
\Cref{prop:multiplicativity}) is equivalent to the following
proposition.

\begin{prop}\label{pro:key}
  Let \( \gamma\in\D_n \), \( \tilde{\gamma}=\gamma+N^kE^k \), and
  \( \nu\vdash n+k \). Then we have
\begin{equation}\label{eq:10}
        r_{\tilde{\gamma},\nu}(q) =
   \sum_{\substack{\mu\vdash n \\ \nu/\mu: \text{ vertical strip}}}  q^{n(\nu)-n(\mu)-\binom{k}{2}} r_{\gamma,\mu}(q)
     \frac{[k]_q !}{[\nu_1'-\mu_1']_q!} 
       \prod_{i\ge 1}  \qbinom{m_i(\mu)}{\nu_{i+1}'-\mu_{i+1}'} .
  \end{equation}
\end{prop}

The rest of this subsection is devoted to proving \Cref{pro:key}. We
first illustrate the idea of the proof. Observe that to prove
\eqref{eq:10} when \( q=1 \), it suffices to construct a bijection
\[
  \phi:\LRP(\tilde{\gamma},\nu) \to \bigcup_{\substack{\mu\vdash n \\ \nu/\mu: \text{ vertical strip}}}
  \LRP(\gamma,\mu) \times W(\nu,\mu),
\]
for some set \( W(\nu,\mu) \) of cardinality
\( \frac{k!}{(\nu_1'-\mu_1')!} \prod_{i\ge 1}
\binom{m_i(\mu)}{\nu_{i+1}'-\mu_{i+1}'} \). There is a natural way to
construct such a map as follows.

Consider \( P\in \LRP(\tilde{\gamma},\nu) \) as shown in the left
diagram in \Cref{fig:image6}. By restricting \( P \) to the square
\( [0,n]\times[0,n] \), we obtain \( Q\in \LRP(\gamma,\mu) \) for some
partition \( \mu \); see the right diagram in \Cref{fig:image6}. Note
that \( P \) is determined by \( Q \) and the rooks in rows
\( n+1,\dots,n+k+1 \) including the extended rooks. These rooks can be
encoded as a word \( w=w_1 \cdots w_k \), where \( w_i=j_t \) if the
leftmost rook in row \( n+i \) has rank \( j \) and it is the
\( t \)-th leftmost rank \( j \) rook in rows \( n+1,\dots,n+k+1 \).
We will see later that the map \( \phi(P) = (Q,w) \) is a desired
bijection.

\begin{figure}
  \centering
  \begin{tikzpicture}[scale=.45]
  \fill[color=pink!25!white] (0,1)--(1,1)--(1,2)--(2,2)--(2,4)--(3,4)--(3,6)--(5,6)--(5,7)--(7,7)--(7,12)--(0,12)--(0,1)--cycle;
\draw[dotted] (0,0)--(0,12);
\foreach \i in {1,...,12}{
  \draw[dotted] (0,\i)--(\i,\i);}
\foreach \i in {1,...,12}
\draw[dotted] (\i,12)--(\i,\i);
\draw[dotted] (0,0)--(12,12);
\draw[dotted, thick] (0,7)--(7,7);
\draw[very thick] (0,0)--(0,1)--(1,1)--(1,2)--(2,2)--(2,4)--(3,4)--(3,6)--(5,6)--(5,7)--(7,7)--(7,12)--(12,12);
\draw[thick] (.5, .5)--(.5, 12.5);
\filldraw (.5, .5) circle (5pt)
(.5,12.5) circle (5pt);
\draw[thick, color=cyan] (1.5, 1.5)--(1.5, 8.5)--(8.5, 8.5)--(8.5, 12.5);
\filldraw[color=cyan] (1.5, 1.5) circle (5pt)
(1.5, 8.5) circle (5pt)
(8.5, 8.5) circle (5pt)
(8.5, 12.5) circle (5pt);
\draw[thick, color=red] (2.5, 2.5)--(2.5, 4.5)--(4.5, 4.5)--(4.5, 12.5);
\filldraw[color=red] (2.5, 2.5) circle (5pt)
(2.5, 4.5) circle (5pt)
(4.5, 4.5) circle (5pt)
(4.5, 12.5) circle (5pt);
\draw[thick, color=teal] (3.5, 3.5)--(3.5, 6.5)--(6.5, 6.5)--(6.5, 9.5)--(9.5, 9.5)--(9.5, 12.5);
\filldraw[color=teal] (3.5, 3.5) circle (5pt)
(3.5, 6.5) circle (5pt)
(6.5, 6.5) circle (5pt)
(6.5, 9.5) circle (5pt)
(9.5, 9.5) circle (5pt)
(9.5, 12.5) circle (5pt); 
\draw[thick, color=orange] (5.5, 5.5)--(5.5, 10.5)--(10.5, 10.5)--(10.5, 12.5);
\filldraw[color=orange] (5.5, 5.5) circle (5pt)
(5.5, 10.5) circle (5pt)
(10.5, 10.5) circle (5pt)
(10.5, 12.5) circle (5pt);
\draw[thick, color=violet] (7.5, 7.5)--(7.5, 12.5);
\filldraw[color=violet] (7.5, 7.5) circle (5pt)
(7.5,12.5) circle (5pt);
\draw[thick, color=frenchblue] (11.5, 11.5)--(11.5, 12.5);
\filldraw[color=frenchblue] (11.5, 11.5) circle (5pt)
(11.5,12.5) circle (5pt);
\node [] at (.5, 13.1) {\small\( 1 \)};
\node [] at (1.5, 9.1) {\small\( 1 \)};
\node [] at (7.5, 13.1) {\small\( 1 \)};
\node [] at (9.5, 13.1) {\small\( 3 \)};
\node [] at (8.5, 13.1) {\small\( 2 \)};
\node [] at (2.5, 5.1) {\small\( 1 \)};
\node [] at (3, 6.5) {\small\( 1 \)};
\node [] at (4.5, 13.1) {\small\( 2 \)};
\node [] at (5.5, 11.1) {\small\( 1 \)};
\node [] at (10.5, 13.1) {\small\( 2 \)};
\node [] at (6.5, 10.1) {\small\( 2 \)};
\node [] at (11.5, 13.1) {\small\( 1 \)};
\node [] at (.8, -.3) {\small \( \tilde{R}_{1,1} \)};
\node [] at (1.8, .7) {\small \( \tilde{R}_{1,2} \)};
\node [] at (2.8, 1.7) {\small \( \tilde{R}_{2,1} \)};
\node [] at (3.8, 2.7) {\small \( \tilde{R}_{2,2} \)};
\node [] at (5.8, 4.7) {\small \( \tilde{R}_{1,3} \)};
\node [] at (7.8, 6.7) {\small \( \tilde{R}_{0,1} \)};
\node [] at (11.8, 10.7) {\small \( \tilde{R}_{0,2} \)};
\end{tikzpicture}\quad
\begin{tikzpicture}[scale=.45]
  \fill[color=pink!25!white] (0,1)--(1,1)--(1,2)--(2,2)--(2,4)--(3,4)--(3,6)--(5,6)--(5,7)--(0,7)--(0,1)--cycle;
\draw[dotted] (0,0)--(0,12);
\foreach \i in {1,...,12}{
  \draw[dotted] (0,\i)--(\i,\i);}
\foreach \i in {1,...,12}
\draw[dotted] (\i,12)--(\i,\i);
\draw[dotted] (0,0)--(12,12);
\draw[dotted, thick] (0,7)--(7,7);
\draw[very thick] (0,0)--(0,1)--(1,1)--(1,2)--(2,2)--(2,4)--(3,4)--(3,6)--(5,6)--(5,7)--(7,7)--(7,12)--(12,12);
\draw[thick] (.5, .5)--(.5, 7.5);
\filldraw (.5, .5) circle (5pt)
(.5, 7.5) circle (5pt);
\draw[thick, color=cyan] (1.5, 1.5)--(1.5, 7.5);
\filldraw[color=cyan] (1.5, 1.5) circle (5pt)
(1.5, 7.5) circle (5pt);
\draw[thick, color=red] (2.5, 2.5)--(2.5, 4.5)--(4.5, 4.5)--(4.5, 7.5);
\filldraw[color=red] (2.5, 2.5) circle (5pt)
(2.5, 4.5) circle (5pt)
(4.5, 4.5) circle (5pt)
(4.5, 7.5) circle (5pt);
\draw[thick, color=teal] (3.5, 3.5)--(3.5, 6.5)--(6.5, 6.5)--(6.5, 7.5);
\filldraw[color=teal] (3.5, 3.5) circle (5pt)
(3.5, 6.5) circle (5pt)
(6.5, 6.5) circle (5pt)
(6.5, 7.5) circle (5pt);
\draw[thick, color=orange] (5.5, 5.5)--(5.5, 7.5);
\filldraw[color=orange] (5.5, 5.5) circle (5pt)
(5.5, 7.5) circle (5pt);
\node [] at (.5, 8.1) {\small\( 1 \)};
\node [] at (1.5, 8.1) {\small\( 1 \)};
\node [] at (2.5, 5.1) {\small\( 1 \)};
\node [] at (3, 6.5) {\small\( 1 \)};
\node [] at (4.5, 8.1) {\small\( 2 \)};
\node [] at (5.5, 8.1) {\small\( 1 \)};
\node [] at (6.5, 8.1) {\small\( 2 \)};
\node [] at (.8, -.3) {\small \( R_{1,1} \)};
\node [] at (1.8, .7) {\small \( R_{1,2} \)};
\node [] at (2.8, 1.7) {\small \( R_{2,1} \)};
\node [] at (3.8, 2.7) {\small \( R_{2,2} \)};
\node [] at (5.8, 4.7) {\small \( R_{1,3} \)};
\end{tikzpicture}
  \caption{The left diagram shows \( \ext(P) \) for a linked rook
    placement \( P \). Then \( \phi(P) = (Q,w) \), where \( \ext(Q) \)
    is shown on the right and \( w = 0_1 1_2 2_21_30_2 \).}
  \label{fig:image6}
\end{figure}

Now we introduce necessary definitions. We use the letters \( i_a \),
for integers \( i\ge0 \) and \( a\ge1 \), with the ordering given by
\( i_1>i_2 > \cdots \), and \( i_a<j_b \) for any \( i<j \) and
\( a,b\ge1 \). For a partition \( \mu \), let
  \[
    M_\mu := \{i_a: 1\le i\le \mu_1, 1\le a\le m_i(\mu)\}.
  \]
  For a word \( w=w_1 \cdots w_k \), we write
  \( \set(w) := \{w_1,\dots,w_k\} \).

\begin{defn}\label{def:4}
  Suppose that \( \nu\vdash n+k \) and \( \mu\vdash n \) are partitions such that
  \( \nu/\mu \) is a vertical strip. 
  We denote by \( W(\nu,\mu) \) the set of words
  \( w=w_1 \cdots w_k \) on the alphabet
  \( M_\mu\cup \{0_1,0_2,\dots\} \) satisfying the following
  conditions:
  \begin{enumerate}
  \item there are no repeated letters in \( w \), 
  \item \( |\{j_a\in \set(w): a\ge1\}| = \nu'_{j+1}-\mu'_{j+1} \) for all \( j\ge0 \),
  \item the subword of \( w \) consisting of the letters \( 0_a \) for
    \( a\ge1 \) is \( 0_10_2 \cdots 0_{\nu'_1-\mu'_1} \). 
  \end{enumerate}
  For a word \( w=w_1 \cdots w_k\in W(\nu,\mu) \) and a set
  \( B\subseteq M_\mu \cup \{0_1,0_2,\dots\} \), we define
  \begin{align*}
  f(w)  &= |\{(w_i,w_j): 1\le i<j\le k,\, w_i<w_j\}|, \\
  f_\mu(B)  &= |\{(i_a,j_b): i\ge0,\, j\ge1,\, i_a\in B,\,  j_b\in M_\mu\setminus B,\, i_a<j_b\}|.
  \end{align*}
\end{defn}

The first ingredient of the proof of \Cref{pro:key} is the following
lemma.

\begin{lem}\label{lem:4}
  Let \( \gamma\in\D_n \), \( \tilde{\gamma}=\gamma+N^kE^k \), and
  \( \nu\vdash n+k \). There is a bijection
  \[
    \phi:\LRP(\tilde{\gamma},\nu) \to \bigcup_{\substack{\mu\vdash n \\ \nu/\mu: \text{ vertical strip}}}
    \LRP(\gamma,\mu) \times W(\nu,\mu)
  \]
  such that if \( \phi(P) = (Q,w)\in \LRP(\gamma,\mu) \times W(\nu,\mu) \), then
\begin{equation}\label{eq:12}
  \fc_{\tilde{\gamma}}(P) = \fc_{\gamma}(Q) + f_\mu(\set(w)) + f(w).
\end{equation}
\end{lem}

\begin{proof}
  As a guiding example of the construction of the map \( \phi \), see
  Figure~\ref{fig:image6}. Consider
  \( P\in \LRP(\tilde{\gamma},\nu) \). Let \( Q \) be the restriction
  of \( P \) to the square \( [0,n]\times[0,n] \), and let \( \mu \)
  be the type of \( Q \). For each \( i\ge1 \), there are
  \( m_i(\mu) \) extended linked rooks of size \( i \) in
  \( \ext(Q) \), say \( R_{i,1},R_{i,2}, \dots,R_{i,m_i(\mu)}\), which
  are arranged so that, if the last extended rook of \( R_{i,t} \) is
  in column \( c_{i,t} \), then
  \( c_{i,1}<c_{i,2} < \cdots < c_{i,m_i(\mu)} \).

  For \( i\ge1 \), let \( \tilde{R}_{i,t} \) be the extended linked
  rook of \( P \) containing \( R_{i,t} \) except for its last
  extended rook. Let
  \( \tilde{R}_{0,1},\tilde{R}_{0,2},\dots,\tilde{R}_{0,d} \) be
  extended linked rooks of size \( 1 \) in \( \ext(P) \) that lie in
  the last \( k \) columns, and suppose that they lie in columns
  \( c_{0,1},\dots,c_{0,d} \), respectively, where
  \( n+1\le c_{0,1} < \cdots <c_{0,d}\le n+k \). Then every extended
  linked rook \( \tilde{R}_{i,t} \) in \( \ext(P) \) falls into one of
  the following cases:
  \begin{description}
  \item[Case 1] \( i=0 \), i.e., \( \tilde{R}_{i,t} \) is contained in columns
    \( n+1,\dots,n+k \).
  \item[Case 2] \( i\ge1 \) and \( \tilde{R}_{i,t} \) is obtained from
    \( {R}_{i,t} \) by moving the last extended rook in row \( n+1 \)
    to row \( n+k+1 \).
  \item[Case 3] \( i\ge1 \) and \( \tilde{R}_{i,t} \) is obtained from
    \( {R}_{i,t} \) by moving the last extended rook in row \( n+1 \)
    to row \( n+j \) for some \( 1\le j\le k \) and linking it to the new
    rooks at \( (n+j,n+j) \) and \( (n+j,n+k+1) \).
  \end{description}

  In the first case, the size of \( \tilde{R}_{i,t} \) is \( 1 \). In
  the second case, \( \tilde{R}_{i,t} \) has the same size as
  \( R_{i,t} \), while in the third case, the size is increased by
  \( 1 \). This implies that the Young diagram of \( \nu \) is
  obtained from that of \( \mu \) by adding at most one cell in each
  row. Therefore, \( d=\nu'_1-\mu'_1 \) and \( \nu/\mu \) is a
  vertical strip; see \Cref{fig:vertical-strip}.

\begin{figure}
    \centering
    \begin{tikzpicture}[scale=.45, every node/.style={font=\small}]
      \fill[color=gray!25]
        (0,4) rectangle (1,13)
        (1,6) rectangle (2,13)
        (4.5,8.5) rectangle (5.5,13)
        (5.5,11.5) rectangle (6.5,13);
      \draw[dashed] (2,13)--(4.5,13) (6.5,13)--(8.5,13);
      \fill[color=red!13!white]
        (0,2.5) rectangle (1,4)
        (1,5) rectangle (2,6)
        (4.5,7) rectangle (5.5,8.5)
        (5.5,10) rectangle (6.5,11.5);
      \draw[thick]
        (0,2.5)--(0,13)--(2,13)
        (0,4)--(1,4)--(1,6)--(2,6)--(2,7)
        (4.5,13)--(6.5,13)
        (5.5,8.5)--(5.5,10);
      \draw[thick, color=red]
        (0,2.5) rectangle (1,4)
        (1,5) rectangle (2,6)
        (4.5,7) rectangle (5.5,8.5)
        (5.5,10) rectangle (6.5,11.5);
      \draw[thick, decorate, decoration={brace,raise=5pt}]
        (1,4)--(1,2.5);
      \draw[thick, decorate, decoration={brace,raise=5pt}]
        (5.5,9.9)--(5.5,8.6);
      \draw[thick, decorate, decoration={brace,raise=5pt}]
        (5.5,8.4)--(5.5,7);
      \draw[thick, decorate, decoration={brace,raise=5pt}]
        (6.5,11.5)--(6.5,10);
      \draw[thick, decorate, decoration={brace,raise=5pt}]
        (10.6,10)--(10.6,7);
      \draw[thick, decorate, decoration={brace,raise=5pt}]
        (14,11.5)--(14,8.5);
      \node at (.5,13.6) {\( 1 \)};
      \node at (1.5,13.6) {\( 2 \)};
      \node at (4.8,13.6) {\( j \)};
      \node at (6.2,13.6) {\( j+1 \)};
      \node at (1,10) {\( \mu \)};
      \node at (3.25,8) {\( \cdots \)};
      \node at (8,12) {\( \cdots \)};
      \node at (3.7,3.25) {\( d=\nu_1'-\mu_1' \)};
      \node at (7.5,7.7) {\( \nu_j'-\mu_j' \)};
      \node at (7.8,9.25) {\( \mu_j'-\nu_{j+1}' \)};
      \node at (9.25,10.75) {\( \nu_{j+1}'-\mu_{j+1}' \)};
      \node at (12.3,8.5) {\( m_j(\nu) \)};
      \node at (15.7,10) {\( m_j(\mu) \)};
    \end{tikzpicture}
    \caption{The Young diagram of \( \nu \), where the cells in
      \( \mu \) are colored gray and the cells in \( \nu/\mu \) are
      colored red.}
    \label{fig:vertical-strip}
  \end{figure}

  We define \( \phi(P) = (Q,w) \), where \( w=w_1 \cdots w_k \) is the
  word such that \( w_i=j_t \) if row \( n+i \) contains an extended
  rook of \( \tilde{R}_{j,t} \) for some \( j\ge0 \) and \( t\ge1 \).
  Note that \( w \) is well-defined because each row contains some
  extended rooks from exactly one extended linked rook.

  We first show that
  \( \phi(P)=(Q,w)\in \LRP(\gamma,\mu) \times W(\nu,\mu) \). By
  construction, we have \( Q\in \LRP(\gamma,\mu) \). To prove that
  \( w\in W(\nu,\mu) \), we need to verify the three conditions in
  \Cref{def:4}. The first and third conditions follow immediately from
  the construction of \( w \). For the second condition, suppose that
  \( w_i=j_t \) for some \( j\ge0 \) and \( t\ge1 \). This means that
  \( \tilde{R}_{j,t} \) has an extended rook in row \( n+i \). Suppose
  that the leftmost such extended rook is in column \( c \). If
  \( j=0 \), then \( n+1\le c\le n+k \), and the size of
  \( \tilde{R}_{0,t} \) is \( 1 \). If \( j\ge 1 \), then
  \( 1\le c\le n \), and the size of \( \tilde{R}_{j,t} \) is
  \( j+1 \), while the size of \( R_{j,t} \) is \( j \). This implies
  that \( |\{j_a\in \set(w): a\ge1\}| \) is equal to the number of
  cells of \( \nu/\mu \) in column \( j+1 \), which is
  \( \nu'_{j+1}-\mu'_{j+1} \), for all \( j\ge0 \); see
  \Cref{fig:vertical-strip}. Hence, the second condition holds.

Observe that, using the above notation, \( P \) can be reconstructed
from \( Q \) by adding a rook at the cell \( (c_{j,t},n+i) \) for each
\( i\in [k] \) such that \( w_i=j_t \) for some \( j\ge1 \) and
\( t\ge1 \). This establishes that \( \phi \) is a bijection. Hence, it
remains to prove \eqref{eq:12}. To do this, we define the following:
\begin{enumerate}
\item \( \fc_{\tilde{\gamma}}^1(P) \) is the number of free cells of
\( P \) in rows \( 1,2,\dots,n \),
\item \( \fc_{\tilde{\gamma}}^2(P) \) is the number of free cells of \( P \) 
  in rows \( n+1,\dots,n+k \) whose fc-pairs have a cell in row \( n+k+1 \),
\item \( \fc_{\tilde{\gamma}}^3(P) \) is the number of free cells of \( P \) 
  in rows \( n+1,\dots,n+k \) whose fc-pairs have no cell in row \( n+k+1 \).
\end{enumerate}
Then, \( \fc_{\tilde{\gamma}}(P)= \fc_{\tilde{\gamma}}^1(P)+\fc_{\tilde{\gamma}}^2(P)+\fc_{\tilde{\gamma}}^3(P) \).
Thus, to prove \eqref{eq:12}, it suffices to show that
\begin{align}
\label{eq:24}  \fc_{\tilde{\gamma}}^1(P)  &= \fc_{\gamma}(Q), \\
\label{eq:25}  \fc_{\tilde{\gamma}}^2(P)  &= f_\mu(\set(w)), \\
\label{eq:26}  \fc_{\tilde{\gamma}}^3(P)  &= f(w).
\end{align}

Observe that every extended rook of \( R_{i,j} \) remains in the same
position in \( \tilde{R}_{i,j} \) except for its last extended rook,
which only moves upward in the same column. Hence, the free cells of
\( Q \) are exactly those of \( P \) in rows \( 1,2,\dots,n \), which
shows \eqref{eq:24}.

For \eqref{eq:25}, consider a free cell in row \( n+i \) counted in
\( \fc_{\tilde{\gamma}}^2(P) \) whose fc-pair consists of \( A \) and \( B \),
where \( A \) is the extended rook in row \( n+i \) and \( B \) is the
extended rook above the free cell. Let \( \tilde{R}_{u,a} \) and
\( \tilde{R}_{v,b} \) be the extended linked rooks containing \( A \)
and \( B \), respectively. Then, we have \( u_a\in \set(w) \),
\( v_b\in M_\mu\setminus \set(w) \), and \( u_a<v_b \). This shows
\eqref{eq:25}.

The last identity \eqref{eq:26} can be proved similarly. Consider a
free cell in row \( n+i \) counted in \( \fc_{\tilde{\gamma}}^3(P) \)
whose fc-pair consists of \( A \) and \( B \), where \( A \) is the extended
rook in row \( n+i \) and \( B \) is the extended rook in row
\( n+j \) with \( i<j\le k \). Let \( \tilde{R}_{u,a} \) and
\( \tilde{R}_{v,b} \) be the extended linked rooks containing \( A \)
and \( B \), respectively. Then, we have \( u_a,v_b\in \set(w) \) and
\( u_a<v_b \). Thus, we obtain \eqref{eq:26}, which completes the
proof.
\end{proof}

By \Cref{lem:4}, we have
\[
  r_{\tilde{\gamma},\nu}(q) =
  \sum_{\substack{\mu\vdash n \\ \nu/\mu: \text{ vertical strip}}}
  r_{\gamma,\mu}(q) \sum_{w\in W(\nu,\mu)} q^{f_\mu(\set(w))+f(w)}.
\]
Therefore, to complete the proof of \Cref{pro:key}, it remains to
prove the following lemma.

\begin{lem}\label{lem:3}
  Let \( \nu\vdash n+k \) and \( \mu\vdash n \) be partitions such
  that \( \nu/\mu \) is a vertical strip. Then
  \[
    \sum_{w\in W(\nu,\mu)} q^{f_\mu(\set(w))+f(w)}
    =q^{n(\nu)-n(\mu)-\binom{k}{2}} \frac{[k]_q !}{[\nu_1'-\mu_1']_q!} 
       \prod_{i\ge 1}  \qbinom{m_i(\mu)}{\nu_{i+1}'-\mu_{i+1}'}.
  \]
\end{lem}

To prove \Cref{lem:3}, we need a simple lemma.

\begin{lem}\label{lem:2}
  Let \( (a_1,a_2,\dots) \) and \( (b_1,b_2,\dots) \) be any sequences
  of integers with finitely many nonzero elements.
  Then
  \begin{equation}\label{eq:11}
    \sum_{1\le i\le j}  (a_i-b_i)(b_j-a_{j+1})
    = \sum_{i\ge1}  \left( \binom{a_i}{2} - \binom{b_i}{2} \right)
    - \binom{\sum_{i\ge1} (a_i-b_i)}{2}.
  \end{equation}
\end{lem}

\begin{proof}
  Let \( L \) and \( R \) be the left-hand side and the right-hand
  side of \eqref{eq:11}, respectively. It suffices to show that
  \( L=R \) as formal power series in the variables in
  \( X = \{a_1,a_2,\dots, b_1,b_2,\dots\} \), where we consider each
  binomial coefficient \( \binom{c}{2} = c(c-1)/2 \) as a polynomial
  in these variables. Then, it is straightforward to check that
  \( [uv]L = [uv]R \) for any \( u,v\in X \), where \( [uv]L \) is the
  coefficient of \( uv \) in \( L \). More precisely, for every
  \( i \) and \( j \) with \( 1\le i<j \), we have
  \begin{align*}
  [a_i^2] L  &= [a_i^2] R =0,\\
  [a_ia_j] L  &= [a_ia_j] R =-1,\\
  [b_i^2] L  &= [b_i^2] R =-1,\\
  [b_ib_j] L  &= [b_ib_j] R =-1,\\
  [a_ib_i] L  &= [a_ib_i] R =1,\\
  [a_ib_j] L  &= [a_ib_j] R =1,\\
  [a_jb_i] L  &= [a_jb_i] R =1.
  \end{align*}
  Hence, \( L=R \) as formal power series.
\end{proof}

Finally, we prove \Cref{lem:3}.

\begin{proof}[Proof of \Cref{lem:3}]
  Let \( \mathcal{B} = \{\set(w):w\in W(\nu,\mu)\} \). Equivalently,
  \( \mathcal{B} \) is the set of all subsets
  \( B\subseteq M_\mu\cup \{0_1,0_2,\dots\} \) such that
  \( B\cap\{0_1,0_2,\dots\}=\{0_a:1\le a\le\nu'_1-\mu'_1\} \) and
  \begin{equation}\label{eq:6}
    \mbox{\( |\{i_a\in B: a\ge1 \}| = \nu'_{i+1}-\mu'_{i+1} \) \quad for all \( i\ge0 \).}
  \end{equation}
  By definition, we have
  \[
    \sum_{w\in W(\nu,\mu)} q^{f_\mu(\set(w))+f(w)}
    = \sum_{B\in\mathcal{B}} q^{f_\mu(B)}
    \sum_{\substack{w\in W(\nu,\mu)\\ \set(w)=B}} q^{f(w)}.
  \]

  We claim that
\begin{equation}\label{eq:15}
  \sum_{B\in \mathcal{B}}  q^{f_\mu(B)} = q^{\sum_{1\le i\le j} (\nu'_i-\mu'_i) (\mu'_j-\nu'_{j+1})}
  \prod_{i\ge1} \qbinom{m_i(\mu)}{\nu'_{i+1}-\mu'_{i+1}},
\end{equation}
and that for any \( B\in \mathcal{B} \),
\begin{equation}
  \label{eq:13}
  \sum_{\substack{w\in W(\nu,\mu)\\ \set(w)=B}}  q^{f(w)} = \frac{[k]_q!}{[\nu'_1-\mu'_1]_q!}.
\end{equation}
By \Cref{lem:2}, the exponent of the power of \( q \) in the
right-hand side of \eqref{eq:15} is
\[
  \sum_{1\le i\le j} (\nu'_i-\mu'_i) (\mu'_j-\nu'_{j+1})
    = \sum_{i\ge1}  \left( \binom{\nu'_i}{2} - \binom{\mu'_i}{2} \right)
    - \binom{\sum_{i\ge1} (\nu'_i-\mu'_i)}{2}
  = n(\nu)-n(\mu)-\binom{k}{2}.
\]
Thus, it suffices to prove \eqref{eq:15} and \eqref{eq:13}.

To prove \eqref{eq:15}, suppose \( B\in\mathcal{B} \).
Then, by definition, \( f_\mu(B) \) is equal to
\[
   |\{(i_a,j_b): 0\le i<j,\,  i_a\in  B,\,  j_b\in M_\mu\setminus B\}|
  +|\{(i_a,i_b): i\ge1,\,  i_a\in B,\,  i_b\in M_\mu\setminus B,\,  a>b\}|.
\]
Note that by \eqref{eq:22} and \eqref{eq:6}, we have
\begin{equation}\label{eq:23}
    |\{a: i_a\in M_\mu \setminus B\}| =
  m_i(\mu)- (\nu_{i+1}'-\mu_{i+1}') = \mu'_i-\nu'_{i+1}
  \quad \mbox{for all \( i\ge1 \)}.
\end{equation}
Therefore, by \eqref{eq:6} and \eqref{eq:23}, we always have
\begin{equation}\label{eq:14}
  |\{(i_a,j_b): 0\le i<j,\, i_a\in  B,\,  j_b\in M_\mu\setminus B\}|
= \sum_{0\le i< j} (\nu'_{i+1}-\mu'_{i+1})(\mu'_j-\nu'_{j+1}) .
\end{equation}
By the well-known property of \( q \)-binomial coefficients
\cite[Proposition~1.7.1]{EC1}, we have
\[
  \sum_{B\in\mathcal{B}} q^{|\{(i_a,i_b)\, :\,  i\ge1,\,  i_a\in B,\,  i_b\in M_\mu\setminus B,\,  a>b\}|}
  = \prod_{i\ge1} \qbinom{m_i(\mu)}{\nu'_{i+1}-\mu'_{i+1}}.
\] 
By combining these results, we obtain \eqref{eq:15}.

Suppose now that \( w\in W(\nu,\mu) \) and \( \set(w)= B \) for some
fixed \( B\in \mathcal{B} \). Then \( w \) is a permutation of the
elements in \( B \) such that the letters \( 0_a \) for \( a\ge1 \)
are arranged as \( 0_10_2 \cdots 0_{\nu'_1-\mu'_1} \). Thus,
\eqref{eq:13} also follows from \cite[Proposition~1.7.1]{EC1}, which
completes the proof.
\end{proof}

\section{Concluding remarks}\label{sec:remarks}

Here, we propose some open problems that require further investigation.

\subsection{Direct bijection}
The Schur expansion of the chromatic quasisymmetric function is
well-known in terms of \(P\)-tableaux \cite[Theorem 6.3]{SW}: for \( \gamma\in\D_n \),
\begin{equation}\label{eqn:remark1}
X_\gamma(\vx;q)=\sum_{\lambda\vdash n} \left( \sum_{\substack{T: P \text{-tableau}\\ {\rm sh}(T)=\lambda}} q^{\inv(T)}\right) s_{\lambda}(\vx).
\end{equation}
Here, \( P \) is the natural unit interval order associated with
\( \gamma \), \( {\rm sh}(T) \) is the shape of \( T \), and
\( \inv(T) \) counts the edges \( \{i,j\} \) of \( {\rm inc}(P) \)
such that \( i<j \) and \( i \) appears in a lower row than \( j \)
in \( T \).
It is also well known \cite[(6.5), p.242]{Mac} that the Schur
function has the following Hall--Littlewood expansion:
\begin{equation}\label{eqn:remark2}
s_{\lambda}(\vx) = \sum_{\mu\vdash n}K_{\lambda\mu}(q) P_{\mu}(\vx;q).
\end{equation}
Combining \eqref{eqn:remark1} and \eqref{eqn:remark2} gives another Hall--Littlewood expansion of chromatic quasisymmetric functions computed over the 
set of pairs of a \(P\)-tableau and a semistandard Young tableau. It would be interesting to see if there exists any correspondence between the set of pairs of \( P \)-tableaux and semistandard Young tableaux and 
the set of linked rook placements.

\subsection{Refinement of hit polynomials}

For \( \gamma\in\D_n \) and an integer \( \alpha\ge n \),
the principal specialization of the chromatic symmetric function is given in
\cite[Proposition 2.6]{HOY} as
\[
  X_\gamma(1,q,\dots, q^{\alpha-1};q)= q^{\area(\gamma)}\prod_{i=1}^n[\alpha-a_i(\gamma)]_q=q^{\area(\gamma)}\sum_{k=0}^{n}R_{n-k}(B_\gamma;q)[\alpha]_q^{\underline{k}},
\]
where \( B_\gamma \) is the Ferrers board with \( i \)-th column
height \( c_i=i-1-a_i(\gamma) \), \( R_k(B_\gamma;q) \) is the
\( k \)-th \emph{\( q \)-rook polynomial} \cite{GR86} on the board
\( B_\gamma \), and \( [\alpha]_q^{\underline{k}} \) is the
\( k \)-falling factorial
\( [\alpha]_q[\alpha-1]_q\cdots [\alpha-k+1]_q \). If we compare the
result of the principal specialization of the Hall--Littlewood
expansion of \( X_\gamma (\vx;q) \) in \eqref{eqn:HLexpthm}, by noting
the fact \cite[Example 1, p.~213]{Mac} that
\[
  \prod_{i\ge 1}[m_i (\mu)]_q ! P_\mu(1,q,\dots, q^{\alpha -1};q)=\frac{q^{n(\mu)}}{(1-q)^{\ell(\mu)}}
  \frac{\prod_{i=1}^{\alpha}(1-q^i)}{\prod_{i=1}^{\alpha -\ell(\mu)}(1-q^i)}=q^{n(\mu)}[\alpha]_q ^{\underline{\ell(\mu)}},
\]
we get
\[
  \sum_{k=0}^n R_{n-k}(B_\gamma ;q)[\alpha]_q ^{\underline{k}}
  =\sum_{\mu\vdash n}r_{\gamma, \mu}(q)[\alpha]_q ^{\underline{\ell(\mu)}}
  = \sum_{k=0}^n \left( \sum_{\substack{\mu\vdash n\\ \ell(\mu) = k}}r_{\gamma, \mu}(q) \right) [\alpha]_q ^{\underline{k}}.
\]
This shows that
\begin{equation}\label{eq:27}
   R_{n-k}(B_\gamma ;q) = \sum_{\substack{\mu\vdash n\\ \ell(\mu) = k}}r_{\gamma, \mu}(q).
\end{equation}
Thus, we can consider \( r_{\gamma, \mu} (q)\) as a refinement of the original rook polynomial.
On the other hand, we also have the following identity \cite{BHR}:
\[
  \prod_{i=1}^n[\alpha-a_i(\gamma)]_q=  \sum_{k=0}^{n}R_{n-k}(B_\gamma;q)[\alpha]_q^{\underline{k}}
  = \sum_{k=0}^n h_k (B_\gamma ;q)\qbinom{\alpha +k}{n} ,
\]
where \( h_k (B_\gamma ;q) \) is the \( k \)-th \emph{\( q \)-hit polynomial} \cite{Dwo98, Hag98a}. Hence, a natural question would be whether one can find
a refinement of the \( q \)-hit polynomials in a similar vein.

\appendix
\section{Comparison with known Hall--Littlewood expansions}\label{sec:appendix}

In this appendix, we compare our main result (\Cref{thm:main}) with
known Hall--Littlewood expansions of the chromatic symmetric function
\( X_\gamma (\bm x;q) \).

The \emph{flag variety}
\[
\Fl(\CC^n) := 
\{ V_{\bullet}=(\{0\}\subseteq V_1 \subseteq V_2\subseteq  \cdots\subseteq \CC^n) \,:\, \dim_{\CC}(V_i) = i \, \textup{ for all }\, i\in[n] \}
\]
is the collection of full flags in \( \CC^n \). For a sequence
\( \vm=(m_1,\dots, m_n) \) corresponding to a given Dyck path
\( \gamma\in \D_n \), define a function
\( h:[n]\rightarrow \{0,1,2,\dots, n-1\} \) by
\( h(j)=|\{ i: m_i <j\}| \). The \emph{Hessenberg variety} associated
with \( X\in \mathfrak{gl}_n (\CC) \) and \( h \) is
\[
\Hess(X,h) := \{V_{\bullet}\in \Fl(\CC^n) \,:\, X(V_i)\subseteq V_{h(i)} \, \textup{ for all }\, i\in [n] \}.
\]
Its Poincar\'e polynomial is given by
\[
  \Poin (\Hess(X,h);q)=\sum_{d\ge 0}\dim H^{2d}(\Hess (X,h))q^d.
\]

For a partition \( \mu=(\mu_1, \dots, \mu_{\ell})\vdash n \),
let \(N_\mu \) be a nilpotent matrix whose Jordan block sizes are the parts of \( \mu \).
If
\[
  X_\gamma (\bm x;q)=\sum_{\mu\vdash n}c_{\gamma,\mu}(q)P_\mu (\bm x; q),
\]
then Precup--Sommers \cite[Equation (4.7)]{PS} proved that
\[
 \Poin (Y_{\gamma,\mu};q)=q^{n(\mu)} c_{\gamma,\mu}(q^{-1}),
\]
where \( Y_{\gamma,\mu} =\Hess (N_\mu , h) \); see \cite[Theorem~4.4]{Gag} for a more explicit formulation using symmetric functions.
Combining this identity with the tableau formula of Ji--Precup
\cite[Theorem~5.7]{JP} gives
\begin{equation}\label{eq:32}
  q^{n(\mu)}c_{\gamma,\mu}(q^{-1}) = \sum_{T\in \RS_h (\mu)}q^{\inv_h (T) }, 
\end{equation}
where \( \RS_h (\mu) \) is the set of \( h \)-strict tableaux of composition shape \( \mu \) and \( \inv_h (T) \)
is the number of inversion pairs in \( T \) satisfying certain conditions.
See \cite{JP} for a detailed description.

Comparing \eqref{eq:32} with our result \eqref{eq:31} yields 
\[
  q^{\area(\gamma)-n(\mu)}r_{\gamma,\mu}(q)  \prod_{i\ge 1}[m_i (\mu)]_q ! 
  = q^{n(\mu)}\sum_{T\in \RS_h (\mu)}q^{-\inv_h (T) }.
\]
The right-hand side is a generating function for \( h \)-strict
tableaux, whereas the left-hand side is a generating function for
linked rook placements with an extra factor
\( \prod_i [m_i(\mu)]_q! \).

There is another expansion related to the Poincar\'e polynomials of Hessenberg varieties
due to Abreu, Nigro, and Ram \cite{ANR}.
In \cite[Corollary~3.8]{ANR}, they proved that
\[
  B_{\gamma,\mu}(q):=\Poin (\Hess(N_\mu, \vm);q)
  =\left\langle  \tilde{H}_\mu (\bm x;q),\omega X_\gamma (\bm x;q) \right\rangle,
\]
where \( \langle\cdot,\cdot\rangle \) is the Hall inner product.
Since
\( \langle \tilde{H}_\lambda (\bm x;q),P_\mu (\bm x;q^{-1})\rangle
=q^{n(\mu)}\delta_{\lambda\mu} \), this is equivalent to
\begin{equation}\label{eq:ANR}
 \omega X_\gamma (\bm x;q )=\sum_{\mu\vdash n}q^{-n(\mu)}B_{\gamma,\mu}(q)P_\mu (\bm x;q^{-1}).
\end{equation}
In \cite[Theorems~1.1 and~7.1]{Tym}, Tymoczko gave a tableau rule to
compute \( B_{\gamma,\mu}(q) \). We note that \eqref{eq:ANR} yields an
expansion of \( X_\gamma (\bm x;q) \) in terms of
\( \omega P_\mu (\bm x;q^{-1}) \), not \( P_\mu (\bm x;q) \). Thus an
additional change of basis is required to obtain an expansion in terms
of \( P_\mu (\bm x;q) \).

Finally, we compare \Cref{thm:main} with the Hall--Littlewood expansion obtained as a specialization
of the Macdonald expansion of \( X_\gamma (\bm x;q) \).
In \cite{Griffin2025}, Griffin et al. found an explicit formula for the coefficients
appearing in the expansion of \( X_\gamma (\bm x;q) \) in terms of the modified Macdonald
polynomials \( \tilde{H}_\mu [(q-1)X;q,t] \). As a specialization, they obtained
an expansion of \( X_\gamma (\bm x;q) \) into the Hall--Littlewood polynomials \( Q_\mu (\bm x;q^{-1}) \),
where \( Q_\mu (\bm x;q)=\prod_{i\ge 1}(q;q)_{m_i (\mu)}P_\mu (\bm x;q) \).
More precisely, if
\begin{equation}\label{eq:GMR+}
  X_\gamma (\bm x;q)=\sum_{\mu \vdash n}\frac{q^{n+\area (\gamma)}}{(q-1)^{\ell(\mu)}} d_{\gamma, \mu}(q) Q_\mu (\bm x;q^{-1}),
\end{equation}
Griffin et al. \cite[Theorem~4.1]{Griffin2025} gave a combinatorial description of
\( d_{\gamma,\mu}(q) \) as a sum over certain standard Young tableaux of shape \( \mu' \).
By comparing the coefficients in \eqref{eq:GMR+} and \eqref{eq:31}, and using
the reciprocity relation \( X_\gamma (\bm x;q)=q^{\area (\gamma)} X_\gamma (\bm x; q^{-1}) \), we get
\[
  q^{\area (\gamma)-n(\mu)}r_{\gamma,\mu}(q)=q^{\ell (\mu)-n}d_{\gamma,\mu}(q^{-1}),
\]
where the left-hand side is computed using linked rook placements whereas 
the right-hand side is computed using certain standard Young tableaux of shape \( \mu' \).

\section*{Acknowledgements}

The authors would like to thank the anonymous referees for their
thorough review of the paper and for providing numerous insightful and
constructive comments. The second author is grateful to Jeong Hyun
Sung for helpful conversations. J.~S.~Kim was supported by the
National Research Foundation of Korea (NRF) grant funded by the Korean
government RS-2025-00557835. S.~J.~Lee was supported by the National
Research Foundation of Korea (NRF) grant funded by the Korean
government (MSIT) (No.0450-20240021). M.~Yoo was supported by the
National Research Foundation of Korea (NRF) grant funded by the Korean
government RS-2024-00344076.

\bibliographystyle{alpha}  
\newcommand{\etalchar}[1]{$^{#1}$}

\end{document}